\documentclass[aop]{imsart}

\RequirePackage{amsthm,amsmath,amsfonts,amssymb}
\RequirePackage[numbers]{natbib} 
\RequirePackage[colorlinks,citecolor=blue,urlcolor=blue]{hyperref}

\startlocaldefs
\usepackage{graphicx}
\usepackage[capitalise, noabbrev]{cleveref}
\usepackage{bbm}
\usepackage{enumerate}
\numberwithin{equation}{section}
\theoremstyle{plain}
\newtheorem{theorem}{Theorem}[section]
\newtheorem*{theorem2}{Theorem}
\newtheorem{lemma}[theorem]{Lemma}

\newtheorem{proposition}[theorem]{Proposition}
\theoremstyle{remark}

\newtheorem{remark}[theorem]{Remark}

\newtheoremstyle{conditionstyle}
{\smallskipamount}
{\smallskipamount}
{\normalfont}
{0 pt}
{\normalfont}
{ }
{ }
{(C\thmnumber{#2})}

\theoremstyle{conditionstyle}
\newtheorem{condition}{}
\creflabelformat{condition}{(#2C#1#3)}
\crefname{condition}{}{}
\Crefname{condition}{}{}

\DeclareMathAlphabet{\mathpzc}{OT1}{pzc}{m}{it}
\DeclareMathOperator*{\argmin}{argmin}

\def\too{\longrightarrow}

 \renewcommand{\epsilon}{\varepsilon}
 \let\hat\widehat
 \renewcommand{\theta}{\vartheta}
 \let\tilde\widetilde

\newcommand{\E}{\mathbb{E}}
\newcommand{\N}{\mathbb{N}}
\renewcommand{\P}{\mathbb{P}}

\newcommand{\R}{\mathbb{R}}
\newcommand{\Z}{\mathbb{Z}}
\newcommand{\1}{\mathbbm{1}}

\newcommand{\arsinh}{\ensuremath{\mathop{\mathrm{arsinh}}}}

\newcommand{\cB}{\mathcal{B}}

\newcommand{\cF}{\mathcal{F}}

\newcommand{\cN}{\mathcal{N}}
\newcommand{\cO}{\mathcal{O}}
\newcommand{\co}{\begin{smallmatrix}\!{\mathcal {O}}\!\end{smallmatrix}}
\newcommand{\cU}{\mathcal{U}}

\newcommand{\ddT}[1]{{\frac{\mathrm{d}^{#1}}{\mathrm{d}t^{#1}}}}
\newcommand{\ddt}{\ddT{}}
\newcommand{\ddtt}{\ddT{2}}
\newcommand{\ddttt}{\ddT{3}}
\newcommand{\ddta}{\ddT{a}}
\newcommand{\ddtk}{\ddT{k}}
\newcommand{\ddtl}{\ddT{l}}

\newcommand{\ddtr}{\ddT{r}}
\newcommand{\ddU}[1]{{\frac{\mathrm{d}^{#1}}{\mathrm{d}u^{#1}}}}

\newcommand{\ddur}{\ddU{r}}
\newcommand{\ddX}[1]{{\frac{\mathrm{d}^{#1}}{\mathrm{d}x^{#1}}}}
\newcommand{\ddx}{\ddX{}}
\newcommand{\dd}{\,\mathrm{d}}

\newcommand{\dmu}{\,\mathrm{d}\mu}

\newcommand{\ds}{\,\mathrm{d}s}
\newcommand{\dt}{\,\mathrm{d}t}

\newcommand{\du}{\,\mathrm{d}u}
\newcommand{\dv}{\,\mathrm{d}v}
\newcommand{\dw}{\,\mathrm{d}w}
\newcommand{\dx}{\,\mathrm{d}x}
\newcommand{\dy}{\,\mathrm{d}y}
\newcommand{\dz}{\,\mathrm{d}z}
\newcommand{\bP}{\mathbf{P}}

\newcommand{\Cov}{\ensuremath{\mathop{\mathrm{Cov}}}}

\newcommand\lbullet{	\raisebox{-0.5ex}{\scalebox{2}{$\cdot$}}}

\newcommand{\tE}{\tilde{\mathbb{E}}}
\newcommand{\tF}{\tilde{F}}
\newcommand{\tf}{\tilde{f}}

\newcommand{\tk}{\tilde{\kappa}}
\newcommand{\tlambda}{\tilde{\lambda}}
\newcommand{\tL}{\tilde{L}}
\newcommand{\tmu}{\tilde{\mu}}

\newcommand{\tP}{\tilde{P}}
\newcommand{\tiP}{\tilde{\P}}
\newcommand{\tp}{\tilde{p}}
\newcommand{\tphi}{\tilde{\varphi}}

\newcommand{\tU}{\tilde{U}}
\newcommand{\TV}{\textup{TV}}

\newcommand{\Var}{\ensuremath{\mathop{\mathrm{Var}}}}

\usepackage{xr-hyper}
\usepackage{hyperref}
\endlocaldefs

\begin{document}

\begin{frontmatter}
\title{
	Non-uniform bounds and Edgeworth expansions in self-normalized limit theorems
}
\runtitle{
	Edgeworth expansions in self-normalized limit theorems
}

\begin{aug}
\author[A]{\fnms{Pascal}~\snm{Beckedorf}\ead[label=e1]{pascal.beckedorf@stochastik.uni-freiburg.de}}
\and
\author[A]{\fnms{Angelika}~\snm{Rohde}\ead[label=e2]{angelika.rohde@stochastik.uni-freiburg.de}}

\address[A]{Albert-Ludwigs-Universit\"at Freiburg \\ \printead[presep={ }]{e1,e2}}
\end{aug}

\begin{abstract}
	We study Edgeworth expansions in limit theorems for self-normalized sums. Non-uniform bounds for expansions in the central limit theorem are established while only imposing minimal moment conditions. Within this result, we address the case of non-integer moments leading to a reduced remainder. Furthermore, we provide non-uniform bounds for expansions in local limit theorems. The enhanced tail-accuracy of our non-uniform bounds allows for deriving an Edgeworth-type expansion in the entropic central limit theorem as well as a central limit theorem in total variation distance for self-normalized sums.
\end{abstract}

\begin{keyword}[class=MSC]
\kwd[Primary ]{60F05}
\kwd[; secondary ]{62E20}
\end{keyword}

\begin{keyword}
	\kwd{Edgeworth expansion}
	\kwd{non-uniform bounds}
	\kwd{central limit theorem}
	\kwd{local limit theorem}
	\kwd{entropy}
	\kwd{total variation distance}
	\kwd{self-normalized sums}
	\kwd{rate of convergence}
\end{keyword}

\end{frontmatter}
\tableofcontents

\section{Introduction and main results}\label{ch.intro}


Let $X_1,\dots,X_n$ be independent, identically distributed random variables with mean $\E X_1=0$ and variance $\Var X_1=1$. Here and throughout this article, 
\[
S_n=\sum_{j=1}^{n} X_j,
\qquad
V_n=\sqrt{\sum_{j=1}^{n} X_j^2},
\qquad
T_n=\begin{cases}
S_n/V_n& \text{ if }V_n>0\\
0 & \text{ otherwise.}
\end{cases}
\]
The self-normalized sum $T_n$ was shown to be asymptotically normal if and only if $X_1$ belongs to the domain of attraction of the normal law in \cite{GGM97}. This result has been generalized in \cite{CG04}.

The distance of the distribution of $T_n$ to a normal distribution has been investigated in various forms such as Berry--Esseen bounds \cite{BG96,Fri89,Hal88,vZw84} and non-uniform Berry--Esseen bounds \cite{CG03,JSW03,RW05, WJ99}, to mention just a few. For an extensive survey on further self-normalized limit theorems we refer to \cite{SW13Survey}.

Concerning asymptotic expansions of the distribution of self-normalized sums, Chung \cite{Chu46} was the first to prove an Edgeworth expansion for the Student $t$-statistic. Bhattacharya and Ghosh \cite{BG78} provided an approach of Edgeworth expansion including self-normalized sums, relaxing Chung's moment conditions. The moment and smoothness assumptions of \cite{BG78} were further weakened in \cite{BR91,BG88}. \cite{Hal87} attained an expansion for the $t$-statistic under minimal moment conditions. One-term Edgeworth expansions were treated in \cite{BGvZ97,BP99,BP03,HW04,JW10,PvZ98}. Recently, \cite{BG22} used the first two expansion terms to derive a sharper bound in a Berry--Esseen type result. Although not explicitly stated for self-normalized sums, Edgeworth expansions for densities of more general statistics are given in \cite[Section 2.8]{Hal92edgeworth}.

For normalized sums $Z_n=S_n/\sqrt{n}$, Edgeworth expansions with non-integer moment conditions and non-uniform bounds have been investigated for a long time in the CLT and LLT (see e.g. \cite{BR76normal,Pet75sums}). In case of non-integer moments, a CLT for $Z_n$ has been derived in \cite{BCG11} using fractional calculus. These stronger results lay the foundation for limit theorems in more powerful metrics such as expansions in the entropic CLT \cite{BCG13} or related results for other information-theoretic distances \cite{BCG14-2,BCG19,BCK15,BM19}.

For \textit{self-normalized sums} however, non-uniform bounds for Edgeworth expansions have not been studied so far. The main aim of this article is to derive such non-uniform bounds in the CLT and in the LLT for an Edgeworth expansion. These provide substantially stronger accuracy in the tails, enabling us to prove an entropic CLT and a CLT in total variation for self-normalized sums. As compared to normalized sums, the crucial obstacle here is the missing product form of the statistic's characteristic function. Besides, we also treat the case of non-integer moments in the CLT. Subsequently, we give an overview on our results.

Let $\Phi$ and $\phi$ denote the standard normal distribution function and density function, respectively. Write $\mu_k=\E X_1^k$ and $F_n(x)=\P(T_n\le x)$ for $k,n\in\N$ and $x\in\R$. Let $\lfloor x \rfloor$ denote the integer part of $x$. With $\Phi^Q_{m,n}$ we denote the Edgeworth expansion
\begin{equation}\label{eq.Phi^Q-def}
\Phi^Q_{m,n}(x)=\Phi(x)+\sum_{r=1}^{m-2}Q_{r}(x)n^{-r/2}, \quad x\in\R,
\end{equation}
with $m-2$ expansion terms of the distribution of $T_n$. Throughout this article, $X_1$ is assumed to be symmetric such that all uneven moments vanish and thus $Q_{r}=0$ for uneven $r\in\N$. In contrast to the approximation functions of $Z_n$, the polynomials $Q_{r}$ generally do not have a closed form. However, for even $r$, $Q_{r}$ provide a useful form which consist of $\phi$ multiplied with an uneven polynomial of degree $2r-1$ in $x$. The coefficients of $Q_{r}$ are functions of the moments $\mu_3,\dots,\mu_{r+2}$. The first approximating functions for the distribution function have the form
\begin{equation}\label{eq.Q-2+4}
\begin{split}
Q_{2}(x)&=-\phi(x) H_{3}(x) \big(- \tfrac{1}{12}\big) \mu_4,\\
Q_{4}(x)&=-\phi(x) \Big( H_{7}(x) \tfrac{1}{288} \mu_4^2
+ H_{5}(x) \tfrac{1}{45} \mu_6
+ H_{3}(x) \tfrac{1}{12} \big(2\mu_6 +\mu_4 - 3 \mu_4^2\big) \Big),
\end{split}
\end{equation}
where $H_k$ is the $k$-th Hermite polynomial.

Unless stated otherwise, all orders of convergence or divergence in $n$ are understood for $n\to\infty$. Our first result is

\begin{theorem}\label{t.clt}
	Assume that $X_1$ is symmetric, the distribution of $X_1$ is non-singular and $\E|X_1|^{s}<\infty$ for some ${s}\ge2$. Then for $m=\lfloor {s} \rfloor$,
	\begin{equation}\label{eq.t-clt}
	\sup_{x\in\R} \, (1+|x|)^{m}|F_n(x)-\Phi^{Q}_{m,n}(x)|= \co\big( n^{-({s}-2)/2} \, (\log n)^{({s}+2m)/2} \big).
	\end{equation}
\end{theorem}

Note that only the moments used in the expansion have to be finite, thus we only pose minimal moment conditions. The distribution of $X_1$ being non-singular (with respect to the Lebesgue measure) is only slightly stronger than Cram\'er's condition
\begin{equation}\label{eq.cramer}
\limsup_{|t|\to\infty}|\E\exp(itX_1)|<1
\end{equation}
which is usually imposed when working with $Z_n$ (see e.g. \cite{BR76normal,Pet75sums}).
Due to the denominator $V_n$, the characteristic function of $T_n$ lacks the typical product form, required for an Edgeworth expansion. To overcome this problem, we first condition on $\cF_n=\sigma(|X_1|,\dots,|X_n|)$ as in \cite{Hal87}. On the one hand, the independence of the factors then provides a desired structure as a product. On the other hand, the random variables $X_1,\dots,X_n$ conditional on $\cF_n$ are discrete (non-lattice) random variables that are not identically distributed any longer. Therefore, the necessary Cram\'er condition \labelcref{eq.cramer} is not fulfilled and the standard theorems for Edgeworth expansions (such as from \cite{BR76normal,Pet75sums}) do not apply. These properties necessitate adjustments in the derivation of an Edgeworth expansion, but lead to rather involved integral remainder terms which have to be bounded thoroughly. At this point, the non-singularity is getting essential.

If $T_n$ has a density, then by taking derivatives in \labelcref{eq.Phi^Q-def}, the Edgeworth expansion of the density of $T_n$ has the form
\begin{equation}\label{eq.phi^q-def}
\phi^q_{m,n}(x)=\phi(x)+\sum_{r=1}^{m-2}q_{r}(x)n^{-r/2}
\end{equation}
for all $x\in\R$. By differentiating \labelcref{eq.Q-2+4}, the first approximating functions are
\begin{equation}\label{eq.q-2}
\begin{split}
q_{2}(x)&=\phi(x) H_{4}(x) \big(- \tfrac{1}{12}\big) \mu_4,\\
q_{4}(x)&=\phi(x) \Big( H_{8}(x) \tfrac{1}{288} \mu_4^2
+ H_{6}(x) \tfrac{1}{45} \mu_6
+ H_{4}(x) \tfrac{1}{12} \big(2\mu_6 +\mu_4 - 3 \mu_4^2\big) \Big).
\end{split}
\end{equation}
In the case that $r$ is even, $q_{r}$ provides a useful form which constitutes of $\phi$ multiplied with an even polynomial of degree $2r$ in $x$. The coefficients of $q_{r}$ are functions of the moments $\mu_3,\dots,\mu_{r+2}$. As above, $q_{r}=0$ for uneven $r\in\N$.

Our second result is an application of our \cref{t.llt} which has weaker but more technical conditions. Further LLTs with more formal but less stringent conditions are provided in \cref{ch.density}. For the classical normalized statistic $Z_n$, such a  result has been derived in \cite{BCG11}.

\begin{theorem}\label{t.llt-red}
	Assume that $X_1$ is symmetric, has a bounded density and $\E|X_1|^{2m}<\infty$ for some $m\in\N$, $m\ge3$. Then there exists $N\in\N$ such that for all $n\ge N$, $T_{n}$ have densities $f_{n}$ that satisfy
	\begin{equation*}
	\sup_{x\in\R} \, (1+|x|)^{m}|f_n(x)-\phi^{q}_{m,n}(x)|= \co\big(n^{-(m-3)/2}\big).
	\end{equation*}
	Moreover, if $\E|X_1|^{2m+2}<\infty$ for some $m\in\N, m\ge2$, the order of convergence reduces to $\co\big(n^{-(m-1)/2} \, \log n \big)$.
\end{theorem}

The route of the proof of \cref{t.llt-red} is related to the proof of \cref{t.clt}. However, conditional on $\cF_n$, the self-normalized sum $T_n$ does not have a density as it is discrete. In particular, the Fourier inversion formula is not available any longer. The idea is to tackle this severe obstacle  by perturbing $T_n$ with an independent normal random variable with suitably chosen variance. Besides enjoying the properties that tails of density and characteristic function decay in the same magnitude, a Gaussian perturbation does not hinder convergence to the normal distribution. This perturbation trick allows to work with densities and deriving an Edgeworth expansion but comes at the cost of a deconvolution after returning to the unconditional setting. For performing a tight approximation in the final deconvolution step, verifiable conditions have to be imposed that guarantee sufficient smoothness of the density of $T_n$.

Based on \cref{t.llt-red}, we prove an entropic CLT and its rates of convergence. For a random variable $X$ with density $p$, mean $\mu$ and variance $\sigma^2$ we define its entropy (with $0\log 0:=0$)
\begin{align*}
h(X)=-\int_{-\infty}^{\infty}p(x)\log p(x)\dx,
\end{align*}
and its relative entropy
\begin{align}\label{eq.def-rel-entropy}
D(X)=D(X\|Z)=h(Z)-h(X)=\int_{-\infty}^{\infty}p(x)\log \frac{p(x)}{\phi_{\mu,\sigma^2}(x)}\dx,
\end{align}
where $Z$ is $\cN(\mu,\sigma^2)$-distributed with density $\phi_{\mu,\sigma^2}$. As the normal distribution maximizes the entropy for given mean and variance, we know $D(X)\ge0$.

Barron \cite{Bar86} was the first to prove an entropic CLT for the classical statistic $Z_n$, i.e. $D(Z_n)\to 0$. Many more results on entropic CLTs can be found in \cite{Johnson04}. Moreover, in \cite{BCG13}, an Edgeworth expansion was used to determine the rates of convergence in the entropic CLT for the classical statistic.
We prove such type of result for the self-normalized sum.

\begin{theorem}\label{t.entropy}
	Assume that $X_1$ is symmetric, has a bounded density and $\E|X_1|^{2m}<\infty$ for some $m\in\N$, $m\ge3$. Then
	\begin{equation*}
	D(T_n)=\frac{c_2}{n^2}+\dots+\frac{c_{\lfloor (m-2)/2 \rfloor}}{n^{\lfloor (m-2)/2 \rfloor}}+ \co\big((n\log n)^{-(m-2)/2}\,\log n\big),
	\end{equation*}
	where
	\begin{align*}
	c_l=\sum_{k=2}^{2l} \frac{(-1)^k}{k(k-1)} \sum \int_{\R} \frac{q_{r_1}(x)\dots q_{r_k}(x)}{\phi(x)^{k-1}} \dx, \quad l\in\N,
	\end{align*}
	with the inner sum running over all positive integers $r_1,\dots,r_k$ such that $r_1+\dots+r_k=2l$.
\end{theorem}

\cref{t.entropy} also holds under the weaker but more technical conditions of \cref{t.llt}, see \cref{t.entropy-general}.

\begin{remark}Comparing the relative entropy of $T_n$ and $Z_n$ under the conditions of \cref{t.entropy} with $m=6$, we obtain
\begin{align*}
D(T_n)
&= n^{-2} \tfrac{1}{12} \mu_4^2 + \co\big((n\log n)^{-2}\,\log n\big)
\end{align*}
and (from \cite{BCG13})
\begin{align*}
D(Z_n)
= n^{-2} \tfrac{1}{48}(\mu_4-3)^2 + \co\big((n\log n)^{-2}\big).
\end{align*}
Note that under these conditions, $\mu_4^2/12 > (\mu_4-3)^2 /48$. The relative entropies are always of the same magnitude and particularly close if $\mu_4$ is close to 1. The limit case $\mu_4=1$ is excluded because $X_1$ has a density.
\end{remark}

Our final goal is prove a bound on the total variation. On the linear space of finite signed measures on $(\R, \mathcal{B}(\R))$, the total variation norm is defined as 
$$\Arrowvert\mu\Arrowvert_{TV}=\sup_{B\in\mathcal{B}(\R)}\Big(\arrowvert \mu(B)\arrowvert +\arrowvert \mu(B^c)\arrowvert\Big)$$  To circumvent unnecessary notational complexity, we identify finite (signed) measures with their corresponding cumulative distribution functions.
If $F_P$ and $F_Q$ have densities, bounding their distance in total variation is equivalent to bounding the $L^1$ norm of the difference of the corresponding densities
\begin{equation}\label{eq.TV-L1}
\|F_n-\Phi^{Q}_{m,n}\|_{\TV}=\|f_n-\phi^{q}_{m,n}\|_{L^1}.
\end{equation}

Assume that the conditions of \cref{t.llt-red} are satisfied, in particular $\E|X_1|^{2m}<\infty$ for some $m\in\N$, $m\ge3$. Then the second statement of \cref{t.llt-red} applied with $m-1$ implies
\begin{align*}
\|f_n-\phi^{q}_{m-1,n}\|_{L^p}^p
=\int_\R\big|f_n(x)-\phi^{q}_{m-1,n}(x)\big|^p\dx
\le\int_{\R}\big|(1+|x|)^{-(m-1)} r_n\big|^p\dx
\le c \, r_n^p
\end{align*}
for $p\in[1,\infty)$ and an absolute constant $c>0$. Here, $r_n=\co\big(n^{-(m-2)/2} \, \log n \big)$ and thus
\begin{equation}\label{eq.fn-phi-Lp}
\|f_n-\phi^{q}_{m-1,n}\|_{L^p}
= \co\big(n^{-(m-2)/2} \log n\big).
\end{equation}

The following theorem shows that we can achieve a stronger order of convergence by methods that are particularly suitable for optimizing the $\TV$ distance.

\begin{theorem}\label{t.TV}
	Assume that $X_1$ is symmetric, has a bounded density and $\E|X_1|^{2m}<\infty$ for some $m\in\N$, $m\ge3$. Then
	\begin{equation*}
	\|F_n-\Phi^{Q}_{m,n}\|_{\TV} = \co\big(n^{-(m-2)/2}\big).
	\end{equation*}
\end{theorem}

\cref{t.TV} also holds under the weaker but more technical conditions of \cref{t.llt}, see \cref{t.dichte-L1}.

The remainder of this article is organized as follows. In \cref{ch.pre}, we introduce further notation and the formal Edgeworth expansion for self-normalized sums. The CLT-type result (\cref{t.clt}) is proven in \cref{ch.distr}. \cref{ch.density} provides more general LLTs that lead to \cref{t.llt-red}. In \cref{ch.entropy,ch.TV}, the prior results are used to attain the entropic CLT (\cref{t.entropy}) and the CLT in total variation distance (\cref{t.TV}), respectively.

\section{Notation and preliminaries}\label{ch.pre}
The notation $\varphi$ is used for the characteristic functions of random variables or the Fourier transforms of functions. Their dependence is indicated by a subscript, e.g. $\varphi_{T_n}$ or $\varphi_{\Phi^{\tP}_{m,n}}$, while $\varphi_{X_j}$ is abbreviated as $\varphi_j$. For $k\ge0$, $c(k)$ represents different positive constants in different (or even in the same) formulas, only depending on $k$ and the distribution of $X_1$. $\lceil x \rceil$ denotes the smallest integer larger than or equal to $x$.

If a sum $\sum$, product $\prod$ or a maximum $\max$ does not have an index, it is always meant to run over $j=1,\dots,n$. The sum $\sum_{*(k_\cdot,r,u)}$ is carried out over all non-negative integer solutions $(k_1,k_2,\dots,k_r)$ of the equalities \mbox{$k_1+2k_2+\dots+rk_r=r$} and $k_1+k_2+\dots + k_r=u$. In many occurrences, only the variable $r$ is fixed which will be denoted by $\sum_{*(k_\cdot,r,\cdot)}$. Then the $k_i$ only have to satisfy the equation $k_1+2k_2+\dots+rk_r=r$ and $u=u(k_\cdot)$ is defined by $u(k_\cdot)=k_1+k_2+\dots +k_r$. In some situations only the variable $u$ is fixed. In this case we need information on the maximal index of $k_\cdot$. This is denoted by $\sum_{*(k_\cdot,\cdot, u, t)}$ meaning the summation over all non-negative integer solutions $(k_1,k_2,\dots,k_t)$ of the equation $k_1+k_2+\dots +k_t=u$ and $r=r(k_\cdot)$ is defined by $r(k_\cdot)=k_1+2k_2+\dots+tk_t$. Clearly $u\le r$ holds.

We set
\[
M_n=\max |X_j|
\qquad \text{and} \qquad
B_n=V_n/M_n.
\]
As $V_n=0$ and $M_n=0$ are equivalent, we define $B_n=1$ in this case, whence  $1 \le B_n \le \sqrt{n}$ is always satisfied. For $l\ge0$ the $l$-th cumulant of $X_1$ will be denoted by $\kappa_{l}$ and is defined by
\begin{equation}\label{eq.cumulant-abl}
\kappa_{l} = i^{-l}\ddtl \log\varphi_1(t)\Big|_{t=0} \, .
\end{equation}

The conditional expected value with respect to $\mathcal{F}_n$ is shortened as $\tE[\ \cdot\ ]=\E[\ \cdot\ |\cF_n]$. Correspondingly, the $\sim$-notation is extended to $\P, \mu, \kappa, F, f, p, P, U, \varphi$ for describing the variants conditional on $\mathcal{F}_n$. We define
\[
\tL_{k,n}
= V_n^{-k}\sum|X_j|^{k}
\]
for $k\ge2$ integer. Clearly $V_n=0$ and $\sum|X_j|^k=0$ are equivalent. In this case, we define $\tL_{k,n}=0$. Note that for all $k\ge2$,
\begin{equation}\label{eq.Lle1}
\tL_{k,n}
= V_n^{-k}\sum|X_j|^{k}
\le V_n^{-2}\,\frac{M_n^{k-2}}{V_n^{k-2}}\,\sum|X_j|^{2}
\le V_n^{-2}\,\sum|X_j|^{2}
=1.
\end{equation}

When we state that the sequence $(A_n)$ increases (or decreases, respectively) in polynomial order, we mean that there exist $a,b>0$ (or $a,b<0$, respectively) such that
\begin{equation*}
\limsup_{n\to\infty}\Big|\frac{n^a}{A_n}\Big|<\infty 
\quad\text{and}\quad 
\limsup_{n\to\infty}\Big|\frac{A_n}{n^b}\Big|<\infty.
\end{equation*}

When writing (probability) densities, we exclusively mean Lebesgue densities. Densities with respect to the counting measure will be denoted as probability mass functions. When mentioning a density (or a function in general), we always refer to the continuous version of it in case it exists.





\subsection{Edgeworth expansions for self-normalized sums}\label{s.pre.self-normalized}

In this subsection, we introduce the Edgeworth expansion for $T_n$ conditional on $\cF_n$ and establish a connection to \labelcref{eq.Phi^Q-def,eq.phi^q-def}. In particular, the explicit form of the approximating polynomials is derived for later purposes. To this end, we assume that $X_1$ is symmetric such that $\tiP(X_j= \pm |X_j|)=1/2$, $j=1,\dots,n$. For $k\in\N$, we denote by $\tmu_{k,j}=\tE[X_j^k]$ the $k$-th conditional moment of $X_j$. Accordingly, we extend the notation to $\tk_{k,j}$.
Note that
\begin{equation}\label{eq.cumulants-bed-ug}
\tk_{2k-1,j}=\tmu_{2k-1,j}=0 \qquad \text{and} \qquad \tmu_{2k,j}=|X_j|^{2k}.
\end{equation}
Although not necessary, we state the absolute value for even powers of $X_j$ to point out that we operate in the conditional setting where the absolute values are deterministic. By explicitly computing the values of the Bell polynomials, the $k$-th cumulant can be expressed in terms of moments up to degree $k$
which yields
\begin{equation}\label{eq.tk-2r}
\begin{split}
\tk_{2r,j}
=|X_j|^{2r} \sum_{*(k_\cdot,r,\cdot)}(-1)^{u(k_\cdot)-1} (u(k_\cdot)-1)! (2r)! \prod_{l=1}^{r} \frac{1}{k_l!(2l)!^{k_l}} 
\end{split}
\end{equation}
for all $r\in\N$ because all summands with one or more $k_{2l-1}>0$ vanish due to \labelcref{eq.cumulants-bed-ug}. In particular,
\begin{equation*}
\tk_{2,j}=|X_j|^{2},
\qquad
\tk_{4,j}
=-2 |X_j|^{4},
\qquad
\tk_{6,j}
=16 |X_j|^{6}.
\end{equation*}


Below, we outline the characteristics of the Edgeworth expansion of $T_n$, conditioned on $\cF_n$. Clearly, the classical Edgeworth expansion (see e.g. \cite[Section VI.1]{Pet75sums}) needs to be adjusted in terms of approximating polynomials and functions. In particular, for $l,r,m\in\N$ and $t\in\R$, we define
\begin{align}
\tlambda_{l,n}\label{eq.tlambda}
&=\Big(\sum\tmu_{2,j}\Big)^{-l/2}\sum\tk_{l,j}
=V_n^{-l}\sum\tk_{l,j}
\quad (\tlambda_{l,n}=0 \text{ if } V_n=0),\\
\tU_{r,n}(it)\label{eq.tU}
&=\sum_{*(k_\cdot,r,\cdot)} (it)^{r+2u(k_\cdot)} \prod_{l=1}^{r} \frac{1}{k_l!} \Big(\frac{\tlambda_{l+2,n}}{(l+2)!}\Big)^{k_l},\\
\tphi_{\Phi^{\tP}_{m,n}}(t)\label{eq.tphi}
&=e^{-t^2/2}\Big(1+\sum_{r=1}^{m-2} \tU_{r,n}(it)\Big).
\end{align}

Next, the Edgeworth expansion for the conditional distribution function of $T_n$ is
\begin{equation}\label{eq.Phi^tP}
\Phi^{\tP}_{m,n}(x)
=\Phi(x)+\sum_{r=1}^{m-2}\tP_{r,n}(x),
\end{equation}
where
\begin{equation}\label{eq.tP}
\tP_{r,n}(x)
=-\phi(x)\sum_{*(k_\cdot,r,\cdot)} H_{r+2u(k_\cdot)-1}(x) \prod_{l=1}^{r} \frac{1}{k_l!} \Big(\frac{\tlambda_{l+2,n}}{(l+2)!}\Big)^{k_l}
\end{equation}
for all $x\in\R$.
In \labelcref{eq.tU,eq.tP}, all summands with one or more $k_{2l-1}>0$ vanish due to \labelcref{eq.cumulants-bed-ug}. For uneven indices the polynomials only consist of such summands and therefore $\tU_{2r-1,n}=\tP_{2r-1,n}=0$ for all $r\in\N$. The even polynomials $\tP_{2r,n}$ reduce to 
\begin{align}\label{eq.tP-2+4}
\tP_{2r,n}(x)&=-\phi(x)\sum_{*(k_\cdot,r,\cdot)} H_{2r+2u(k_\cdot)-1}(x) \prod_{l=1}^{r} \frac{1}{k_l!} \Big(\frac{\tlambda_{2l+2,n}}{(2l+2)!}\Big)^{k_l},\nonumber\\
\tP_{2,n}(x)&=-\phi(x) H_{3}(x) \frac{\tlambda_{4,n}}{4!},\\
\tP_{4,n}(x)&=-\phi(x) \Big( H_{7}(x) \frac{1}{2} \Big(\frac{\tlambda_{4,n}}{4!}\Big)^2
+ H_{5}(x) \frac{\tlambda_{6,n}}{6!}\Big).\nonumber
\end{align}

Taking expectations of $\Phi^{\tP}_{m,n}$ and $\tP_{r,n}$ respectively returns us to the unconditional setting. In \cref{r.V_n-lambda}, the expectation of the $\tlambda_{l,n}$ appearing in $\tP_{2,n}$ and $\tP_{4,n}$ is evaluated -- the only source of randomness in \labelcref{eq.tP-2+4}. For now, let all moments be finite for clear illustration. The evaluation results in
\begin{equation}\label{eq.E-tlambda_2k}
\begin{split}
\E\bigg[\frac{\tlambda_{4,n}}{4!}\bigg]
&= n^{-1} \big(- \tfrac{1}{12}\big) \mu_4
+ n^{-2} \tfrac{1}{12} \big(2\mu_6 +\mu_4 - 3 \mu_4^2\big)
+ \cO\big(n^{-3}\big),\\
\E\bigg[\frac{\tlambda_{6,n}}{6!}\bigg]
&= n^{-2} \tfrac{1}{45} \mu_6
+ n^{-3} \big(- \tfrac{1}{45}\big) (3\mu_8+3\mu_6-6\mu_4\mu_6)
+ \cO\big(n^{-4}\big),\\
\E\bigg[\frac{1}{2} \Big(\frac{\tlambda_{4,n}}{4!}\Big)^2\bigg]
&=n^{-2} \tfrac{1}{288} \mu_4^2
+ n^{-3} \big(- \tfrac{1}{288}\big) (8\mu_6\mu_4-7\mu_4^2- \mu_8)
+ \cO\big(n^{-4}\big).
\end{split}
\end{equation}
As the expected value of $\tlambda_{l,n}$ (and its powers) admit an expansion in powers of $n^{-1}$, it becomes evident that $\E\big[\Phi^{\tP}_{m,n}\big]$ also admits an expansion in powers of $n^{-1}$. However, we will keep using the common notation of an expansion in powers of $n^{-1/2}$. For $r\in\N$, the approximating polynomial $Q_{r}$ is thus defined as the coefficient of $n^{-r/2}$ in the expansion of
\begin{equation}\label{eq.Q-def}
\E\big[\Phi^{\tP}_{m,n}(x)\big] = \Phi(x)+\sum_{r=1}^{m-2}Q_{r}(x)n^{-r/2} + \cO\big(n^{-(m-1)/2}\big)
\end{equation}
for all $x\in\R$. As a result, we may define the approximating functions for the distribution function of $T_n$ by \labelcref{eq.Phi^Q-def}.

Combining \labelcref{eq.tP-2+4,eq.Q-def,eq.E-tlambda_2k}, it is apparent  that the first polynomials have the form \labelcref{eq.Q-2+4}. As $\tP_{2r-1,n}=0$ and the expected value of the relevant $\tlambda$ and its powers do not provide terms of orders $n^{-(2r-1)/2}$ for any $r\in\N$, also $Q_{2r-1}=0$ for all $r\in\N$.
The approximating polynomials $Q_r$ can be directly calculated by Chung's method \cite{Chu46,FD10}, by the $\delta$-method or by the smooth function model which are both explained in detail in \cite[Chapter 2]{Hal92edgeworth}.

The Edgeworth expansion can also be applied to the density of a normalized sum (see e.g. \cite[Chapter 4]{BR76normal}, \cite{BCG11}, \cite[Section 2.8]{Hal92edgeworth}, \cite[Section VII.3]{Pet75sums}). Concerning this expansion for the density, we state the suitable density-related functions for our setting. By differentiating \labelcref{eq.tP}, for $r\in\N$ we get the approximating polynomials for the conditional density function
\begin{equation}\label{eq.tp}
\tp_{r,n}(x)=\phi(x)\sum_{*(k_\cdot,r,\cdot)} H_{r+2u(k_\cdot)}(x) \prod_{l=1}^{r} \frac{1}{k_l!} \Big(\frac{\tlambda_{l+2,n}}{(l+2)!}\Big)^{k_l},
\end{equation}
(where clearly $\tp_{2r-1,n}=0$) and thus differentiation \labelcref{eq.Phi^tP} yields the Edgeworth expansion for the density of $T_n$
\begin{equation}\label{eq.phi^tp}
\phi^{\tp}_{m,n}(x)=\phi(x)+\sum_{r=1}^{m-2}\tp_{r,n}(x)
\end{equation}
for all $x\in\R$.
As above, for $r\in\N$ we define the approximating polynomials $q_r$ as the coefficient of $n^{-r/2}$ in the expansion of
\begin{equation*}
\E\big[\phi^{\tp}_{m,n}(x)\big] = \phi(x)+\sum_{r=1}^{m-2}q_{r}(x)n^{-r/2} + \cO\big(n^{-(m-1)/2}\big)
\end{equation*}
for all $x\in\R$ and the approximating functions for the density of $T_n$ by \labelcref{eq.phi^q-def}. In consequence, the first approximating polynomials $q_r$ have the form \labelcref{eq.q-2}. Note
\begin{equation*} 
\phi^{q}_{m,n}(x)=\ddx \Phi^{Q}_{m,n}(x) \quad \text{and} \quad q_{r}(x)=\ddx Q_{r}(x).
\end{equation*}

As we will not assume all moments to be finite, the remainder terms have to be specified. In view of the non-uniform bounds, their dependence on the argument $x$ is crucial.

\begin{proposition}\label{p.E-tP-Q}
	Assume that $X_1$ is symmetric and $\E |X_1|^{{s}}<\infty$ for some ${s}\ge2$. Then for $m=\lfloor {s} \rfloor$,
	\begin{equation}\label{eq.E-tP-Q}
	\sup_{x\in\R} \,\exp(x^2/4)\, \big|\E\big[\Phi^{\tP}_{m,n}(x)\big] - \Phi^Q_{m,n}(x)\big| = \co\big(n^{-({s}-2)/2}\big)
	\end{equation}
	and
	\begin{equation}\label{eq.E-tp-q}
	\sup_{x\in\R} \,\exp(x^2/4)\, \big|\E\big[\phi^{\tp}_{m,n}(x)\big] - \phi^q_{m,n}(x)\big| = \co\big(n^{-({s}-2)/2}\big).
	\end{equation}
\end{proposition}

Although this result could have been expected, its proof (see \cref{app.proofs.pre}) is quite elaborate. Since the expectation of $\tlambda_{l,n}$ cannot be explicitly evaluated, we study the expectation of its Taylor expansion around $n^{-l/2}\sum\tk_{l,j}$. Here, powers of $X_1$ that are higher than $s$ arise. As we need a bound in order of $n$ while controlling the powers of $X_1$,  tight bounds are developed on the critical summands. We show that by definition (see \labelcref{eq.Phi^tP} and \labelcref{eq.tP}),  $\E\big[\Phi^{\tP}_{m,n}(x)\big]$ possesses a similar expansion and similar bounds that result in \cref{p.E-tP-Q} when combined with \labelcref{eq.Q-def}.
Finally, we need some bounds on the approximation functions.

\begin{remark}
	All summands in \labelcref{eq.tU} with $k_1>0$ are 0 as $\tlambda_{3,n}=0$. Therefore, in all non-zero summands $k_1=0$ and thus $u(k_\cdot)\le r/2$. Additionally, note that the only $n$-dependence in the following expressions is in $\tlambda_{l,n}$, which satisfies $|\tlambda_{l,n}|\le c(l)$ for all $l=0,\dots,m$ due to \labelcref{eq.Lle1}. Then from \labelcref{eq.tphi},
	\begin{equation}\label{eq.ddtl-tphi}
	\Big|\ddtl\tphi_{\Phi^{\tp}_{m,n}}(t)\Big|
	=\Big|\ddtl \Big(e^{-t^2/2}\Big(1+\sum_{r=1}^{m-2} \tU_{r,n}(it)\Big)\Big)\Big|
	\le c(m) e^{-t^2/2}\big(1+|t|^{2m-4+l}\big)
	\end{equation}
	for $l=0,\dots,m$ and $t\in\R$. Equivalently, from \labelcref{eq.tP},
	\begin{equation}\label{eq.tP-bound}
	\big|\tP_{r,n}(x)\big|
	\le c(m) \phi(x) (1+|x|^{2r-1})
	\le c(m)
	\end{equation}
	and from \labelcref{eq.tp},
	\begin{equation}\label{eq.tp-bound}
	\big|\tp_{r,n}(x)\big|
	\le c(m) \phi(x) (1+|x|^{2r})
	\le c(m)
	\end{equation}
	for $r=0,\dots,m$ and $x\in\R$. Here, the constants $c(m)$ can be chosen uniformly for all $x\in\R$.
\end{remark}

\section{Non-uniform bounds for Edgeworth expansions in the central limit theorem for self-normalized sums}\label{ch.distr}

Recall $\Phi^Q_{m,n}(x)=\Phi(x)+\sum_{r=1}^{m-2}Q_{r}(x)n^{-r/2}$ from \labelcref{eq.Phi^Q-def}. For convenience, we restate \cref{t.clt} which is the goal of this section:

\begin{theorem2}
	Assume that $X_1$ is symmetric, the distribution of $X_1$ is non-singular and $\E|X_1|^{s}<\infty$ for some ${s}\ge2$. Then for $m=\lfloor {s} \rfloor$,
	\begin{equation*}
	\sup_{x\in\R} \, (1+|x|)^{m}|F_n(x)-\Phi^{Q}_{m,n}(x)|= \co\big( n^{-({s}-2)/2} \, (\log n)^{({s}+2m)/2} \big).
	\end{equation*}
\end{theorem2}

For $s\in[2,4)$, $\Phi^{Q}_{m,n}$ reduces to $\Phi$. Here, the non-uniform Berry--Esseen bounds in \cite{JSW03,RW05} allow for better bounds in view of the argument $x$ than \cref{t.clt}.
Without the factor $(1+|x|)^{m}$, this result has been derived in \cite{Hal87}, but only for integer ${s}$. For $Z_n$, non-uniform bound in the spirit of ours can be found in \cite[Theorem 2, Chapter VI]{Pet75sums}, for instance.

\paragraph*{Structure of the proof} First, we condition on $\cF_n$ and derive a result in the conditional setting (\cref{p.cond}) by showing a bound on the characteristic functions and their derivatives (\cref{p.lemma4}) and using an extension of the Fourier inversion formula. Afterwards we return to the unconditional setting by taking the expectation. On the left-hand side, we employ \cref{p.E-tP-Q} to move from $\E\Phi^{\tP}_{m,n}$ to $\Phi^Q_{m,n}$ while on the right-hand side, we evaluate the expectation of multiple rather involved random remainder terms (\cref{p.E.I,l.E.tL,l.exp.B_n,l.E.mom}), proving \cref{t.clt}.
However, when returning to the unconditional setting by taking expectations, we face similar problems as \cite{Chu46,Hal87} including the expectation of the absolute value of the conditional characteristic function. Same as our precursors, we solve this problem by demanding for non-singularity.

\begin{remark}[Non-singularity]\label{r.non-singularity}
	The condition of the distribution of $X_1$ being non-singular with respect to the Lebesgue measure is sometimes also denoted as the distribution of $X_1$ having a non-zero absolutely continuous component.
	For $Z_n$ the analogue condition is Cram\'er's condition \labelcref{eq.cramer} which we compare to non-singularity now.
	
	By the Lebesgue decomposition, for every distribution function there exists a unique decomposition
	\begin{equation}\label{eq.lebesgue-dec}
	F=p_{ac}F_{ac}+p_{d}F_{d}+p_{sc}F_{sc},
	\end{equation}
	with $p_{ac}+p_{d}+p_{sc}=1$, $p_\cdot\ge0$ where $F_\cdot$ are the distribution functions of an absolutely continuous distribution ($F_{ac}$), a discrete distribution ($F_{d}$) and a singular continuous distribution ($F_{sc}$) respectively. This corresponding decomposition holds for the characteristic function.
	
	By the Riemann--Lebesgue lemma \cite[Theorem 4.1 (iii)]{BR76normal},  $\limsup_{|t|\to\infty}|\varphi_{1,ac}(t)]|=0$ and by \cite{Sch41}, $\limsup_{|t|\to\infty}|\varphi_{1,d}(t)]|=1$ and $\limsup_{|t|\to\infty}|\varphi_{1,sc}(t)]|\in[0,1]$ (the $\limsup$ can in fact take all values in the interval).
	That is why non-singularity implies \labelcref{eq.cramer}. If $\limsup_{|t|\to\infty}|\varphi_{1,sc}(t)]|=1$ (for examples see \cite[p. 164]{Durrett05} or \cite{Sch41}) or $p_{sc}=0$, non-singularity and \labelcref{eq.cramer} are equivalent and both imply that the distribution of $X_1$ is not purely discrete.
	
	
	The reason for our need for the stronger condition of non-singularity is that we need a variant of Cram\'er's condition including the expectation of the absolute value of the conditional characteristic function in \labelcref{eq.cases} in the proof of \cref{p.E.I}. This requires the stronger assumption of non-singularity due to the absolute value within the expectation.
	
	For discrete lattice distributions (which are neither non-singular nor satisfy \labelcref{eq.cramer}), Edgeworth expansions for CLT and LLT of $Z_n$ can be found in \cite[Chapter 5]{BR76normal}. For $T_n$, the recent reference \cite{GvZ21} presents limit theorems and interesting occurring phenomena.
	
\end{remark}

\subsection{The conditional setting}\label{s.distr.cond}

Recall the definitions $M_n=\max |X_j|$, $B_n=V_n/M_n$, $\tL_{k,n}=V_n^{-k}\sum|X_j|^k$ and $\tphi_{\Phi^{\tP}_{m,n}}(t)
=e^{-t^2/2}\big(1+\sum_{r=1}^{m-2} \tU_{r,n}(it)\big)$, 
where
\[
\tU_{r,n}(it)
= \sum_{*(k_\cdot,r,\cdot)} (it)^{r+2u(k_\cdot)} \prod_{l=1}^{r} \frac{1}{k_l!} \Big(\frac{\tlambda_{l+2,n}}{(l+2)!}\Big)^{k_l}
\]
from \labelcref{eq.tU,eq.tphi} and the sum $\sum_{*(k_\cdot,r,\cdot)}$ is introduced in \cref{ch.pre}.

The following proposition examines differences of derivatives of the Fourier--Stieltjes transforms of $\tF_n$ and $\Phi^{\tP}_{m,n}$ around 0. Although versions of this proposition are a key part of every CLT or LLT including an Edgeworth expansion, the crucial point here is that the random variables $X_1,\dots,X_n$ conditional on $\cF_n$ are discrete and furthermore not identically distributed any longer.

\begin{proposition}\label{p.lemma4}
	Assume that $X_1$ is symmetric, $V_n^2>0$ and $m\ge2$ is an integer. Then
	\[
	\Big|\ddtk \Big(\tphi_{T_n}(t)-\tphi_{\Phi^{\tP}_{m,n}}(t)\Big)\Big|
	\le c(m)\tL_{m+1,n}e^{-t^2/6}\big(|t|^{m+1-k}+|t|^{3m-1+k}\big)
	\]
	holds for $k=0,\dots, m$ in the interval $|t|< B_n$.
\end{proposition}

Although there are many significant differences, the structure of our proof has its origins in the proof of \cite[Lemma 4, Chapter VI]{Pet75sums}. Prior to proving the proposition, we cite Lemma 2 from \cite[Chapter VI]{Pet75sums} in the conditional case. It gives an upper bound on $\tL_{a,n}$.

\begin{lemma}\label{l.lemma2}
	Let $X_1,\dots,X_n$ be independent random variables and $\E X_j=0, \E X_j^2<\infty$ for all $j=1,\dots,n$. If $3\le a\le b$, then
	\[
	\tL_{a,n}^{1/(a-2)}\le \tL_{b,n}^{1/(b-2)}.
	\]
\end{lemma}

\begin{proof}[Proof of \cref{p.lemma4}]
	$X_j$ has the conditional characteristic function
	\begin{equation*}
	\tphi_j(t)=\tE[\exp(itX_j)]=\frac{1}{2} \Big(\exp(it|X_j|)+\exp(-it|X_j|)\Big)
	=\cos(t|X_j|)
	\end{equation*}
	for $j=1,\dots,n$. Clearly $\cos(t|X_j|)=\cos(tX_j)$, but we keep the first notation to highlight that we are operating in the conditional setting where $|X_j|$ is deterministic. By conditional independence we can write for the conditional Fourier transform of $T_n$
	\begin{equation}\label{eq.phi_n}
	\tphi_{T_n}(t)
	=\prod\tE\big[\exp(itV_n^{-1}X_j)\big]
	=\prod\tphi_j(tV_n^{-1})
	=\prod \cos\big(tV_n^{-1}|X_j|\big).
	\end{equation}
	
	
	For applying the logarithm later, we need to bound $\tphi_j(tV_n^{-1})=\cos(tV_n^{-1}|X_j|)$ away from 0 for all $j=1,\dots,n$. So we restrict ourselves on the interval $|t|< B_n$ throughout the rest of the proof. Thus,
	\begin{equation}\label{eq.phi_j2}
	\tphi_j(tV_n^{-1})=\cos\big(tV_n^{-1}|X_j|\big)\ge \cos(1) > \tfrac12
	\end{equation}
	for all $j=1,\dots,n$.
	
	Since $X_j$ are symmetric, $\sum\tk_{1,j} = 0$ and $\sum\tk_{2,j} = V_n^2$. Recall $\tlambda_{l,n}=V_n^{-l}\sum\tk_{l,j}$ from \labelcref{eq.tlambda}. By taking the logarithm, we thus get (see \labelcref{eq.cumulant-abl})
	\begin{align*}
	\log \tphi_{T_n}(t)
	&=\sum\log\tphi_j(tV_n^{-1})
	=\sum\sum_{l=1}^{m}\tk_{l,j}\frac{(itV_n^{-1})^l}{l!}+\cO(t^{m+1})\\
	&=\sum_{l=1}^{m-2}\frac{\tlambda_{l+2,n}}{(l+2)!}(it)^{l+2}-\frac{t^2}{2}+\cO(t^{m+1}).
	\end{align*}
	The series expansion is valid as all conditional cumulants are finite. To eliminate the normal part (which is $-t^2/2$), we define
	\begin{equation}\label{eq.v_n}
	v_n(t,z):
	=\frac{t^2}{2}+\frac{1}{z^2}\log \tphi_{T_n}(tz)
	=\sum_{l=1}^{m-2}\frac{\tlambda_{l+2,n}}{(l+2)!}(it)^{l+2}z^l+\cO(t^{m+1},z^{m-1})
	\end{equation}
	for $0<z\le1$.
	Now we expand $e^{v_n(t,z)}$ as a series in $z$ (with fixed $t$)
	\begin{equation*}
	e^{v_n(t,z)}=1+\sum_{r=1}^{m-2}\tU_{r,n}(it)z^r+\cO(z^{m-1})
	\end{equation*}
	such that
	\begin{equation}\label{eq.phi_n^z}
	\tphi_{T_n}(tz)^{1/z^2}
	=\exp\Big(-\frac{t^2}{2}+v_n(t,z)\Big)
	=e^{-t^2/2}\Big(1+\sum_{r=1}^{m-2}\tU_{r,n}(it)z^r+R_n(t,z)\Big).
	\end{equation}
	Recalling $\tphi_{\Phi^{\tP}_{m,n}}(t) =e^{-t^2/2}\big(1+\sum_{r=1}^{m-2} \tU_{r,n}(it)\big)$, we now have to show that $R_n(t,z)$ and its derivatives are small for $z\to1$.
	
	First, we examine (see e.g. \cite[Theorem 3.3.18]{Durrett05})
	\begin{equation}\label{eq.dl-phi}
	\begin{split}
	\lefteqn{\Big|\ddtl \tphi_j(tzV_n^{-1})\Big|}\quad\\
	&=\Big|\tE\Big[\big(izV_n^{-1}X_j\big)^l\exp(itzV_n^{-1}X_j)\Big]\Big|\\
	&\le \tE\Big[z^lV_n^{-l}|X_j|^l\Big]
	= z^lV_n^{-l}|X_j|^l
	\end{split}
	\end{equation}
	for $j=1,\dots,n$ and $l\in\N$.
	
	According to the chain rule (see e.g. \cite[Lemma 1, Chapter VI]{Pet75sums}), we derive
	\begin{equation}\label{eq.dr-logphi}
	\begin{split}
	&\Big|\ddtr \log\tphi_j(tzV_n^{-1})\Big|\\
	&\quad=\Big|r!\sum_{*(k_\cdot,r,\cdot)} (-1)^{u(k_\cdot)-1}(u(k_\cdot)-1)!\tphi_j(tzV_n^{-1})^{-u(k_\cdot)} \prod_{l=1}^{r} \frac{1}{k_l!} \Big(\frac{1}{l!}\ddtl \tphi(tzV_n^{-1})\Big)^{k_l} \Big|\\
	&\quad\le c(r)\sum_{*(k_\cdot,r,\cdot)} \prod_{l=1}^{r}\big(z^lV_n^{-l}|X_j|^l\big)^{k_l}
	= c(r)z^rV_n^{-r}|X_j|^r
	\end{split}
	\end{equation}
	for $j=1,\dots,n$ and $r=1,\dots,m+1$, where we used \labelcref{eq.phi_j2,eq.dl-phi}.
	
	By Taylor expansion of $v_n(t,z)$ (at the point $0$, with an intermediate point $u\in\R$ between $0$ and $t$, depending on $v_n(\cdot,z)$ and $t$) using \labelcref{eq.cumulant-abl}, we get
	\begin{align*}
	v_n(t,z)&=\frac{t^2}{2}+\frac{1}{z^2}\sum\log\tphi_j(tzV_n^{-1})\\
	&=v_n(0,z)\\
	&\quad+t\Big(0+\frac{1}{z^2}\sum zV_n^{-1}\tk_{1,j}\,i\Big)\\
	&\quad+\frac{t^2}{2}\Big(1+\frac{1}{z^2}\sum (zV_n^{-1})^2\tk_{2,j}\,i^2\Big)\\
	&\quad+\frac{t^3}{6z^2}\sum\ddttt \log\tphi_j(tzV_n^{-1})\Big|_{t=u}\\
	&=\frac{t^3}{6z^2}\sum\ddttt \log\tphi_j(tzV_n^{-1})\Big|_{t=u}\, .
	\end{align*}
	%
	Now \labelcref{eq.dr-logphi} implies
	\begin{equation}\label{eq.v_n-1}
	|v_n(t,z)|
	\le\frac{|t|^3}{6z^2}\sum c(3)z^3V_n^{-3}|X_j|^3
	\le c(3)|t|^3V_n^{-3}\sum|X_j|^3
	=c(3)|t|^3\tL_{3,n}.
	\end{equation}
	
	By a Taylor expansion of $\ddt v_n(t,z)$ (at the point $0$, with an intermediate point $u\in\R$ between $0$ and $t$ and $u$ depending on $\ddt v_n(\cdot,z)$ and $t$), we get in the same manner
	\begin{align*}
	\ddt v_n(t,z)
	&=\frac{t^2}{2z^2}\sum\ddttt \log\tphi_j(tzV_n^{-1})\Big|_{t=u}
	\end{align*}
	which leads to
	\begin{equation}\label{eq.d1-v_n}
	\Big|\ddt v_n(t,z)\Big|
	\le\frac{t^2}{2z^2}\sum c(3)z^3V_n^{-3}|X_j|^3
	\le c(3)|t|^2V_n^{-3}\sum|X_j|^3
	=c(3)|t|^2\tL_{3,n}
	\end{equation}
	by \labelcref{eq.dr-logphi}. Similarly, 
	\begin{equation}\label{eq.d2-v_n}
	\Big|\ddtt v_n(t,z)\Big|\le c(3)|t|V_n^{-3}\sum|X_j|^3=c(3)|t|\tL_{3,n}.
	\end{equation}
	For $l=3,\dots,m$, the definition of $v_n(t,z)$ and \labelcref{eq.dr-logphi} directly give
	\begin{equation}\label{eq.dl-v_n}
	\Big|\ddtl v_n(t,z)\Big|
	\le\frac{1}{z^2}\sum c(l)z^lV_n^{-l}|X_j|^l
	= c(l)\tL_{l,n}\,z^{l-2}
	\le c(l)\tL_{l,n}.
	\end{equation}
	
	The necessity of bounding the next term will become evident in \labelcref{eq.R_n,eq.R_n^1,eq.R_n^2,eq.da-R_n^2}. For ${1\le r\le m \le h}$, the chain rule for powers of functions (see for example \cite[Lemma 3, Chapter VI]{Pet75sums}) and our derived estimates \labelcref{eq.v_n-1,eq.d1-v_n,eq.d2-v_n,eq.dl-v_n} gives
	\begin{align*}
	\lefteqn{\Big|\frac{1}{h!}\ddtr v_n^h(t,z)\Big|
	=\Big|\frac{r!\cdot h!}{h!} \sum_{u=1}^{r\land h} \sum_{*(k_\cdot,r,u)} \frac{v_n^{h-u}(t,z)}{(h-u)!} \prod_{l=1}^{r} \frac{1}{k_l!} \Big(\frac{1}{l!}\ddtl v_n(t,z)\Big)^{k_l}\Big|}\qquad\\
	&\le c(m)\sum_{u=1}^{r} \sum_{*(k_\cdot,r,u)} \frac{|v_n(t,z)|^{h-m}}{(h-m)!}|v_n(t,z)|^{m-u} \prod_{l=1}^{r} \Big|\ddtl v_n(t,z)\Big|^{k_l}\\
	&\le c(m) \frac{|v_n(t,z)|^{h-m}}{(h-m)!} \sum_{u=1}^{r} \sum_{*(k_\cdot,r,u)} 
	\big(|t|^3\tL_{3,n}\big)^{m-u} 
	\big(|t|^2\tL_{3,n}\big)^{k_1} 
	\big(|t|\tL_{3,n}\big)^{k_2} 
	\prod_{l=3}^{r} \tL_{l,n}^{k_l}\\
	&= c(m) \frac{|v_n(t,z)|^{h-m}}{(h-m)!} \sum_{u=1}^{r} \sum_{*(k_\cdot,r,u)} 
	|t|^{3m-3u+2k_1+k_2}
	\tL_{3,n}^{m-u+k_1+k_2} 
	\prod_{l=3}^{r} \tL_{l,n}^{k_l}.
	\end{align*}
	Now, \cref{l.lemma2} (with $b=m+1$) implies that this is bounded by
	\begin{align*}
	c(m) &\frac{|v_n(t,z)|^{h-m}}{(h-m)!} \sum_{u=1}^r \sum_{*(k_\cdot,r,u)} 
	|t|^{3m-3u+2k_1+k_2}
	\tL_{m+1,n}^{(m-u+k_1+k_2)/(m-1)} 
	\prod_{l=3}^{r} \tL_{m+1,n}^{(l-2)k_l/(m-1)}\\
	&= c(m) \frac{|v_n(t,z)|^{h-m}}{(h-m)!} \sum_{u=1}^r \sum_{*(k_\cdot,r,u)}
	|t|^{3m-3u+2k_1+k_2}
	\tL_{m+1,n}^{(m-3u+2k_1+k_2+r)/(m-1)}\\
	&= c(m) \frac{|v_n(t,z)|^{h-m}}{(h-m)!} \tL_{m+1,n} |t|^{3m-r-1}\sum_{u=1}^r \sum_{*(k_\cdot,r,u)}
	\big(|t|\tL_{m+1,n}^{1/(m-1)}\big)^{1-3u+2k_1+k_2+r},
	\end{align*}
	where the last exponent is greater than 0 due to
	\[2k_1+k_2+r=3k_1+3k_2+3k_3+4k_4+\dots+rk_r\ge3(k_1+k_2+\dots+k_r)=3u.\]
	For $k\ge3$, we can bound
	\begin{equation}\label{eq.tL-k-1}
	|t|^{k-2}\tL_{k,n}
	\le B_n^{k-2}\tL_{k,n}
	=\frac{V_n^{k-2}\sum|X_j|^k}{V_n^kM_n^{k-2}}
	\le\frac{\sum|X_j|^2}{V_n^2}
	=1
	\end{equation}
	and thus
	\begin{equation}\label{eq.tL-k-2}
	|t|\tL_{k,n}^{1/(k-2)}\le1
	\end{equation}
	which in case of $k=m+1$ leads to
	\begin{equation}\label{eq.dr-v_n^h}
	\Big|\frac{1}{h!}\ddtr v_n^h(t,z)\Big|\le c(m) \frac{|v_n(t,z)|^{h-m}}{(h-m)!} \tL_{m+1,n} |t|^{3m-r-1}
	\end{equation}
	for $1\le r\le m \le h$.
	
	We look back at \labelcref{eq.phi_n^z} and define
	\begin{equation}\label{eq.R_n}
	R_n(t,z)=R_{n,1}(t,z)+R_{n,2}(t,z),
	\end{equation}
	where
	\begin{align}\label{eq.R_n^1}
	R_{n,1}(t,z)
	&= \sum_{u=0}^{m-1}\frac{v_n^u(t,z)}{u!}-\Big(1+\sum_{r=1}^{m-2}\tU_{r,n}(it)z^r\Big),\\
	R_{n,2}(t,z)
	&= \sum_{u=m}^{\infty}\frac{v_n^u(t,z)}{u!}.\label{eq.R_n^2}
	\end{align}
	By \labelcref{eq.dr-v_n^h}, this implies (with $h=u$ and $r=a$)
	\begin{equation}\begin{split}\label{eq.da-R_n^2}
	\Big|\ddta R_{n,2}(t,z)\Big|
	&\le\sum_{u=m}^{\infty}\Big|\frac{1}{u!}\ddta v_n^u(t,z)\Big|\\
	&\le\sum_{u=m}^{\infty}c(m) \frac{|v_n(t,z)|^{u-m}}{(u-m)!} \tL_{m+1,n} |t|^{3m-a-1}\\
	&=c(m) e^{|v_n(t,z)|} \tL_{m+1,n} |t|^{3m-a-1}
	\end{split}\end{equation}
	for all $a=1,\dots, m$. By \labelcref{eq.v_n-1}, \labelcref{eq.tL-k-1} and \cref{l.lemma2},
	\begin{equation}\begin{split}\label{eq.d0-R_n^2}
	\Big|R_{n,2}(t,z)\Big|
	&\le\sum_{u=m}^{\infty}\Big|\frac{1}{u!} v_n^u(t,z)\Big|
	\\
	&\leq e^{|v_n(t,z)|}|v_n(t,z)|^{m}
	\\
	&\le c(m) e^{|v_n(t,z)|} \tL_{3,n}^{m-1}|t|^{3m-1}\\
	&
	\le c(m) e^{|v_n(t,z)|} \tL_{m+1,n}|t|^{3m-1},
	\end{split}\end{equation}
	so \labelcref{eq.da-R_n^2} also holds for $a=0$.
	
	
	Let $x>-1$. As
	\[
	\ddx \big(\log(1+x)-(x-x^2)\big) = \frac{1}{1+x}-1+2x
	\]
	is less than or equal to 0 for $x\in[-1/2,0]$ and greater than or equal to 0 for $x\ge 0$, the function $x\mapsto \log(1+x)-(x-x^2)$ is decreasing for $x\in[-1/2,0]$ and increasing for $x\ge 0$. Additionally, it is continuous and $\log(1+0)-(0-0^2)=0$. Hence, $\log(1+x)-(x-x^2)\ge0$ for $x\ge-1/2$ and therefore $\log(1+x)\ge x-x^2$ for $x\ge-1/2$.
	Additionally $\log(1+x)\le x$. Thus, for every $x\ge-1/2$ there exists a $\theta$ (depending on $x$) with $0 \le \theta \le 1$ such that $\log(1+x) = x-\theta x^2$.
	By \labelcref{eq.phi_j2}, $\tphi_j(tzV_n^{-1})>1/2$ for all $j=1,\dots,n$ and thus by inserting $x=\tphi_j(tzV_n^{-1})-1$ into the equation above,
	\begin{align*}
	|v_n(t,z)|
	&= \Big|\frac{t^2}{2}+\frac{1}{z^2}\sum\log\tphi_j(tzV_n^{-1})\Big|\\
	&= \Big|\frac{t^2}{2}+\frac{1}{z^2} \sum\Big(\big(\cos(tzV_n^{-1}|X_j|)-1\big)+\theta_j \big(\tphi_j(tzV_n^{-1})-1\big)^2\Big)\Big|
	\end{align*}
	for some $\theta_j$ (depending on $\tphi_j(tzV_n^{-1})$), where $|\theta_j|\le1$ for $j=1,\dots,n$. Additionally, there exists $\theta'$ (depending on $x$) with the properties $\cos(x)=1-x^2/2+\theta'x^4/(4!)$ and ${0\le\theta'\le1}$ and $|\cos(x)-1|\le x^2/2$ for $x\in\R$. By \labelcref{eq.tL-k-1}, we deduce
	\begin{align}
	|v_n(t,z)|
	&= \Big|\frac{t^2}{2}+\frac{1}{z^2} \sum\Big(\big(\cos(tzV_n^{-1}|X_j|)-1\big)+\theta_j \big(\tphi_j(tzV_n^{-1})-1\big)^2\Big)\Big|\nonumber\\
	&= \Big|\frac{1}{z^2} \sum\Big(\big(\theta_j'\tfrac{(tzV_n^{-1}|X_j|)^4}{24}\big)+\theta_j \big(\tphi_j(tzV_n^{-1})-1\big)^2\Big)\Big|\nonumber\\
	&\le \frac{t^4z^2}{24\,V_n^4} \sum|X_j|^4+\frac{1}{z^2}\sum \big|\tphi_j(tzV_n^{-1})-1\big|^2\label{eq.v_n-3}\\
	&\le \frac{t^4z^2}{24} \tL_{4,n}+\frac{1}{2z^2}\sum \big|\tphi_j(tzV_n^{-1})-1\big|\nonumber\\
	&\le \frac{t^2z^2}{24} (t^2\tL_{4,n})+\frac{1}{2z^2}\sum \frac{(tzV_n^{-1}|X_j|)^2}{2}
	\le \frac{t^2}{3}\nonumber
	\end{align}
	for some $\theta_j'$ (depending on $tzV_n^{-1}|X_j|$), where $0\le\theta_j'\le1$ for $j=1,\dots,n$.
	
	Putting together \labelcref{eq.da-R_n^2,eq.d0-R_n^2,eq.v_n-3} implies
	\begin{equation}\begin{split}\label{eq.da-R_n^2-2}
	\Big|\ddta R_{n,2}(t,z)\Big|
	\le c(m) e^{|v_n(t,z)|} \tL_{m+1,n} |t|^{3m-a-1}
	\le c(m) \tL_{m+1,n} |t|^{3m-a-1}e^{t^2/3}
	\end{split}\end{equation}
	for all $a=0,\dots,m$.
	
	Now to $R_{n,1}$. By \labelcref{eq.v_n}, we have
	\begin{equation}\label{eq.v_n-4}
	v_n(t,z)=\sum_{k=1}^{m-2}\frac{\tlambda_{k+2,n}}{(k+2)!}(it)^{k+2}z^k+r_n(t,z)z^{m-1},
	\end{equation}
	where $r_n$ is the remainder of a Taylor expansion (at the point $0$, with intermediate point $u\in\R$ between $0$ and $t$ and $u$ depending on $v_n(\cdot,z)$ and $t$)
	\begin{equation*}
	r_n(t,z)=\frac{t^{m+1}}{(m+1)!z^{m-1}}\;\ddT{m+1} v_n(t,z)\Big|_{t=u}.
	\end{equation*}
	
	For $l=0,\dots,m$ and $j=1,\dots,n$, by Taylor expansion of $\ddtl\log\tphi_j(tzV_n^{-1})$ (at $0$, with intermediate point $u_j\in\R$ between $0$ and $t$, depending on $\ddtl\log\tphi_j(\cdot \,zV_n^{-1})$ and $t$) and \labelcref{eq.cumulant-abl},
	\begin{align*}
	\lefteqn{\ddtl\log\tphi_j(tzV_n^{-1})}\quad\\*
	&=\sum_{k=0}^{m-l} \frac{t^{k}}{k!}\;\ddT{l+k}\log\tphi_j(tzV_n^{-1})\Big|_{t=0}
	+\frac{t^{m-l+1}}{(m-l+1)!}\;\ddT{m+1}\log\tphi_j(tzV_n^{-1})\Big|_{t=u_j}\\
	&=\sum_{k=l}^{m} \frac{\tk_{k,j}}{V_n^k}z^{k}i^{k} \frac{t^{k-l}}{(k-l)!}
	+\frac{t^{m-l+1}}{(m-l+1)!}\;\ddT{m+1}\log\tphi_j(tzV_n^{-1})\Big|_{t=u_j}
	\end{align*}
	for some $0\le u_j\le t$.
	For $l=0,\dots,m$, differentiating \labelcref{eq.v_n-4} and using \labelcref{eq.dr-logphi} implies
	\begin{align*}
	\lefteqn{\Big|\ddtl r_n(t,z)\Big|
	}\:\:\\
	&= z^{-(m-1)}\Big|\ddtl\Big(\frac{t^2}{2}+\frac{1}{z^2}\sum\log\tphi_j(tzV_n^{-1})\Big) -\ddtl\sum_{k=1}^{m-2}\frac{\tlambda_{k+2,n}}{(k+2)!}(it)^{k+2}z^k\Big|\\
	&= z^{-(m-1)}\Big|\Big(\ddtl\frac{t^2}{2}\Big)+\Big(\frac{1}{z^2}\sum\ddtl\log\tphi_j(tzV_n^{-1})\Big)
	-\sum_{k=3\vee l}^{m}\tlambda_{k,n}i^{k}\frac{t^{k-l}}{(k-l)!}z^{k-2} \Big|\\
	&= z^{-(m-1)}\bigg|\Big(\ddtl\frac{t^2}{2}\Big)\\
	&\qquad+\Big(\frac{1}{z^2}\sum\Big(\sum_{k=l}^{m} \frac{\tk_{k,j}}{V_n^k}z^{k}i^{k} \frac{t^{k-l}}{(k-l)!}
	+\Big(\frac{t^{m-l+1}}{(m-l+1)!}\;\ddT{m+1}\log\tphi_j(tzV_n^{-1})\Big|_{t=u_j}\Big)\Big)\Big)\\
	&\qquad-\sum_{k=l}^{m}\tlambda_{k,n}i^{k}\frac{t^{k-l}}{(k-l)!}z^{k-2} +\1_{\{l\le2\}}\tlambda_{2,n}i^{2}\frac{t^{2-l}}{(2-l)!}\,\bigg|\\
	&= z^{-(m-1)}\Big|\frac{1}{z^2}\sum\frac{t^{m-l+1}}{(m-l+1)!}\;\ddT{m+1}\log\tphi_j(tzV_n^{-1})\Big|_{t=u_j}\Big|\\
	&\le z^{-(m+1)}\frac{|t|^{m-l+1}}{(m-l+1)!}\sum c(m)z^{m+1}V_n^{-(m+1)}|X_j|^{m+1}\\
	&= c(m)|t|^{m-l+1}\tL_{m+1,n}.
	\end{align*}
	For $r,k=0,\dots,m$ by chain rule, this leads to
	\begin{align*}
	\Big|\ddtr r_n^k(t,z)\Big|
	&=\Big|r!\cdot k!\sum_{u=1}^{r\land k} \sum_{*(k_\cdot,r,u)} \frac{r_n^{k-u}(t,z)}{(k-u)!} \prod_{l=1}^{r} \frac{1}{k_l!} \Big(\frac{1}{l!}\ddtl r_n(t,z)\Big)^{k_l}\Big|\\
	&\le c(m)\sum_{u=1}^{r\land k} \sum_{*(k_\cdot,r,u)} \Big(|t|^{m+1}\tL_{m+1,n}\Big)^{k-u} \prod_{l=1}^{r} \Big(|t|^{m-l+1}\tL_{m+1,n}\Big)^{k_l}\\
	&= c(m)\sum_{u=1}^{r\land k} \sum_{*(k_\cdot,r,u)} |t|^{(m+1)(k-u)}\tL_{m+1,n}^{k-u} |t|^{(m+1)u-r}\tL_{m+1,n}^{u}\\
	&= c(m)|t|^{(m+1)k-r}\tL_{m+1,n}^{k}
	\end{align*}
	and for $a,k=0,\dots,m$ and $v=0,\dots,3m$ by the Leibniz rule,
	\begin{equation}\label{eq.da-mix}
	\begin{split}
	\Big|\ddta \big(r_n^{k}(t,z)\ t^{v-(m+1)k}\big)\Big|
	&\le c(m)\sum_{b=0}^a\Big(|t|^{(m+1)k-b}\tL_{m+1,n}^{k}\Big)\Big(|t|^{v-(m+1)k-a+b}\Big)\\
	&\le c(m)|t|^{v-a}\tL_{m+1,n}^{k}.\\
	\end{split}
	\end{equation}
	In \labelcref{eq.R_n^1} the summand with $u=0$ of the first sum and the 1 cancel each other out, so we analyse the remaining first sum using \labelcref{eq.v_n-4}
	\begin{align*}
	\sum_{u=1}^{m-1}\frac{v_n^u(t,z)}{u!}
	&=\sum_{u=1}^{m-1}\frac{1}{u!}\Big(
	\sum_{l=1}^{m-2}\frac{\tlambda_{l+2,n}}{(l+2)!}(it)^{l+2}z^l+r_n(t,z)z^{m-1}
	\Big)^u\\
	&=\sum_{u=1}^{m-1}\sum_{*(k_\cdot,\cdot,u,m-1)}
	\prod_{l=1}^{m-2}\frac{1}{k_l!}\Big(\frac{\tlambda_{l+2,n}}{(l+2)!}(it)^{l+2}z^l\Big)^{k_l}\;\frac{1}{k_{m-1}!}\Big(r_n(t,z)z^{m-1}\Big)^{k_{m-1}}\\
	&=\sum_{u=1}^{m-1}\sum_{*(k_\cdot,\cdot,u,m-1)}
	\prod_{l=1}^{m-2}\frac{1}{k_l!}\Big(\frac{\tlambda_{l+2,n}}{(l+2)!}\Big)^{k_l}\;\frac{1}{k_{m-1}!}\Big(r_n(t,z)\Big)^{k_{m-1}} \\ 
	&\qquad\qquad\qquad\qquad \cdot (it)^{r(k_\cdot)+2u-(m+1)k_{m-1}}\,z^{r(k_\cdot)}.
	\end{align*}
	Let $r(k_\cdot)\le m-2$, so particularly $u\le r\le m-2$. Since $k_r$ is the highest $k_l$ that can be greater 0, also $k_{r+1}=k_{r+2}=\dots =k_{m-1}=0$. After shortening the term, we see that instead of summing over $u$, we can sum over $r$ as well without changing any summand. This yields
	\begin{align*}
	&\sum_{u=1}^{m-1}\sum_{*(k_\cdot,\cdot,u,m-1)} \1_{\{r(k_\cdot)\le m-2\}}
	\prod_{l=1}^{m-2}\frac{1}{k_l!}\Big(\frac{\tlambda_{l+2,n}}{(l+2)!}\Big)^{k_l}\;\frac{1}{k_{m-1}!}\Big(r_n(t,z)\Big)^{k_{m-1}} \\*
	&\qquad\qquad\qquad\qquad\qquad\qquad \cdot (it)^{r(k_\cdot)+2u-(m+1)k_{m-1}}\,z^{r(k_\cdot)}\\
	&=\sum_{u=1}^{m-2}\sum_{*(k_\cdot,\cdot,u,m-2)} \1_{\{r(k_\cdot)\le m-2\}}
	\prod_{l=1}^{r(k_\cdot)}\frac{1}{k_l!}\Big(\frac{\tlambda_{l+2,n}}{(l+2)!}\Big)^{k_l}(it)^{r(k_\cdot)+2u}\,z^{r(k_\cdot)}\\
	&=\sum_{u=1}^{m-2}\sum_{r=1}^{m-2}\sum_{*(k_\cdot,r,u)}
	\prod_{l=1}^{r}\frac{1}{k_l!}\Big(\frac{\tlambda_{l+2,n}}{(l+2)!}\Big)^{k_l}(it)^{r+2u}\,z^{r}\\
	&=\sum_{r=1}^{m-2}\sum_{*(k_\cdot,r,\cdot)}
	\prod_{l=1}^{r}\frac{1}{k_l!}\Big(\frac{\tlambda_{l+2,n}}{(l+2)!}\Big)^{k_l}(it)^{r+2u(k_\cdot)}\,z^r\\
	&=\sum_{r=1}^{m-2}\tU_{r,n}\,z^r.
	\end{align*}
	As a result, in \labelcref{eq.R_n^1} everything is eliminated except for $c(m)$ many summands from the first sum, having the form
	\begin{equation}\label{tlambda-bound}
	c(m)\prod_{l=1}^{m-2}\tlambda_{l+2,n}^{k_l}\;r_n(t,z)^{k_{m-1}}(it)^{r(k_\cdot)+2u-(m+1)k_{m-1}}\,z^{r(k_\cdot)}
	\end{equation}
	for some $k_1,\dots, k_{m-1}\ge0, \sum_{l=1}^{m-1}lk_l=r(k_\cdot)\ge m-1, u=\sum_{l=1}^{m-1}k_l, 1\le u\le m-1$. The sum over these $k_i$ will be denoted by $\sum_{*(k_\cdot,\cdot,u,m-1)*}$.
	The cumulants are upper bounded by constants times the absolute moments, so by \cref{l.lemma2},
	\[\tlambda_{l+2,n}
	=\frac{\sum\tk_{l+2,j}}{V_n^{l+2}}
	\le c(l)\frac{\sum|X_j|^{l+2}}{V_n^{l+2}}
	= c(l)\tL_{l+2,n}
	\le c(l) \tL_{m+1,n}^{l/(m-1)}
	\]
	for $l=1,\dots,m-1$. For the second part of \labelcref{tlambda-bound}, we use \labelcref{eq.da-mix} (with $k=k_{m-1}$ and ${v=r+2u}$) and get by \labelcref{eq.tL-k-2}
	\begin{equation}\label{eq.da-R_n^1}
	\begin{split}
	\Big|\ddta R_{n,1}(t,z)\Big|
	&\le c(m)\sum_{*(k_\cdot,\cdot,u,m-1)*}\prod_{l=1}^{m-2}\tL_{m+1,n}^{k_l\cdot l/(m-1)}|t|^{r(k_\cdot)+2u-a}\tL_{m+1,n}^{k_{m-1}}\,z^{r(k_\cdot)}\\
	&\le c(m)\sum_{*(k_\cdot,\cdot,u,m-1)*}\tL_{m+1,n}^{r(k_\cdot)/(m-1)}|t|^{r(k_\cdot)+2u-a}\\
	&\le c(m)\sum_{*(k_\cdot,\cdot,u,m-1)*}\tL_{m+1,n}\Big(\tL_{m+1,n}^{1/(m-1)}|t|\Big)^{r(k_\cdot)-(m-1)}|t|^{m-1+2u-a}\\
	&\le c(m)\sum_{*(k_\cdot,\cdot,u,m-1)*}\tL_{m+1,n}|t|^{m-1+2u-a}\\
	&\le c(m)\tL_{m+1,n}\Big(|t|^{m+1-a}+|t|^{3m-3-a}\Big)
	\end{split}
	\end{equation}
	for $a=0,\dots,m$.
	Combining \labelcref{eq.R_n,eq.da-R_n^2-2,eq.da-R_n^1} gives
	\begin{equation*}
	\Big|\ddta R_n(t,z)\Big|
	\le c(m) \tL_{m+1,n} \Big(|t|^{m+1-a}+|t|^{3m-1-a}\Big)e^{t^2/3}
	\end{equation*}
	for $a=0,\dots,m$ and by the Leibniz rule
	\begin{align*}
	\lefteqn{\Big|\ddtk \Big(e^{-t^2/2}R_n(t,z)\Big)\Big|}\quad\\
	&\le c(m)\sum_{l=0}^k (1+|t|^l)e^{-t^2/2}\tL_{m+1,n} \Big(|t|^{m+1-(k-l)}+|t|^{3m-1-(k-l)}\Big)e^{t^2/3}\\
	&\le c(m) \tL_{m+1,n} e^{-t^2/6} \Big(|t|^{m+1-k}+|t|^{3m-1+k}\Big)
	\end{align*}
	for $k=0,\dots,m$. By \labelcref{eq.phi_n^z}, we get
	\begin{equation*}
	\tphi_{T_n}(tz)^{1/z^2}-e^{-t^2/2}\Big(1+\sum_{r=1}^{m-2}\tU_{r,n}(it)z^r\Big)
	=e^{-t^2/2}R_n(t,z)
	\end{equation*}
	and by differentiating $k$ times and setting $z=1$, this implies
	\begin{align*}
	&\Big|\ddtk \Big(\tphi_{T_n}(t)-e^{-t^2/2}\Big(1+\sum_{r=1}^{m-2}\tU_{r,n}(it)\Big)\Big)\Big|
	\le c(m) \tL_{m+1,n} e^{-t^2/6} \Big(|t|^{m+1-k}+|t|^{3m-1+k}\Big)
	\end{align*}
	for $k=0,\dots,m$.
\end{proof}

In the next step, we prove the non-uniform bound for our Edgeworth expansion in the conditional setting.
Recall the definition $\Phi^{\tP}_{m,n}(x)=\Phi(x)+\sum_{r=1}^{m-2}\tP_{r,n}(x)$, where
\[
\tP_{r,n}(x)=-\phi(x)\sum_{*(k_\cdot,r,\cdot)} H_{r+2u(k_\cdot)-1}(x) \prod_{l=1}^{r} \frac{1}{k_l!} \Big(\frac{\tlambda_{l+2,n}}{(l+2)!}\Big)^{k_l}
\]
from \labelcref{eq.tP,eq.Phi^tP}.

\begin{proposition}\label{p.cond}
	Assume that $X_1$ is symmetric, $V_n^2>0$ and $m\ge2$ is an integer. Then 
	\begin{equation}\label{eq.pcond}
	\begin{split}
	&\sup_{x\in\R} \, (1+|x|)^{m}|\tF_n(x)-\Phi^{\tP}_{m,n}(x)|
	\le c(m) \Big(\tL_{m+1,n} 
	+ I_2
	+ \sum_{k=0}^{m} I_{k,2}
	+ e^{-B_n^2/4}\Big),
	\end{split}
	\end{equation}
	where $I_2=\int_{B_n\le|t|<\tL_{m+1,n}^{-1}}|t|^{-1}\big|\tphi_{T_n}(t)\big|\dt$ and for $k=0,\dots,m$,
	\[I_{k,2}=\int_{B_n\le|t|<\tL_{m+1,n}^{-1}}|t|^{k-m-1}\Big|\ddtk \tphi_{T_n}(t)\Big|\dt\].
	\end{proposition}

\begin{proof}
	First, we use Lemma 8 with $F=\tF_n$ and $G=\Phi^{\tP}_{m,n}$, $T=\tL_{m+1,n}^{-1}$ and Lemma 7 with $G=\tF_n-\Phi^{\tP}_{m,n}$ from \cite[Chapter VI]{Pet75sums}, each applied with conditional Fourier transforms and finite $m$-th conditional moment. Next, $\ddx \Phi^{\tP}_{m,n}=\phi^{\tp}_{m,n}$, thus by \labelcref{eq.phi^tp,eq.tp-bound},
	\[
	\big(1+|x|\big)^m \Big| \ddx \Phi^{\tP}_{m,n}(x)\Big|
	= \big(1+|x|\big)^m \big| \phi^{\tP}_{m,n}(x)\big|
	= \big(1+|x|\big)^m \Big| \phi(x)+\sum_{r=1}^{m-2}\tp_{r,n}(x) \Big|
	\le c(m)
	\]
	uniformly for all $x$ and $n$ and all conditions of the lemmas are satisfied.
	Let
	\[\delta_m(t)=\int_{-\infty}^{\infty}e^{itx}\dd\big(x^m(\tF_n(x)-\Phi^{\tP}_{m,n}(x))\big)\,.\]
	Now the lemmas yield
	\begin{align}\label{eq.pcond-1}
	&\big|\tF_n(x)-\Phi^{\tP}_{m,n}(x)\big|\nonumber\\*
	&\le c(m)(1+|x|)^{-m}\Big(\int_{-\tL_{m+1,n}^{-1}}^{\tL_{m+1,n}^{-1}}\Big|\frac{\tphi_{T_n}(t)-\tphi_{\Phi^{\tP}_{m,n}}(t)}{t}\Big|\dt + \int_{-\tL_{m+1,n}^{-1}}^{\tL_{m+1,n}^{-1}}\Big|\frac{\delta_m(t)}{t}\Big|\dt + \frac{c(m)}{\tL_{m+1,n}^{-1}}\Big)\nonumber\\
	&\le c(m)(1+|x|)^{-m}
	\Big(\int_{-\tL_{m+1,n}^{-1}}^{\tL_{m+1,n}^{-1}}|t|^{-1}\big|\tphi_{T_n}(t)-\tphi_{\Phi^{\tP}_{m,n}}(t)\big|\dt\nonumber\\*
	&\qquad\qquad\qquad\qquad+ \sum_{k=0}^{m}\int_{-\tL_{m+1,n}^{-1}}^{\tL_{m+1,n}^{-1}}|t|^{k-m-1}\Big|\ddtk \tphi_{T_n}(t)-\tphi_{\Phi^{\tP}_{m,n}}(t)\Big|\dt + c(m) \, \tL_{m+1,n} \Big)\nonumber\\
	&= c(m)(1+|x|)^{-m}
	\Big(\underbrace{\int_{-\tL_{m+1,n}^{-1}}^{\tL_{m+1,n}^{-1}}|t|^{-1}\big|\tphi_{T_n}(t)-\tphi_{\Phi^{\tP}_{m,n}}(t)\big|\dt}_{=I}\nonumber\\
	&\qquad\qquad\qquad\qquad + \sum_{k=0}^{m}\underbrace{\int_{|t|<B_n}|t|^{k-m-1}\Big|\ddtk \tphi_{T_n}(t)-\tphi_{\Phi^{\tP}_{m,n}}(t)\Big|\dt}_{=I_{k,1}}\nonumber\\
	&\qquad\qquad\qquad\qquad + \sum_{k=0}^{m}\underbrace{\int_{B_n\le|t|<\tL_{m+1,n}^{-1}}|t|^{k-m-1}\Big|\ddtk \tphi_{T_n}(t)\Big|\dt}_{=I_{k,2}}\\
	&\qquad\qquad\qquad\qquad + \sum_{k=0}^{m}\underbrace{\int_{B_n\le|t|<\tL_{m+1,n}^{-1}}|t|^{k-m-1}\Big|\ddtk \tphi_{\Phi^{\tP}_{m,n}}(t)\Big|\dt}_{=I_{k,3}}\nonumber\\*
	&\qquad\qquad\qquad\qquad + c(m) \, \tL_{m+1,n}\Big)\nonumber
	\end{align}
	for $x\in\R$ with $\tphi_{\Phi^{\tP}_{m,n}}$ from \labelcref{eq.tphi}.
	
	Since in our conditional setting all moments exist, we can deduce from \cref{p.lemma4}
	\begin{align*}
	I_{k,1}
	&\le\int_{|t|<B_n}|t|^{k-m-1}c(m)\tL_{m+1,n}e^{-t^2/6}\big(|t|^{m+1-k}+|t|^{3m-1+k}\big)\dt\\
	&\le c(m)\tL_{m+1,n}\int_{|t|<B_n}e^{-t^2/6}\big(1+|t|^{2m-2+2k}\big)\dt\le c(m)\tL_{m+1,n}
	\end{align*}
	for $k=0,\dots,m$.
	
	
	Now we examine $I_{k,3}$. 
	By \labelcref{eq.ddtl-tphi} and since $|t|\ge B_n\ge 1$,
	\begin{equation}\label{eq.Ik3}
	\begin{split}
	I_{k,3}
	&=\int_{B_n\le|t|<\tL_{m+1,n}^{-1}}|t|^{k-m-1}\Big|\ddtk \tphi_{\Phi^{\tP}_{m,n}}(t)\Big|\dt\\
	&\le c(m) \int_{|t|\ge B_n} |t|^{m-5+2k} e^{-t^2/2} \dt
	\le c(m) e^{-B_n^2/4}
	\end{split}
	\end{equation}
	for $k=0,\dots,m$, where we use \cref{l.int-t-exp} (with $\beta=1$, $\nu_n=B_n$) in the last step.
	
	In the same manner we proceed for $I$ and get
	\begin{align*}
	I
	&=\int_{-\tL_{m+1,n}^{-1}}^{\tL_{m+1,n}^{-1}}|t|^{-1}\big|\tphi_{T_n}(t)-\tphi_{\Phi^{\tP}_{m,n}}(t)\big|\dt\\
	&\le\int_{|t|<B_n}c(m)\tL_{m+1,n}e^{-t^2/6}\big(|t|^{m}+|t|^{3m-2}\big)\dt\\
	&\quad + \underbrace{\int_{B_n\le|t|<\tL_{m+1,n}^{-1}}|t|^{-1}\Big|\tphi_{T_n}(t)\Big|\dt}_{=I_2} \ +\  \int_{|t|\ge B_n}\Big| e^{-t^2/2}\Big(1+\sum_{r=1}^{m-2} \tU_{r,n}(it)\Big)\Big|\dt\\
	&\le c(m)\tL_{m+1,n} + I_2 + c(m) e^{-B_n^2/4}.
	\end{align*}
	Inserting everything in \labelcref{eq.pcond-1}, we arrive at
	\begin{equation*}
	\big|\tF_n(x)-\Phi^{\tP}_{m,n}(x)\big|
	\le c(m)(1+|x|)^{-m}
	\Big(\tL_{m+1,n} 
	+ I_2
	+ \sum_{k=0}^{m} I_{k,2}
	+ e^{-B_n^2/4}\Big)
	\end{equation*}
	for $x\in\R$, which can be rewritten as \labelcref{eq.pcond}.
\end{proof}

\subsection{Proof of Theorem \ref{t.clt}}\label{s.distr.main}

Before proving the main theorem, we formulate some auxiliary propositions in which we bound expectation values of summands from \labelcref{eq.pcond} separately.

The proof of the following proposition involves thorough investigation and precise bounding of integral terms. The crucial point here is that $B_n$ is random and its expectation is smaller in order than $n^{-1/2}$ for our self-normalized sum $T_n$. In contrast, for the classical CLT-statistic $Z_n$, the integration interval starts at a deterministic point of order $n^{-1/2}$. Our larger integration interval complicates the bounding of the integral enormously and finally leads to the logarithmic factor in the remainder of \cref{t.clt}. The randomness in our integral, combined with the expectation results in a proof that differs fundamentally from usual proofs for Edgeworth expansions.

\begin{proposition}\label{p.E.I}
	Assume that the distribution of $X_1$ is non-singular and ${\E|X_1|^{s}<\infty}$ for some ${s}\ge2$. Let $\nu_n,\tau_n\to \infty$ in polynomial order, bounded by $\nu_n,\tau_n\le n^{c}$ for some ${c= c(s)}$. Then for $m=\lfloor {s} \rfloor$ and $\zeta \in \R$,
	\begin{align*}
	&\E\bigg[\1_{\{\zeta < M_n< \tau_n\}} \int_{B_n\le|u|< \nu_n} \Big|\ddur \prod \cos\big(uV_n^{-1}|X_j|\big)\Big|\du \bigg]\\
	&\quad= \co\big( n^{-({s}-2)/2} \, (\log n)^{({s}+\max\{2r,1\})/2} \big)
	\end{align*}
	holds for all $r=0,\dots,m$.
\end{proposition}

The following lemma treats the expectation of mixed moments (with an additional term) and is required in the proof of \cref{p.E.I}. Its proof can be found in \cref{app.proofs.distr}.

\begin{lemma}\label{l.E.mom}
Assume that $\E |X_1|^{{s}}<\infty$ for some ${s}\ge2$. Then for all $l=1,\dots, \lfloor s \rfloor$, ${1\le a_i \le {s}}$ for $i=1,\dots,l$, and $1\le \eta \le c(s)$,
\begin{align*}
\E\Big[\1_{\big\{M_n\ge \sqrt{\frac{n}{\log(n)\, \eta}}\big\}} |X_1|^{a_1} \cdots |X_l|^{a_l}\Big]
= \co\Big(n^{-({s}-\max\{2,a_1,\dots,a_l\})/2} \, (\log n)^{{s}/2} \, \eta^{{s}/2} \Big).
\end{align*}
\end{lemma}

\begin{remark}[Moment assumptions]\label{r.moment}
	Assuming finite moments ($\E |X_1|^{{s}}<\infty$ for some ${s}\ge2$) provides useful truncation inequalities. We truncate the maximum $M_n$ at $a_n\to\infty$ and get by Markov's inequality
	\begin{equation}\label{eq.Mn-an}
	\P\big(M_n\ge a_n\big)
	\le n \, \P\big(|X_1|\ge a_n\big)
	\le n \, a_n^{-{s}} \int_{|x|\ge a_n} |x|^{{s}} \dd\P_{X_1}(x)
	= \co \big( n \, a_n^{-{s}} \big).
	\end{equation}
	Note 
	that
	$\E\big[ \1_{\{|X_1|\ge n^{1/4}\}} |X_1|^s \big] \too 0
	$ 
	for $n\to\infty$. Let $(\delta_n)$ be a sequence that satisfies $n^{-1/4} \le \delta_n <1$ and $\delta_n\to0$ sufficiently slowly that
	\begin{align*}
	\delta_n^{-s} \E\Big[ \1_{\{|X_1|\ge n^{1/4}\}} |X_1|^s \Big] \too 0.
	\end{align*}
	Then,
	\begin{align}\label{eq.delta_n-def}
	\E\Big[ \1_{\{|X_1|\ge \delta_n \sqrt{n} \, \}} \delta_n^{-s} |X_1|^s \Big]
	\le \delta_n^{-s} \E\Big[ \1_{\{|X_1|\ge n^{-1/4} \sqrt{n} \, \}} |X_1|^s \Big]
	 \too 0
	\end{align}
	and thus
	\begin{equation}\label{eq.Mn-delta}
	\begin{split}
	\P\big(M_n\ge \delta_n \sqrt{n} \, \big)
	&\le n \, (\delta_n \sqrt{n} \, )^{-{s}} \int_{|x|\ge \delta_n \sqrt{n}} |x|^{{s}} \dd\P_{X_1}(x)\\
	&= n^{-(s-2)/2} \E\Big[ \1_{\{|X_1|\ge \delta_n \sqrt{n} \, \}} \delta_n^{-s} |X_1|^s \Big]
	= \co \big( n^{-(s-2)/2} \big).
	\end{split}
	\end{equation}
	
	Additionally by \cite[Theorem 28, Chapter IX]{Pet75sums},
	\begin{equation*}
	\P\big(|V_n^2-n\mu_2|\ge \epsilon n\big)
	=\co\big(n^{-({s}-2)/2}\big)
	\end{equation*}
	for all $\epsilon>0$ and particularly (by choosing $\epsilon=\mu_2/2=1/2$)
	\begin{align}\label{eq.Vn-bound-1/2}
	\P\big(V_n^2 \le n/2\big)
	&\le\P\big(|V_n^2-n|\ge n/2\big)
	=\co\big(n^{-({s}-2)/2}\big).
	\end{align}
\end{remark}

\begin{proof}[Proof of \cref{p.E.I}]
	For the sake of legibility, we suppress the notation of the indicator $\1_{\{\zeta < M_n< \tau_n\}}$ which is denoted by $\E_*[\,\cdot\,]:=\E[\1_{\{\zeta < M_n< \tau_n\}}\,\cdot\,]$ throughout this proof. Every time these bounds are used, we will mention it explicitly. In the following set of inequalities, we use the indicator. By $V_n\ge M_n >\zeta$ and Fubini's theorem, we get
	\begin{equation}\label{eq.fubini}
	\begin{split}
	\lefteqn{\E_*\bigg[\int_{B_n\le|u|< \nu_n} \Big|\ddur \prod \cos\big(uV_n^{-1}|X_j|\big)\Big|\du\bigg]}\quad\\
	&= \E_*\bigg[\int_{M_n^{-1}\le|u|< \nu_nV_n^{-1}} V_n^{-(r-1)}\Big|\ddur \prod \cos(u|X_j|)\Big|\du\bigg]\\
	&\le \E_*\bigg[\int_{0\le|u|< \nu_n\zeta^{-1}} \1_{\{M_n^{-1}\le|u|\}} V_n^{-(r-1)}\Big|\ddur \prod \cos(u|X_j|)\Big|\du\bigg]\\
	&= \int_{0\le|u|< \nu_n\zeta^{-1}}\E_*\Big[\1_{\{M_n^{-1}\le|u|\}} V_n^{-(r-1)}\Big|\ddur \prod \cos(u|X_j|)\Big|\Big]\du.
	\end{split}
	\end{equation}
	
	Subsequently, we deduce rather tight $u$-dependent bounds on the $\E_*$ expression within the integral. To this end, we derive precise bounds on the derivatives of the product inside.
	
	The higher derivative of a product can be evaluated by the generalized Leibniz formula which states in our case
	\[
	\ddur \prod \cos(u|X_j|)
	=\sum _{i_{1}+i_{2}+\cdots +i_{n}=r} {r \choose i_{1},\ldots ,i_{n}}\prod\Big(\ddU{i_j}\cos(u|X_j|)\Big)\,.
	\]
	Here the sum extends over all $n$-tuples $(i_1,...,i_n)$ of non-negative integers with $\sum_{j=1}^ni_j=r$ and 
	\[
	{r \choose i_{1},\ldots ,i_{n}}={\frac {r!}{i_{1}!\,i_{2}!\cdots i_{n}!}}.
	\]
	Since $X_1,\dots,X_n$ are i.i.d.,
	\begin{align*}
	\E\Big[ {r \choose i_{\pi(1)},\ldots ,i_{\pi(n)}}\prod\Big(\ddU{i_{\pi(j)}}\cos(u|X_{j}|)\Big)	\Big]
	=	\E\Big[ {r \choose i_{1},\ldots ,i_{n}}\prod\Big(\ddU{i_j}\cos(u|X_j|)\Big)	\Big]
	\end{align*}
	for all (non-random) permutations $\pi$ from $\{1,\dots,n\}$ to itself.
	
	
	For $l=1,\dots,r$, let $k_l$ be the number of appearances of $l$ in the tuple $(i_1,...,i_n)$ (so $\sum_{l=1}^rlk_l=\sum_{j=1}^ni_j=r$). Then for each tuple $(k_1,\dots,k_r)$, there exist ${n \choose k_{1},\ldots ,k_{r},(n-\sum_{l=1}^rk_l)}$ tuples $(i_1,...,i_n)$ such that $k_l$ is the number of appearances of $l$ in the tuple $(i_1,...,i_n)$. So instead of summing over all $n$-tuples $(i_1,...,i_n)$ with $\sum_{j=1}^ni_j=r$, we can sum over all $r$-tuples $(k_1,...,k_r)$ of non-negative integers with $\sum_{l=1}^rlk_l=r$ and multiply each summand by ${n \choose k_{1},\ldots ,k_{r},(n-\sum_{l=1}^rk_l)}$. We perform the just described summation by using the sum $\sum_{*(k_\cdot,r,\cdot)}$ and $u(k_\cdot)$ which are explained in \cref{ch.pre}. Additionally, we convert
	\[
	{r \choose i_{1},\ldots ,i_{n}}={\frac {r!}{i_{1}!\,i_{2}!\cdots i_{n}!}}=\frac{r!}{1!^{k_1}\cdots r!^{k_r}}.
	\]
	
	Now for any $(k_1,\dots,k_r)$, we identify that element of the corresponding set
	\[
	\{(i_{\pi(1)},...,i_{\pi(n)}): \pi \text{ permutations from }\{1,\dots,n\}\text{ to itself} \}
	\]
	for which
	\begin{align*}
	i_1&=\dots=i_{k_1}=1,\\
	i_{k_1+1}&=\dots=i_{k_1+k_2}=2,\\
	&\vdots\\
	i_{\sum_{l=1}^{r-1}k_l+1}&=\dots=i_{\sum_{l=1}^rk_l}=r,\\
	i_{\sum_{l=1}^rk_l+1}&=\dots=i_n=0.
	\end{align*}
	When summing over $(k_1,\dots,k_r)$, this will be the representative, evaluated within the expectation.
		As $V_n^2$ and $M_n$ are invariant under permutation, we convert the above expression \labelcref{eq.fubini} with $A_n:=\1_{\{M_n^{-1}\le|u|\}} V_n^{-(r-1)}$ into
	\begin{align}\label{eq.reordering}
	\lefteqn{\E_*\bigg[\1_{\{M_n^{-1}\le|u|\}} V_n^{-(r-1)}\Big|\ddur \prod \cos(u|X_j|)\Big|\bigg]}\quad\nonumber\\
	&=\E_*\bigg[A_n\cdot\Big|\sum_{i_{1}+i_{2}+\cdots +i_{n}=r}{r \choose i_{1},\ldots ,i_{n}}\prod\Big(\ddU{i_j}\cos(u|X_n|)\Big)\Big|\bigg]\\
	&=\E_*\bigg[A_n\cdot\Big|\sum_{*(k_\cdot,r,\cdot)} {n \choose k_{1},\ldots ,k_{r},(n-u(k_\cdot))} \frac{r!}{1!^{k_1}\cdots r!^{k_r}}\nonumber\\*
	&\qquad\qquad\quad\cdot\prod_{l=1}^r\Big(\prod_{a=1}^{k_l}\Big(\ddU{l}\cos\big(u|X_{\sum_{b=1}^{l-1}k_b+a}|\big)\Big)\Big)\prod_{l=u(k_\cdot)+1}^n\cos(u|X_l|) \Big|\bigg].\nonumber
	\end{align}
	In the next step we apply the triangle inequality and evaluate the derivates of these cosine-structures which result in multiplying by the respective $|X_j|$ and switching to the sine for uneven derivatives. Overall we get
	\begin{equation}\label{eq.triangle}
	\begin{split}
	&\E_*\bigg[A_n\cdot\Big|\sum_{*(k_\cdot,r,\cdot)} {n \choose k_{1},\ldots ,k_{r},(n-u(k_\cdot))} \frac{r!}{1!^{k_1}\cdots r!^{k_r}}\\
	&\qquad\qquad\cdot\prod_{l=1}^r\Big(\prod_{a=1}^{k_l}\ddU{l}\cos\big(u|X_{\sum_{b=1}^{l-1}k_b+a}|\big)\Big)\prod_{l=u(k_\cdot)+1}^n\cos(u|X_l|) \Big|\bigg]\\
	&\quad\le \E_*\bigg[A_n\,c(r)\sum_{*(k_\cdot,r,\cdot)} n^{u(k_\cdot)} \prod_{l=1}^r\Big(\prod_{a=1}^{k_l}\Big|\ddU{l}\cos\big(u|X_{\sum_{b=1}^{l-1}k_b+a}|\big)\Big|\Big)\prod_{l=u(k_\cdot)+1}^n\big|\cos(u|X_l|)\big|\bigg]
	\end{split}
	\end{equation}
	\begin{equation*}
	\begin{split}
	&\quad= c(r)\sum_{*(k_\cdot,r,\cdot)} n^{u(k_\cdot)} \, 
	\E_*\bigg[A_n
	\prod_{l=1}^{\lceil r/2\rceil} \Big(\prod_{a=1}^{k_{2l-1}}\big|X_{\sum_{b=1}^{(2l-1)-1}k_b+a}\big|^{2l-1}\Big|\sin\big(u\big|X_{\sum_{b=1}^{(2l-1)-1}k_b+a}\big|\big)\Big|\Big)\\
	&\quad\qquad\cdot\prod_{l=1}^{\lfloor r/2\rfloor} \Big(\prod_{a=1}^{k_{2l}}\big|X_{\sum_{b=1}^{2l-1}k_b+a}\big|^{2l}\Big|\cos\big(u\big|X_{\sum_{b=1}^{2l-1}k_b+a}\big|\big)\Big|\Big)
	\prod_{l=u(k_\cdot)+1}^n\big|\cos(u|X_l|)\big|\bigg].
	\end{split}
	\end{equation*}
	
	Since $|\sin|$ and $|\cos|$ are symmetric around 0, we can only work with positive $u$ and multiply by 2 (which is included in $c(r)$). Thus in \labelcref{eq.fubini,eq.reordering,eq.triangle}, we showed
	\begin{equation}\label{eq.E.expanded}
	\begin{split}
	&\E_*\bigg[\int_{B_n\le|u|< \nu_n} \Big|\ddur \prod \cos\big(uV_n^{-1}|X_j|\big)\Big|\du\bigg]
	\le c(r) \sum_{*(k_\cdot,r,\cdot)} n^{u(k_\cdot)} \int_{0\le u< \nu_n\zeta^{-1}}\\
	&\qquad \E_*\bigg[\1_{\{M_n^{-1}\le u\}} V_n^{-(r-1)}\prod_{l=1}^{\lceil r/2\rceil} \Big(\prod_{a=1}^{k_{2l-1}}\big|X_{\sum_{b=1}^{(2l-1)-1}k_b+a}\big|^{2l-1}\Big|\sin\big(u\big|X_{\sum_{b=1}^{(2l-1)-1}k_b+a}\big|\big)\Big|\Big)\\
	&\qquad\quad \cdot\prod_{l=1}^{\lfloor r/2\rfloor} \Big(\prod_{a=1}^{k_{2l}}\big|X_{\sum_{b=1}^{2l-1}k_b+a}\big|^{2l}\Big|\cos\big(u\big|X_{\sum_{b=1}^{2l-1}k_b+a}\big|\big)\Big|\Big)
	\prod_{l=u(k_\cdot)+1}^n\big|\cos(u|X_l|)\big| \bigg]\du\\
	&=:W_1+W_2,
	\end{split}
	\end{equation}
	where $W_1$ represents the integral over $u\in \big[0,\sqrt{n^{-1}\log(n)\, \eta}\big)$ and $W_2$ represents the integral over $u\in \big[\sqrt{n^{-1}\log(n)\, \eta}, \nu_n\zeta^{-1}\big)$ with
	\[
	\eta := 6\big(2r+{s}+4c\big).
	\]
	Note $1\le \eta\le c(s)$. Our choice of the split point $\sqrt{n^{-1}\log(n)\, \eta}$ is motivated by the bounds in \labelcref{eq.n^log(n),eq.W_2-finish}.

	First, we focus on a bound on $W_1$. Here, our basic idea is to simplify the complex expectation in $W_1$ to allow using \cref{l.E.mom}. The bounds we use have to be particularly tight for summands where $k_1$ and $k_2$ are large (i.e. $k_3=\dots=k_r=0$) as these summands result in the largest powers of $n$. We bound $|\cos(x)|\le 1$ and $\sin(x)\le x$ for $x\ge 0$. Thus,
	\begin{align*}
	W_1
	%
	&\le c(r) \sum_{*(k_\cdot,r,\cdot)} n^{u(k_\cdot)} \int_{0\le u< \sqrt{\frac{\log(n)\, \eta}{n}}} \\
	&\qquad\E_*\bigg[\1_{\{M_n^{-1}\le u\}} V_n^{-(r-1)}\prod_{l=1}^{\lceil r/2\rceil} \Big(\prod_{a=1}^{k_{2l-1}}\big|X_{\sum_{b=1}^{(2l-1)-1}k_b+a}\big|^{2l-1}\Big|\big(u\big|X_{\sum_{b=1}^{(2l-1)-1}k_b+a}\big|\big)\Big|\Big)\\
	&\qquad\quad\cdot\prod_{l=1}^{\lfloor r/2\rfloor} \Big(\prod_{a=1}^{k_{2l}}\big|X_{\sum_{b=1}^{2l-1}k_b+a}\big|^{2l} \Big)
	\bigg]\du.
	\end{align*}
	We replace the integral by the supremum of the integrand multiplied with the interval length, which bounds the expression above by
	\begin{align*}
	&c(r) \sum_{*(k_\cdot,r,\cdot)} n^{u(k_\cdot)} n^{-\frac12 \sum_{l=1}^{\lceil r/2\rceil}k_{2l-1}} \big(\log(n)\, \eta\big)^{\frac12 \sum_{l=1}^{\lceil r/2\rceil} k_{2l-1}} \sqrt{\tfrac{\log(n)\, \eta}{n}} \, \E_*\bigg[\1_{\big\{M_n^{-1}\le \sqrt{\frac{\log(n)\, \eta}{n}}\big\}} \\*
	&\qquad \cdot V_n^{-(r-1)}\prod_{l=1}^{\lceil r/2\rceil} \Big(\prod_{a=1}^{k_{2l-1}}\big|X_{\sum_{b=1}^{(2l-1)-1}k_b+a}\big|^{2l}\Big)
	\prod_{l=1}^{\lfloor r/2\rfloor} \Big(\prod_{a=1}^{k_{2l}}\big|X_{\sum_{b=1}^{2l-1}k_b+a}\big|^{2l} \Big)
	\bigg]\\
	&= c(r) \sum_{*(k_\cdot,r,\cdot)} \hspace{-0.8ex} n^{\frac12 (-1+\sum_{l=1}^{\lceil r/2\rceil}k_{2l-1}+2\sum_{l=1}^{\lfloor r/2\rfloor}k_{2l})} \big(\log(n)\, \eta\big)^{\frac12 (1+\sum_{l=1}^{\lceil r/2\rceil} k_{2l-1})} \hspace{0.2ex} \E_*\bigg[\1_{\big\{M_n\ge \sqrt{\frac{n}{\log(n)\, \eta}}\big\}}\\*
	&\qquad \cdot V_n^{-(r-1)}\prod_{l=1}^{\lceil r/2\rceil} \Big(\prod_{a=1}^{k_{2l-1}}\big|X_{\sum_{b=1}^{(2l-1)-1}k_b+a}\big|^{2l}\Big)
	\prod_{l=1}^{\lfloor r/2\rfloor} \Big(\prod_{a=1}^{k_{2l}}\big|X_{\sum_{b=1}^{2l-1}k_b+a}\big|^{2l} \Big)
	\bigg]\\*
	&=:W_3.
	\end{align*}
	Hence, $W_1\le W_3$. For $r=0$, $W_3$ reduces significantly. Using \labelcref{eq.Mn-an} and \cref{l.E.mom} ($l=1$, $a_1=2$) yields
	\begin{align}\label{eq.W_3-0}
	W_3
	&= c(0) \sqrt{\tfrac{\log(n)\, \eta}{n}} \; \E_*\Big[\1_{\big\{M_n\ge \sqrt{\frac{n}{\log(n)\, \eta}}\big\}} V_n \Big]\\
	& \le c(0) \sqrt{\tfrac{\log(n)\, \eta}{n}} \sqrt{\tfrac n2} \; \E_*\Big[\1_{\big\{M_n\ge \sqrt{\frac{n}{\log(n)\, \eta}}\big\}}\1_{\{V_n^2 \le \frac n2\}} \Big]\nonumber\\*
	&\quad+c(0) \sqrt{\tfrac{\log(n)\, \eta}{n}} \sqrt{\tfrac 2n} \; \E_*\Big[\1_{\big\{M_n\ge \sqrt{\frac{n}{\log(n)\, \eta}}\big\}} \1_{\{V_n^2 > \frac n2\}} V_n^2\Big]\nonumber\\
	& \le c(0) \sqrt{\log(n)\, \eta} \; \E_*\Big[\1_{\big\{M_n\ge \sqrt{\frac{n}{\log(n)\, \eta}}\big\}} \Big]
	+c(0) \sqrt{\log(n)\, \eta} \; \E_*\Big[\1_{\big\{M_n\ge \sqrt{\frac{n}{\log(n)\, \eta}}\big\}} X_1^2\Big]\nonumber\\
	& = \co \big( n^{-({s}-2)/2} \, \log(n)^{({s}+1)/2} \big).\nonumber
	\end{align}
	
	For $r\ge1$, we can simplify $W_3$ by using $|X_j| \le V_n$,
	\begin{align*}
	&c(r) \sum_{*(k_\cdot,r,\cdot)} \hspace{-0.9ex} n^{\frac12 (-1+\sum_{l=1}^{\lceil r/2\rceil}k_{2l-1}+2\sum_{l=1}^{\lfloor r/2\rfloor}k_{2l})} \big(\log(n)\, \eta\big)^{\frac12 (1+\sum_{l=1}^{\lceil r/2\rceil} k_{2l-1})} \, \E_*\bigg[\1_{\big\{M_n\ge \sqrt{\frac{n}{\log(n)\, \eta}}\big\}} \\*
	&\qquad \cdot V_n^{-(r-1)} \prod_{l=1}^{\lceil r/2\rceil} \Big(\prod_{a=1}^{k_{2l-1}}\big|X_{\sum_{b=1}^{(2l-1)-1}k_b+a}\big|^{2l}\Big)
	\prod_{l=1}^{\lfloor r/2\rfloor} \Big(\prod_{a=1}^{k_{2l}}\big|X_{\sum_{b=1}^{2l-1}k_b+a}\big|^{2l} \Big)
	\bigg]\\
	&\le c(r) \sum_{*(k_\cdot,r,\cdot)} \hspace{-0.9ex} n^{\frac12 (-1+\sum_{l=1}^{\lceil r/2\rceil}k_{2l-1}+2\sum_{l=1}^{\lfloor r/2\rfloor}k_{2l})} \big(\log(n)\, \eta\big)^{\frac12 (1+\sum_{l=1}^{\lceil r/2\rceil} k_{2l-1})} \, \E_*\bigg[\1_{\big\{M_n\ge \sqrt{\frac{n}{\log(n)\, \eta}}\big\}} \\*
	&\qquad \cdot V_n^{-(r-1)} \prod_{l=1}^{\lceil r/2\rceil} \Big(\prod_{a=1}^{k_{2l-1}}\big|X_{\sum_{b=1}^{(2l-1)-1}k_b+a}\big|^{2} V_n^{2l-2}\Big)
	\prod_{l=1}^{\lfloor r/2\rfloor} \Big(\prod_{a=1}^{k_{2l}}\big|X_{\sum_{b=1}^{2l-1}k_b+a}\big|^{2} V_n^{2l-2}\Big)
	\bigg].
	\end{align*}
	The power of $V_n$ in the respective summands is
	\begin{align*}
	-(r-1)+\sum_{l=1}^{\lceil r/2\rceil} k_{2l-1} (2l-2)+\sum_{l=1}^{\lfloor r/2\rfloor} k_{2l} (2l-2)
	&={1-\sum_{l=1}^{\lceil r/2\rceil}k_{2l-1}-2\sum_{l=1}^{\lfloor r/2\rfloor}k_{2l}}
	\end{align*}
	such that the above expression equals
	\begin{align*}
	& c(r) \sum_{*(k_\cdot,r,\cdot)} n^{\frac12 (-1+\sum_{l=1}^{\lceil r/2\rceil}k_{2l-1}+2\sum_{l=1}^{\lfloor r/2\rfloor}k_{2l})} \big(\log(n)\, \eta\big)^{\frac12 (1+\sum_{l=1}^{\lceil r/2\rceil} k_{2l-1})} \, \E_*\bigg[\1_{\big\{M_n\ge \sqrt{\frac{n}{\log(n)\, \eta}}\big\}} \\
	&\quad \cdot V_n^{1-\sum_{l=1}^{\lceil r/2\rceil}k_{2l-1}-2\sum_{l=1}^{\lfloor r/2\rfloor}k_{2l}}
	\prod_{l=1}^{\lceil r/2\rceil} \Big(\prod_{a=1}^{k_{2l-1}}\big|X_{\sum_{b=1}^{(2l-1)-1}k_b+a}\big|^{2} \Big)
	\prod_{l=1}^{\lfloor r/2\rfloor} \Big(\prod_{a=1}^{k_{2l}}\big|X_{\sum_{b=1}^{2l-1}k_b+a}\big|^{2} \Big)
	\bigg].
	\end{align*}
	As $r\ge1$, we know that $k_l>0$ for some $l=1,\dots,r$ such that $V_n$ has a non-positive power. Thus, using $V_n\ge M_n$, we can bound the above expression by
	\begin{align*}
	&c(r) \sum_{*(k_\cdot,r,\cdot)} n^{\frac12 (-1+\sum_{l=1}^{\lceil r/2\rceil}k_{2l-1}+2\sum_{l=1}^{\lfloor r/2\rfloor}k_{2l})} \big(\log(n)\, \eta\big)^{\frac12 (1+\sum_{l=1}^{\lceil r/2\rceil} k_{2l-1})} \\*
	&\qquad \cdot \Big(\sqrt{\tfrac{n}{\log(n)\, \eta}}\,\Big)^{1-\sum_{l=1}^{\lceil r/2\rceil}k_{2l-1}-2\sum_{l=1}^{\lfloor r/2\rfloor}k_{2l}} \, \E_*\bigg[\1_{\big\{M_n\ge \sqrt{\frac{n}{\log(n)\, \eta}}\big\}}\\*
	&\qquad \cdot \prod_{l=1}^{\lceil r/2\rceil} \Big(\prod_{a=1}^{k_{2l-1}}\big|X_{\sum_{b=1}^{(2l-1)-1}k_b+a}\big|^{2} \Big)
	\prod_{l=1}^{\lfloor r/2\rfloor} \Big(\prod_{a=1}^{k_{2l}}\big|X_{\sum_{b=1}^{2l-1}k_b+a}\big|^{2} \Big)
	\bigg]\\
	&= c(r) \sum_{*(k_\cdot,r,\cdot)} \big(\log(n) \, \eta\big)^{u(k_\cdot)} \\*
	&\qquad \cdot \E_*\bigg[\1_{\big\{M_n\ge \sqrt{\frac{n}{\log(n)\, \eta}}\big\}}
	\prod_{l=1}^{\lceil r/2\rceil} \Big(\prod_{a=1}^{k_{2l-1}}\big|X_{\sum_{b=1}^{(2l-1)-1}k_b+a}\big|^{2} \Big)
	\prod_{l=1}^{\lfloor r/2\rfloor} \Big(\prod_{a=1}^{k_{2l}}\big|X_{\sum_{b=1}^{2l-1}k_b+a}\big|^{2} \Big)
	\bigg]\\
	&=\co\big( n^{-({s}-2)/2} \, (\log n)^{({s}+2r)/2} \big)
	\end{align*}
	where we used \cref{l.E.mom} ($l=u(k_\cdot),a_1=\dots=a_{u(k_\cdot)}=2$) in the last step. Thus, for $r\ge1$,
	\begin{equation}\label{eq.W_3}
	W_1\le W_3=\co\big( n^{-({s}-2)/2} \, (\log n)^{({s}+2r)/2} \big).
	\end{equation}
	
	After having dealt with $W_1$, we direct our attention to bounding the other part of \labelcref{eq.E.expanded}, $W_2$. Due to the area of integration in $W_2$, $u$ now is larger and thus, $\cos(u|X_\cdot|)$ is smaller than in the integral in $W_1$. Therefore, our basic idea is to show that the expectation of the cosines are sufficiently small and to use their $(n-u(k_\cdot))$ times product $\big|\cos(u|X_{u(k_\cdot)+1}|)\big|\cdots \big|\cos(u|X_n|)\big|$ to bound $W_2$. In the process, we use $\sin(x)\le1$ and $\cos(x)\le1$ for all $x\in\R$ and $|X_j|/V_n\le1$ for all $j=1,\dots,n$. This leads to
	\begin{align}\label{eq.W_2}
	W_2&\le c(r) \sum_{*(k_\cdot,r,\cdot)} n^{u(k_\cdot)} \int_{\sqrt{\frac{\log(n)\, \eta}{n}}\le u< \nu_n\zeta^{-1}}
	\E_*\bigg[ V_n \prod_{l=u(k_\cdot)+1}^n\big|\cos(u|X_l|)\big| \bigg]\du\\
	&\le c(r) n^{r} \int_{\sqrt{\frac{\log(n)\, \eta}{n}}\le u< \nu_n\zeta^{-1}}
	\E_*\bigg[ V_n \prod_{l=r+1}^n\big|\cos(u|X_l|)\big| \bigg]\du.\nonumber
	\end{align}
	Using the indicator $\1_{\{\zeta < M_n< \tau_n\}}$ implies $V_n^2\le n M_n^2\le n \tau_n^2$ such that we can eliminate the $V_n$ and since the $X_j$ are i.i.d., by defining $\cU_n:=\big[\sqrt{n^{-1}\log(n)\, \eta}, \nu_n\zeta^{-1}\big]$, we get
	\begin{align}\label{eq.W_2-2}
	\lefteqn{c(r) n^r \int_{\sqrt{\frac{\log(n)\, \eta}{n}}\le u< \nu_n\zeta^{-1}} \E_*\bigg[ V_n \prod_{l=r+1}^n\big|\cos(u|X_l|)\big| \bigg]\du}\quad\nonumber\\*
	&\le c(r) n^{(2r+1)/2}\tau_n\,|\cU_n|\,\sup_{u\in\cU_n}\E_*\bigg[ \prod_{l=r+1}^n\big|\cos(u|X_l|)\big| \bigg]\\
	&\le c(r) n^{(2r+1)/2}\tau_n\,\nu_n\zeta^{-1}\,\sup_{u\in\cU_n}\prod_{l=r+1}^n\E\Big[\big|\cos(u|X_l|)\big|\Big]\nonumber\\
	&\le c(r) n^{(2r+1)/2}\tau_n\,\nu_n\,\sup_{u\in\cU_n}\E\Big[\big|\cos(u|X_1|)\big|\Big]^{n-r}.\nonumber
	\end{align}
	
	
	To achieve the aspired rate of convergence for $W_1$, we will show that
	\begin{align}\label{eq.cases}
	\sup_{u\in\cU_n}\E\Big[\big|\cos(u|X_1|)\big|\Big]\le1-\frac{\log(n)\, \eta / 6}{n}
	\end{align}
	for all $n\in\N$ sufficiently large.
	We conduct the proof by contradiction. Suppose \labelcref{eq.cases} does not hold, meaning for every $N\in\N$, there exists some $n\ge N$ for which there exists at least one $u_n\in\cU_n$ such that
	\begin{align}\label{eq.cases-supp}
	\E\Big[\big|\cos\big(u_n|X_1|\big)\big|\Big]>1-\frac{\log(n)\, \eta / 6}{n}.
	\end{align}
	Let $(n_k)$ be a subsequence of these $n$ such that
	\begin{align*}
	\E\Big[\big|\cos\big(u_{n_k}|X_1|\big)\big|\Big]>1-\frac{\log(n_k)\, \eta / 6}{n_k}.
	\end{align*}
	Hence, this sequence $(u_{n_k})$ either has a convergent subsequence $(u_{n_{k_l}})$ with $u_{n_{k_l}}\to u^*$ for some $u^*\in[0,\infty)$ or $u_{n_{k_l}}\to \infty$. Note $u_{n_{k_l}}>0$. For the ease of notation, we still denote the subsubsequence by $(u_n)$. In consequence, we assume that for all members of this sequence, \labelcref{eq.cases-supp} holds. We subsequently distinguish the cases $u^*=0,u^*\in(0,\infty)$ and $u_n\to\infty$.\\
	
	\underline{Case 1 ($u^*=0$):}
	
	Due to $\cos(x)\le 1-x^2/3$ for $x\in(0,\tfrac{\pi}{2})$, we have
	\begin{align*}
	\E\Big[\big|\cos\big(u_n|X_1|\big)\big|\Big]
	\le1-\E\bigg[\1_{\{0<u_n|X_1|< \pi/2\}} \frac{u_n^2|X_1|^2}{3}\bigg]
	=1-\frac{u_n^2\,\mu_{2,n}'}{3},
	\end{align*}
	where
	\[
	\mu_{2,n}'
	:=\E\Big[\1_{\{0<u_n|X_1|< \pi/2\}}|X_1|^2\Big] 
	= \E\Big[\1_{\{0<|X_1|< \frac{\pi}{2u_n}\}}|X_1|^2\Big]
	\too 1 
	\]
	by dominated convergence. Therefore, $\mu_{2,n}'\ge 1/2$ for $n$ sufficiently large. Since $u_n \in\cU_n$, $u_n\ge \sqrt{n^{-1}\log(n)\, \eta}$ such that
	\begin{equation*}
	\E\Big[\cos\big(u_n|X_1|\big)\Big]
	\le 1-\frac{u_n^2\,\mu_{2,n}'}{3}
	\le 1-\frac{\log(n)\, \eta / 6}{n}
	\end{equation*}
	for $n$ sufficiently large.\\
	
	\underline{Case 2 ($u^*\in(0,\infty)$\,):}
	
	Due to non-singularity, the distribution of $|X_1|$ cannot have mass in $\{j\,\pi/u^*\mid j\in\Z\}$ only. By \cref{l.ana1}, there exists $\beta$ with $0<\beta<1$ such that
	\begin{align*}
	\E\Big[\big|\cos(u^*|X_1|)\big|\Big]
	=\beta.
	\end{align*}
	So due to continuity of the cosine and dominated convergence, for all $\beta'$ with $\beta<\beta'<1$, there exists $N\in\N$ such that
	\[
	\E\Big[\big|\cos(u_n|X_1|)\big|\Big]
	\le\beta'
	\le 1-\frac{\log(n)\, \eta / 6}{n}
	\]
	holds for all $n\ge N$.\\
	
	\underline{Case 3 ($u_n\to \infty$):}
	
	By non-singularity, we know that $p_{ac}>0$ (see \cref{r.non-singularity}), so we first focus on the absolute continuous part $F_{ac}$ of the distribution of $X_1$.
	
	By the Riemann--Lebesgue lemma,
	\[
	\Big|\int_\R\exp(iu_nx)\dd F_{ac}(x)\Big|\too0
	\]
	for $u_n\to\infty$. Thus for the real part, we know
	\[
	\int_\R\cos(u_nx)\dd F_{ac}(x)\too0
	\]
	for $u_n\to\infty$. Since the cosine is symmetric, this yields
	\[
	\tfrac12 \int_\R\big(\cos(2u_n|x|)+1 \big)\dd F_{ac}(x)\too \tfrac12
	\]
	for $u_n\to\infty$. By Jensen's inequality and the trigonometric equation $\cos(2\alpha)= 2 \cos(\alpha)^2-1$ \cite[6.5.10]{Zwillinger},
	\[
	\Big(\int_\R|\cos(u_n|x|)|\dd F_{ac}(x)\Big)^2
	\le \int_\R\cos(u_n|x|)^2\dd F_{ac}(x)
	=\tfrac12 \int_\R\cos(2u_n|x|)+1 \dd F_{ac}(x).
	\]
	The right-hand side converges to $1/2$ for $u_n\to\infty$. Due to the Lebesgue decomposition of the distribution function \labelcref{eq.lebesgue-dec},
	\begin{align*}
	\lefteqn{\E\big[|\cos(u_n|X_{1}|)|\big]}\enspace\\
	&=p_{ac}\int_\R|\cos(u_n|x|)|\dd F_{ac}(x)+p_{d}\int_\R|\cos(u_n|x|)|\dd F_{d}(x)+p_{sc}\int_\R|\cos(u_n|x|)|\dd F_{sc}(x)\\
	&\le p_{ac}\sqrt{\tfrac12 \int_\R\cos(2u_n|x|)+1 \dd F_{ac}(x)}+p_{d}+p_{sc}\too \frac{1}{\sqrt{2}}p_{ac}+p_{d}+p_{sc}<1
	\end{align*}
	for $u_n\to\infty$. Hence, for any $\beta'$ with $p_{ac}/\sqrt{2}+p_{d}+p_{sc}<\beta'<1$, there consequently exists $N\in\N$ such that
	\[
	\E\big[|\cos(u_n|X_1|)|\big]
	\le\beta'
	\le 1-\frac{\log(n)\, \eta / 6}{n}
	\]
	holds for all $n\ge N$.
	
	As demonstrated, \labelcref{eq.cases-supp} leads to a contradiction in all three cases. Thus, we have shown \labelcref{eq.cases} by contradiction.
	
	Since $\log(1-x)\le -x$ for all $x<1$, for $n\ge2m$ and $n>\log(n)\, \eta / 6$, we can write
	\begin{equation}\label{eq.n^log(n)}
	\begin{split}
	\Big(1-\tfrac{\log(n)\, \eta / 6}{n}\Big)^{n-r}
	&\le \Big(1-\tfrac{\log(n)\, \eta / 6}{n}\Big)^{n/2}
	=\exp\Big(\tfrac n2\, \log\big(1-\tfrac{\log(n)\, \eta / 6}{n}\big)\Big)\\
	&\le \exp\Big(\tfrac n2 \big(-\tfrac{\log(n)\, \eta / 6}{n}\big)\Big)
	= n^{-\eta / 12}.
	\end{split}
	\end{equation}
	Recall $\nu_n,\tau_n\le n^{c}$ and $\eta = 6(2r+{s}+4c)$. Thus, by \labelcref{eq.W_2,eq.W_2-2,eq.cases,eq.n^log(n)}
	\begin{align}\label{eq.W_2-finish}
	W_2
	&\le c(r) n^{(2r+1)/2}\tau_n\,\nu_n\,\Big(1-\frac{\log(n)\, \eta / 6}{n}\Big)^{n-r}\\
	&\le c(r) n^{(2r+1)/2}\tau_n\,\nu_n\,n^{-(2r+{s}+4c) / 2}
	=\cO\big(n^{-({s}-1)/2}\big).\nonumber
	\end{align}
	
	We finish the proof by merging \labelcref{eq.E.expanded,eq.W_3-0,eq.W_3,eq.W_2-finish} to
	\begin{align*}
	\E_*\bigg[\int_{B_n\le|u|< \nu_n} \Big|\ddur \prod \cos\big(uV_n^{-1}|X_j|\big)\Big|\du\bigg]
	&\le W_1+W_2
	\le W_3+W_2\\
	&= \co\big( n^{-({s}-2)/2} \, (\log n)^{({s}+\max\{2r,1\})/2} \big).
	\end{align*}
\end{proof}

In the next lemma, we bound the expectation of $\tL_{k,n}$.

\begin{lemma}\label{l.E.tL}
Assume that $\E |X_1|^{{s}}<\infty$ for some ${s}\ge2$. Then for all $k> {s}$,
\begin{equation*}
\E[\tL_{k,n}]=\co\big(n^{-({s}-2)/2}\big).
\end{equation*}
\end{lemma}

The expectation of the remaining summand, an exponential of $B_n$, is bounded in the lemma below.

\begin{lemma}\label{l.exp.B_n}
Assume that $\E |X_1|^{{s}}<\infty$ for some ${s}\ge2$. Then for all $\eta>0$,
\begin{align*}
\E\left[ \exp\big(-\eta\,B_n^2\big)\right]
= \co\big( n^{-({s}-2)/2} \, (\log n)^{{s}/2} (\eta\wedge 1)^{-s/2} \big).
\end{align*}
\end{lemma}

The rather straightforward proofs of \cref{l.E.tL,l.exp.B_n} are deferred to \cref{app.proofs.distr}. The following remark is the last required component for the proof of the main theorem.

\begin{remark}\label{r.T_n}
	By the Cauchy--Schwarz inequality,
	\begin{equation*}
	T_n^2=\frac{S_n^2}{V_n^2}=\frac{(\sum X_j)^2}{\sum X_j^2}\le\frac{n\sum X_j^2}{\sum X_j^2}=n,
	\end{equation*}
	that is,  $|T_n|\le\sqrt{n}$.
\end{remark}

\begin{proof}[Proof of \cref{t.clt}]
	As $|T_n|\le\sqrt{n}$, $F_n(x)=0$ for $x< -\sqrt{n}$ and $F_n(x)=1$ for $x> \sqrt{n}$. Due to the form \labelcref{eq.Phi^Q-def} and the factor $\phi$ in all expansion terms ${Q}_{r}$,
	\[
	\sup_{x>\sqrt{n}}(1+|x|)^{m}\Big(\big|\Phi^{Q}_{m,n}(-x)\big|+\big|1-\Phi^{Q}_{m,n}(x)\big|\Big)=\co\big(n^{-({s}-2)/2}\big).
	\]
	Therefore, the left-hand side of \labelcref{eq.t-clt} is of order $\co\big(n^{-({s}-2)/2}\big)$ for the supremum restricted to $|x|> \sqrt{n}$. That is why we only have to consider the supremum over $|x|\le \sqrt{n}$ throughout this proof.
	
	In order to apply \cref{p.cond}, the left-hand side of \labelcref{eq.t-clt} has to be converted into a form related to the left-hand side of \labelcref{eq.pcond}. By $\E[\tF_n(x)]=F_n(x)$, Jensen's inequality and \cref{p.E-tP-Q}, we have
	\begin{align*}
	\lefteqn{\sup_{|x|\le\sqrt{n}} (1+|x|)^{m}|F_n(x)-\Phi^{Q}_{m,n}(x)|}\enspace\\
	&\le \sup_{|x|\le\sqrt{n}} (1+|x|)^{m}\left|\E\left[\tF_n(x)-\Phi^{\tP}_{m,n}(x)\right]\right| + \sup_{|x|\le\sqrt{n}} (1+|x|)^{m}\left|\E\left[\Phi^{\tP}_{m,n}(x)\right]-\Phi^{Q}_{m,n}(x)\right|\\
	&\le \E\left[\sup_{|x|\le\sqrt{n}} (1+|x|)^{m}\left|\tF_n(x)-\Phi^{\tP}_{m,n}(x)\right|\right] +\co\big(n^{-({s}-2)/2}\big).
	\end{align*}
	
	Due to $|x|\le \sqrt{n}$, the maximum order the factor $(1+|x|)^{m}$ can take is $n^{m/2}$. Next, $\Phi^{\tP}_{m,n}$ from \labelcref{eq.Phi^tP} has the property
	\[
	\sup_n \|\Phi^{\tP}_{m,n}\|_{\sup} \le c(m)
	\]
	due to \labelcref{eq.tP-bound}.
	Thus, the absolute difference $\|\tF_n-\Phi^{\tP}_{m,n}\|_{\sup}$ is bounded by a constant $c(m)$.
	By \labelcref{eq.Mn-an},
	\begin{align*}
	\P(M_n\ge n)
	=\co\big(n^{-({s}-1)}\big).
	\end{align*}
	Let $\zeta>0$ such that $\P(|X_1|\le \zeta)<1$. Then,
	\[
	\P (M_n \le \zeta) = \P(|X_1|\le \zeta)^n = \co\big(n^{-(m+{s}-2)/2}\big)
	\]
	in particular. Therefore, decomposing into $\{M_n \le \zeta \},\{\zeta < M_n< n\}$ and $\{M_n \ge n\}$,
	\begin{align*}
	\lefteqn{\E\left[\sup_{|x|\le\sqrt{n}}(1+|x|)^{m}\left|\tF_n(x)-\Phi^{\tP}_{m,n}(x)\right|\right]}\quad\\
	&= \E\left[ \1_{\{\zeta < M_n< n\}} \sup_{|x|\le\sqrt{n}}(1+|x|)^{m}\left|\tF_n(x)-\Phi^{\tP}_{m,n}(x)\right| \right]
	+\co\big(n^{-({s}-2)/2}\big).
	\end{align*}
	Recall $B_n=V_n/M_n$. Applying \cref{p.cond} ($M_n>\zeta$ guarantees $V_n^2>0$) yields
	\begin{align*}
	\lefteqn{\E\left[\1_{\{\zeta < M_n< n\}} \sup_{|x|\le\sqrt{n}}(1+|x|)^{m}\left|\tF_n(x)-\Phi^{\tP}_{m,n}(x)\right| \right]}\quad\\
	&\le c(m) \E\left[ \1_{\{\zeta < M_n< n\}} \Big(\tL_{m+1,n} 
	+ I_2
	+ \sum_{k=0}^{m} I_{k,2}
	+ e^{-B_n^2/4}\Big)
	\right],
	\end{align*}
	where $I_2=\int_{B_n\le|t|<\tL_{m+1,n}^{-1}}|t|^{-1}\big|\tphi_{T_n}(t)\big|\dt$ and
	\[I_{k,2}=\int_{B_n\le|t|<\tL_{m+1,n}^{-1}}|t|^{k-m-1}\Big|\ddtk \tphi_{T_n}(t)\Big|\dt.\]
	
	Regarding the first summand,
	\begin{equation*}
	\E[\tL_{m+1,n}] = \co\big(n^{-({s}-2)/2}\big)
	\end{equation*}
	by \cref{l.E.tL} with $m+1>{s}$.
	Concerning $I_2$ and $I_{k,2}$, we first bound the upper endpoint of the integration interval $\tL_{m+1,n}^{-1}$ by
	\begin{align*}
	\tL_{m+1,n}^{-1}
	=V_n^{m+1}\Big(\sum|X_j|^{m+1}\Big)^{-1}
	\le \big(n \,M_n^2\big)^{(m+1)/2} \big(M_n^{m+1}\big)^{-1}
	= n^{(m+1)/2}
	=:\nu_n.
	\end{align*}
	Note that $\nu_n$ is deterministic.
	Next, we can leave out negative powers of $|t|$ since $|t|\ge B_n\ge1$. Then we get from \cref{p.E.I} (with $\tau_n=n$)
	\begin{align*}
	\E[\1_{\{\zeta < M_n< n\}} I_2]
	&\le\E\bigg[\1_{\{\zeta < M_n< n\}} \int_{B_n\le|t|< \nu_n} \Big| \prod \cos\big(tV_n^{-1}|X_j|\big)\Big|\dt\bigg]\\
	&= \co\big( n^{-({s}-2)/2} \, (\log n)^{({s}+1)/2} \big)
	\end{align*}
	and
	\begin{align*}
	\E[\1_{\{\zeta < M_n< n\}} I_{k,2}]
	&\le\E\bigg[\1_{\{\zeta < M_n< n\}} \int_{B_n\le|t|< \nu_n} \Big|\ddtk \prod \cos\big(tV_n^{-1}|X_j|\big)\Big|\dt\bigg]\\
	&= \co\big( n^{-({s}-2)/2} \, (\log n)^{({s}+\max\{2k,1\})/2} \big)
	\end{align*}
	for all $k=0,\dots,m$.
	
	\cref{l.exp.B_n} yields
	\begin{align*}
	\E\left[ e^{-B_n^2/4}\right]
	= \co\big( n^{-({s}-2)/2} \, (\log n)^{{s}/2} \big).
	\end{align*}
	
	Therefore \labelcref{eq.t-clt} holds and \cref{t.clt} is proved.
\end{proof}

\section{Non-uniform bounds for Edgeworth expansions in local limit theorems for self-normalized sums}\label{ch.density}

This section is devoted to proving \cref{t.llt-red} and more general LLT-type results for self-normalized sums. The crucial obstacle as compared to the CLT-setting is that, conditional on $\cF_n$, the self-normalized sum $T_n$ does not have a density with respect to the Lebesgue measure. Instead it is continuous with respect to the counting measure. Therefore, the Fourier inversion formula does not apply any longer.

We introduce the following condition, which will be imposed in \cref{t.llt}. It is particularly satisfied if $X_1$ has a bounded density, see \cref{p.prop}. For instance, it has been successfully applied in \cite{JSZ04} to derive a saddle-point approximation for $T_n$.

\begin{condition}\label{c.density-cf}
	The joint characteristic function $\varphi_{(X_1,X_1^2)}$ of $(X_1,X_1^2)$ satisfies
	\begin{align*}
	\int_{\R^2} |\varphi_{(X_1,X_1^2)}(t)|^c \dt < \infty
	\end{align*}
	for some $c\ge1$.
\end{condition}

\cref{c.density-cf} is a smoothness condition for the corresponding density. By \cite[Theorem 19.1]{BR76normal}, it is equivalent to $(\sum_{j=1}^n X_j,\sum_{j=1}^n X_j^2)$
having bounded densities for $n$ sufficiently large.
In this case, $p_{d}=0$ (see \labelcref{eq.lebesgue-dec}) holds necessarily because any discrete part cannot not vanish when summing up i.i.d. random variables (see \cite[Proposition 6.1]{JSZ04}).
\cref{c.density-cf} implies Cram\'er condition (see \cite[p. 78]{Hal92edgeworth}), but it does not imply non-singularity and neither does non-singularity imply \cref{c.density-cf} (see \cite[p. 192]{BR76normal}).

Recall $\phi^{q}_{m,n}(x)=\phi(x)+\sum_{r=1}^{m-2}q_{r}(x)n^{-r/2}$ from \labelcref{eq.phi^q-def}.

\begin{theorem}\label{t.llt}
	Assume that $X_1$ is symmetric, the distribution of $X_1$ is non-singular, \cref{c.density-cf} is satisfied and $\E|X_1|^{2m}<\infty$ for some $m\in\N$, $m\ge3$. Then there exists $N\in\N$ such that for all $n\ge N$, the statistics $T_{n}$ have densities $f_{n}$ that satisfy 
	\begin{equation}\label{eq.t-llt}
	\sup_{x\in\R} \, (1+|x|)^{m}|f_n(x)-\phi^{q}_{m,n}(x)|= \co\big(n^{-(m-3)/2}\big).
	\end{equation}
	Moreover, if $\E |X_1|^{2m+2}<\infty$ for some $m\in\N, m\ge2$, the order of convergence reduces to $\co\big(n^{-(m-1)/2} \, \log n \big)$.
\end{theorem}

Note that \labelcref{eq.t-llt} implies that the densities $f_n$ are uniformly bounded for $n\ge N$. For the normalized sum $Z_n$, a similar non-uniform bound can be found in \cite{BCG11}. Together with the following \cref{p.prop}, \cref{t.llt} implies \cref{t.llt-red}. Its rather technical proof can be found in \cref{app.proofs.density}.

\begin{proposition}\label{p.prop}
	Assume that $X_1$ has a bounded density. Then \labelcref{c.density-cf} holds.
\end{proposition}

\paragraph*{Structure of the proof of \cref{t.llt}} As already mentioned, the starting point is again conditioning on $\cF_n$, leaving us with discrete $T_n$. Perturbing $T_n$ (see \cref{s.density.cond}) allows to derive an LLT in the conditional setting (\cref{p.dichte-p-c}) by bounding the characteristic functions (\cref{l.lemma4'}) and using the Fourier inversion formula. Afterwards, we return to the unconditional setting by taking expectations (\cref{p.dichte-p}). The necessary deconvolution leads to \cref{p.dichte-general}, which applies under \cref{c.density-cf} and is used in the proof of \cref{t.llt}.

\begin{remark}
	The one-dimensional analogue of \cref{c.density-cf}, namely $|\varphi_1|^c$ being integrable for some $c\ge 1$, is equivalent to $Z_n$ having a bounded density $f_{Z_N}$ for some $N\in\N$ (or equivalently for all sufficiently large $n$) by \cite[Theorem 19.1]{BR76normal}. This is necessary for any LLT for $Z_n$ to hold and therefore posed as a condition in all LLTs such as \cite[Theorem 19.2]{BR76normal}, \cite{BCG11} and \cite[Theorem 17, Chapter VII]{Pet75sums}.
	
\end{remark}

\subsection{The conditional setting}\label{s.density.cond}

Given $|X_1|,\dots,|X_n|$, there are (at most) $2^n$ possible values for $T_n$ as each $|X_j|$ can be included in $S_n$ with positive or negative sign. This is denoted by $\pm_i$, $i=1,\dots,2^n$. Due to the symmetry of $X_1$, the self-normalized sum $T_n$ (conditioned on $|X_1|,\dots,|X_n|$) takes each of these values with probability $2^{-n}$. Thus in the conditional setting, $T_n$ does not have a density but the probability mass function
\begin{equation*}
\tf_n(x)
:=f_{T_n\mid|X_1|,\dots,|X_n|}\big(x\mid |X_1| ,\dots,|X_n|\big)
=2^{-n}\sum_{i=1}^{2^n}\1_{\{V_n^{-1}{\sum \pm_i |X_j|}=x\}}\,.
\end{equation*}
for all $x\in\R$.

As $\tf_n$ is not a Lebesgue density, the Fourier inversion formula does not apply. Therefore, we perturb $T_n$ with an independent absolutely continuous random variable $\epsilon_n$ and examine
\[
T_n':=T_n+\epsilon_n.
\]
For the distribution of $\epsilon_n$, the normal distribution seems to be a suitable choice. In contrast to other distributions, the tails of its density and characteristic function decay in the same magnitude. As we operate with both of these functions, their simultaneous exponential decay will minimize the negative effects of perturbation. Furthermore, it does not hinder convergence to a normal distribution. 

Hence, we take $\epsilon_n\sim \cN(0,\beta_n)$ with density $\phi_{\beta_n}$, independent of $X_1,\dots,X_n$. Here, ${\beta_n\in(0,1)}$ for all $n\in\N$ and $\beta_n\to 0$ in polynomial order for $n\to\infty$. The precise sequence $(\beta_n)$ will be fixed later.

Next, we derive the distribution of $T_n'$ in the conditional setting. To this end, let $B\in\cB(\R)$ and $a_1,\dots,a_n\ge0$, then
\begin{align*}
&\P\big(T_n+\epsilon_n\in B \;|\; |X_1|=a_1,\dots,|X_n|=a_n\big)\\*
&\quad=\sum_{i=1}^{2^n}\E\Big[\1\Big\{T_n=\tfrac{\sum\pm_ia_j}{(\sum a_j^2)^{1/2}}\Big\}\1\Big\{\epsilon_n\in B-\tfrac{\sum\pm_ia_j}{(\sum a_j^2)^{1/2}}\Big\} \;\Big|\; |X_1|=a_1,\dots,|X_n|=a_n\Big].
\end{align*}
As $T_n$ and $\epsilon_n$ are independent given $|X_1|,\dots,|X_n|$, the expression above is equal to
\begin{align*}
\sum_{i=1}^{2^n} &\P\Big(T_n=\tfrac{\sum\pm_ia_j}{(\sum a_j^2)^{1/2}} \;\Big|\; |X_1|=a_1,\dots,|X_n|=a_n\Big) \P\Big(\epsilon_n\in B-\tfrac{\sum\pm_ia_j}{(\sum a_j^2)^{1/2}}\Big)\\
&\quad=\int_{B}2^{-n}\sum_{i=1}^{2^n} \phi_{\beta_n}\Big(y-\tfrac{\sum\pm_ia_j}{(\sum a_j^2)^{1/2}}\Big)\dy.
\end{align*}
Thus, the conditional distribution of $T_n'=T_n+\epsilon_n$ is a convolution of the two underlying distributions, which will be denoted by $\ast$. Its conditional density has the form
\begin{align*}
&\tf_n'(x)
:=f_{T_n'\mid|X_1|,\dots,|X_n|}(x\mid |X_1| ,\dots,|X_n|)
=\tf_n\ast \phi_{\beta_n}(x)
=2^{-n}\sum_{i=1}^{2^n} \phi_{\beta_n}\Big(x-\tfrac{\sum \pm_i |X_j|}{V_n}\Big).
\end{align*}

For the following steps until \labelcref{eq.E-fn'}, assume that the self-normalized sums $T_{n}$ have densities $f_{n}$ for $n\ge N$ for some $N\in\N$, and that we are in the regime $n\ge N$. This assumption is implied by the conditions of \cref{t.llt} (see \cref{s.density.uncond}).

In the unconditional setting, $T_n$ and $\epsilon_n$ are independent, too. Thus, the distribution of $T_n'=T_n+\epsilon_n$ has the density
\begin{equation}\label{eq.f_n'-def}
f_n'(x)=f_{T_n'}(x)=f_n\ast \phi_{\beta_n}(x)=\int_{-\infty}^{\infty}f_n(x-y) \phi_{\beta_n}(y)\dy.
\end{equation}
By Fubini's theorem, $\tf_n'$ and $f_n'$ are connected by
\begin{align*}
\int_B \E\big[\tf_n'(x)\big]\dx
&= \E\Big[\int_B \tf_n'(x)\dx\Big]
= \P\big(T_n' \in B \big)= \int_B f_n'(x)\dx
\end{align*}
for all $B\in\cB(\R)$ and thus $\E\big[\tf_n'(x)\big]= f_n'(x)$ $\lambda$-a.e. Since $\phi_{\beta_n}$ is continuous, $\tf_n'$ is also a continuous function. If $x_k\to x$ for $k\to\infty$, then
\[
\lim_{k\to\infty} \E\big[\tf_n'(x_k)\big]
=\E\big[\lim_{k\to\infty} \tf_n'(x_k)\big]
=\E\big[\tf_n'(x)\big]
\]
by dominated convergence which is applicable since $0\le\tf_n'(x)\le(2\pi\beta_n)^{-1/2}$ for all $x\in\R$ and $n\in\N$. Thus, $\E\big[\tf_n'\big]$ is also continuous. As $f_n\in L^1$ and $\phi_{\beta_n}\in L^\infty$, the convolution $f_n'=f_n\ast \phi_{\beta_n}$ also is a bounded and continuous function (see e.g. \cite[Theorem 15.8]{Schilling17}). We can therefore conclude that
\begin{equation}\label{eq.E-fn'}
\E\big[\tf_n'(x)\big]= f_n'(x) \qquad\text{ for all }x\in\R.
\end{equation}

For $T_n'$, we proceed as described above. First, we need a version of \cref{p.lemma4} that is suitable for $T_n'$. We do not need perturbed approximation functions because we will choose $\beta_n$ sufficiently small such that the approximation functions $\tp$ and $\tU$ are still appropriate. Recall the definitions $M_n=\max|X_j|$, $B_n=V_n/M_n$, $\tL_{k,n}=V_n^{-k}\sum|X_j|^k$ and
\[
\tphi_{\Phi^{\tP}_{m,n}}(t)
=e^{-t^2/2}\Big(1+\sum_{r=1}^{m-2} \tU_{r,n}(it)\Big),
\]
where
\[
\tU_{r,n}(it)
= \sum_{*(k_\cdot,r,\cdot)} (it)^{r+2u(k_\cdot)} \prod_{l=1}^{r} \frac{1}{k_l!} \Big(\frac{\tlambda_{l+2,n}}{(l+2)!}\Big)^{k_l}
\]
from \labelcref{eq.tU,eq.tphi}, and where the sum $\sum_{*(k_\cdot,r,\cdot)}$ is introduced in \cref{ch.pre}. The following lemma is a consequence of \cref{p.lemma4}. Its proof is deferred to \cref{app.proofs.density}.

\begin{lemma}\label{l.lemma4'}
	Assume that $X_1$ is symmetric, $V_n^2>0$ and $m\ge2$ is an integer. Then
	\begin{equation}\label{eq.lemma4'}
	\Big|\ddtk \big(\tphi_{T_n'}(t)-\tphi_{\Phi^{\tP}_{m,n}}(t)\big)\Big|
	\le c(m)\big(\tL_{m+1,n}+\beta_n\big)e^{-t^2/6}\big(1+|t|^{3m+1+k}\big)
	\end{equation}
	holds for $k=0,\dots, m$ in the interval $|t|<B_n$.
\end{lemma}

Next, we prove the non-uniform bound for our Edgeworth expansion of the density of the perturbed distribution in the conditional setting.
 Recall the definition $\phi^{\tp}_{m,n}(x)=\phi(x)+\sum_{r=1}^{m-2}\tp_{r,n}(x)$, where
\[
\tp_{r,n}(x)=\phi(x)\sum_{*(k_\cdot,r,\cdot)} H_{r+2u(k_\cdot)}(x) \prod_{l=1}^{r} \frac{1}{k_l!} \Big(\frac{\tlambda_{l+2,n}}{(l+2)!}\Big)^{k_l}
\]
from \labelcref{eq.tp,eq.phi^tp}.

\begin{proposition}\label{p.dichte-p-c}
	Assume that $X_1$ is symmetric, $V_n^2>0$ and $m\ge2$ is an integer. Then
	\begin{equation}\label{eq.p-dichte-p-c}
	\begin{split}
	&\sup_{x\in\R} \, (1+|x|)^{m}|\tf_n'(x)-\phi^{\tp}_{m,n}(x)|\\
	&\enspace\le c(m)\Big(\tL_{m+1,n}
	+ \beta_n
	+ J_{0,2}
	+ J_{m,2}
	+ \beta_n^{-(m+2)/2} n^{m-\log(n)/4}
	+ e^{-B_n^2/4}\Big),
	\end{split}
	\end{equation}
	where $J_{k,2}=\int_{B_n\le|t|< \nu_n} \big|\ddtk\tphi_{T_n'}(t)\big|\dt$ for $k=0,\dots,m$ and $\nu_n=\beta_n^{-1/2}\cdot \log(n)$.
\end{proposition}

\begin{proof}
	By conditional independence and \labelcref{eq.phi_n}, the characteristic function of $T_n'$ has the form
	\begin{equation}\label{eq.phi_n'}
	\begin{split}
	\tphi_{T_n'}(t)
	&= \Big( \prod\tphi_{X_j}(tV_n^{-1}) \Big) \cdot \exp(-\tfrac12\beta_nt^2)\\
	&= \Big( \prod \cos\big(tV_n^{-1}|X_j|\big) \Big) \cdot \exp(-\tfrac12\beta_nt^2).
	\end{split}
	\end{equation}
	For $k=0,\dots,m$, all derivatives $\ddtk\tphi_{T_n'}$ exist and are continuous. Furthermore, $\ddtk\tphi_{T_n'}$ is bounded by $c(m)n^k|t|^k\exp(-\beta_nt^2/2)\in L^1$ and thus in $L^1$. The Fourier transform of our approximating function $\tphi_{\Phi^{\tP}_{m,n}}(t)$ also possesses these properties by \labelcref{eq.ddtl-tphi}. This means that the conditions of the Fourier inversion theorem (see e.g. \cite[Theorem 4.1 (iv,v)]{BR76normal}) are satisfied. This yields
	\begin{equation*}
	\sup_{x\in\R} |x|^{k}|\tf_n'(x)-\phi^{\tp}_{m,n}(x)|
	\le \int_{-\infty}^{\infty} \Big|\ddtk\big(\tphi_{T_n'}(t)-\tphi_{\Phi^{\tP}_{m,n}}(t)\big)\Big|\dt
	\end{equation*}
	for $k=0,\dots,m$. As before, we split up the integral at $B_n$ and additionally at $\nu_n$. This leads to
	\begin{equation}\label{eq.dichte-int-split}
	\begin{split}
	\lefteqn{\sup_{x\in\R} |x|^{k}|\tf_n'(x)-\phi^{\tp}_{m,n}(x)|}\quad\\
	&\le \underbrace{\int_{|t|<B_n} \Big|\ddtk\big(\tphi_{T_n'}(t)-\tphi_{\Phi^{\tP}_{m,n}}(t)\big)\Big|\dt}_{=J_{k,1}}
	\, + \underbrace{\int_{B_n\le|t|< \nu_n} \Big|\ddtk\tphi_{T_n'}(t)\Big|\dt}_{=J_{k,2}}\\
	&\quad + \underbrace{\int_{|t|\ge \nu_n} \Big|\ddtk\tphi_{T_n'}(t)\Big|\dt}_{=J_{k,3}}
	\, + \underbrace{\int_{|t|\ge B_n} \Big|\ddtk\tphi_{\Phi^{\tP}_{m,n}}(t)\Big|\dt}_{=J_{k,4}}.
	\end{split}
	\end{equation}
	
	From \cref{l.lemma4'} we deduce
	\[
	J_{k,1}
	\le \int_{|t|<B_n} c(m)(\tL_{m+1,n}+\beta_n\big)e^{-t^2/6}\big(1+|t|^{3m+1+k}\big) \dt
	\le c(m)(\tL_{m+1,n}+\beta_n\big),
	\]
	for all $k=0,\dots,m$.
	
	
	Next, by the Leibniz rule,
	\begin{equation}\label{eq.J_k3}
	\begin{split}
	J_{k,3}
	&=\int_{|t|\ge \nu_n} \Big|\ddtk \Big( \prod \cos\big(tV_n^{-1}|X_j|\big) \Big) \cdot \exp(-\tfrac12\beta_nt^2)\Big|\dt\\
	&\le\int_{|t|\ge \nu_n} \sum_{l=0}^{k}\binom{k}{l}\Big|\ddtl \prod \cos\big(tV_n^{-1}|X_j|\big)\Big| \Big|\ddT{k-l}\exp(-\tfrac12\beta_nt^2)\Big|\dt.
	\end{split}
	\end{equation}
	As all derivatives of $\cos(x)$ are bounded in absolute value by 1,
	\begin{align*}
	\Big|\ddtl \prod \cos(tV_n^{-1}|X_j|)\Big|
	&\le\sum _{i_{1}+i_{2}+\cdots +i_{n}=l} {l \choose i_{1},\ldots ,i_{n}} \prod\Big|\ddT{i_j}\cos(tV_n^{-1}|X_j|)\Big|\\
	&\le\sum _{i_{1}+i_{2}+\cdots +i_{n}=l} {l \choose i_{1},\ldots ,i_{n}} \prod\Big|(V_n^{-1}|X_j|)^{i_j}\Big|\\
	&\le\sum _{i_{1}+i_{2}+\cdots +i_{n}=l} {l \choose i_{1},\ldots ,i_{n}}=n^l.
	\end{align*}
	Here, the sum extends over all $n$-tuples $(i_1,...,i_n)$ of non-negative integers with $\sum_{j=1}^ni_j=l$ and 
	\[
	{l \choose i_{1},\ldots ,i_{n}}={\frac {l!}{i_{1}!\,i_{2}!\cdots i_{n}!}}.
	\]
	Since $\beta_n<1$ and $|t|\ge \nu_n \ge1$, the $l$-th derivative of the exponential term in \labelcref{eq.J_k3} is bounded by $c(k)|t|^k\exp(-\beta_nt^2/2)$ for $l=0,\dots,k$.
	
	Using \cref{l.int-t-exp}, $J_{k,3}$ can be bounded by
	\begin{align*}
	\lefteqn{\int_{|t|\ge \nu_n} \sum_{l=0}^{k}\binom{k}{l} n^l \, c(k) |t|^{k}\exp(-\beta_nt^2/2)\dt}\quad\\
	&\le c(k) n^k \beta_n^{-(k+2)/2} \exp(-\beta_n \nu_n^2/4)\\
	&\le c(m) n^m \beta_n^{-(m+2)/2} \exp\big(-(\log n)^{2}/4\big)\\
	&= c(m) \beta_n^{-(m+2)/2} n^{m- \log(n)/4}
	\end{align*}
	for all $k=0,\dots,m$.
	
	
	As in \labelcref{eq.Ik3}, by \labelcref{eq.ddtl-tphi} and since $|t|\ge B_n\ge 1$,
	\[
	J_{k,4}
	=\int_{|t|\ge B_n} \Big|\ddtk\tphi_{\Phi^{\tP}_{m,n}}(t)\Big|\dt
	\le c(m)e^{-B_n^2/4}
	\]
	for all $k=0,\dots,m$, where we use \cref{l.int-t-exp} (with $\beta=1$, $\nu_n=B_n$) in the last step.
	
	Setting $k=0$ and $k=m$ in \labelcref{eq.dichte-int-split} now yields \labelcref{eq.p-dichte-p-c}.
\end{proof}

\subsection{The unconditional setting}\label{s.density.uncond}

In the next step towards \cref{t.llt}, we apply the expectation to the result of \cref{p.dichte-p-c} and prove the non-uniform bound for our Edgeworth expansion of the density of the perturbed distribution in the unconditional setting. For this purpose, we introduce the following condition \cref{c.density-ex}.
Together with a moment condition on $X_1$, \cref{c.density-cf} implies \cref{c.density-ex} by \cref{t.llt}. 

\begin{condition}\label{c.density-ex}
	There exists $N\in\N$ such that the statistics $T_{n}$ have densities $f_{n}$
	for
	all $n\ge N$.
\end{condition}

The well-known argument, which allows to conclude that the existence of a density $f_N$ for some $N\in\N$ implies the existence of densities $f_n$ for any $n\ge N$ which holds for the CLT-statistic $Z_n$ does not apply here.
%
%
%
%
For both statistics $T_n$ and $Z_n$, conditions of this type do not imply that $X_1$ is absolutely continuous or even has a non-zero absolutely continuous component. For a counterexample with $X_1,X_2$ having singular continuous distribution and $X_1+X_2$ being absolutely continuous with bounded density, see \cite[p. 17]{Luk72}.

Note that in contrast to the distribution functions in the proof of \cref{t.clt}, the densities $f_n$ do not exist a priori. Therefore, the requirement \cref{c.density-ex} is crucial for applying \labelcref{eq.E-fn'}.

\begin{proposition}\label{p.dichte-p}
	Assume that $X_1$ is symmetric, the distribution of $X_1$ is non-singular and $\E|X_1|^{s}<\infty$ for some ${s}\ge2$. Additionally assume that \cref{c.density-ex} is satisfied and ${\beta_n=\co\big(n^{-({s}-2)/2}\big)}$ in polynomial order. Then for $m=\lfloor {s} \rfloor$,
	\begin{equation}\label{eq.p-dichte-p}
	\sup_{x\in\R} \, (1+|x|)^{m}|f_n'(x)-\phi^{q}_{m,n}(x)|= \co\big( n^{-({s}-2)/2} \, (\log n)^{({s}+2m)/2} \big).
	\end{equation}
\end{proposition}


\begin{proof}
	Let $n\ge N$ with $N$ from \cref{c.density-ex}. Recall $|T_n|\le\sqrt{n}$ from \cref{r.T_n} and thus, $f_n(x)=0$ for all $|x|>\sqrt{n}$. 
	Let $x>2\sqrt{n}$, then
	\begin{align*}
	(1+|x|)^{m}f_n'(x)
	&=(1+x)^{m}\int_{-\sqrt{n}}^{\sqrt{n}} f_n(y) \phi_{\beta_n}(x-y)\dy\\
	&\le(1+x)^{m} \phi_{\beta_n}(x-\sqrt{n} \, ) \int_{-\sqrt{n}}^{\sqrt{n}} f_n(y) \dy\\
	&=(1+x)^{m} (2\pi\beta_n)^{-1/2} \exp\big(-\tfrac12\beta_n^{-1}(x-\sqrt{n} \, )^2\big).
	\end{align*}
	Due to
	\begin{align*}
	\lefteqn{\ddx (1+x)^{m} \exp\big(-\tfrac12\beta_n^{-1}(x-\sqrt{n} \, )^2\big)}\quad\\
	&= \Big(m-(1+x) \beta_n^{-1}(x-\sqrt{n} \, )\Big) (1+x)^{m-1} \exp\big(-\tfrac12\beta_n^{-1}(x-\sqrt{n} \, )^2\big)\\
	&\le \Big(m-(1+2\sqrt{n}\,) \beta_n^{-1} \sqrt{n}\Big) (1+x)^{m-1} \exp\big(-\tfrac12\beta_n^{-1}(x-\sqrt{n} \, )^2\big)< 0,
	\end{align*}
	$(1+x)^{m} \exp\big(-\beta_n^{-1}(x-\sqrt{n} \, )^2/2\big)$ is monotonically decreasing for $x>2\sqrt{n}$ (and $n>m$). Therefore,
	\begin{equation}\label{eq.fn'-bound}
	\begin{split}
	\sup_{x > 2\sqrt{n}} (1+|x|)^{m}f_n'(x)
	&\le \sup_{x > 2\sqrt{n}} (1+x)^{m} (2\pi\beta_n)^{-1/2} \exp\big(-\tfrac12\beta_n^{-1}(x-\sqrt{n} \, )^2\big)\\
	&= (1+2\sqrt{n} \, )^{m} (2\pi\beta_n)^{-1/2} \exp\big(-\tfrac12\beta_n^{-1}n\big)= \co\big(n^{-({s}-2)/2}\big)
	\end{split}
	\end{equation}
	in particular.
	This bound also holds for negative $x$ as $f_n$ and $\phi_{\beta_n}$ are symmetric. Additionally, due to the form \labelcref{eq.phi^q-def} and the factor $\phi$ in all expansion terms $q_r$,
	\begin{equation}\label{eq.phi^q-bound}
	\sup_{|x|>2\sqrt{n}} (1+|x|)^{m} \big|\phi^{q}_{m,n}(x)\big|=\co\big(n^{-({s}-2)/2}\big).
	\end{equation}
	
	Therefore, the left-hand side of \labelcref{eq.p-dichte-p} is of order $\co\big(n^{-({s}-2)/2}\big)$ for the supremum restricted to $|x|> 2\sqrt{n}$. That is why we only have to consider the supremum over $|x|\le 2\sqrt{n}$ throughout the remainder of this proof.
	
	In order to apply \cref{p.dichte-p-c}, the left-hand side of \labelcref{eq.p-dichte-p} has to be converted into a form related to the left-hand side of \labelcref{eq.p-dichte-p-c}. By \labelcref{eq.E-fn'}, Jensen's inequality and \cref{p.E-tP-Q}, we have
	\begin{align*}
	\lefteqn{\sup_{|x|\le2\sqrt{n}} (1+|x|)^{m}|f_n'(x)-\phi^{q}_{m,n}(x)|}\;\\
	&\le \sup_{|x|\le2\sqrt{n}} (1+|x|)^{m}\left|\E\left[\tf_n'(x)-\phi^{\tp}_{m,n}(x)\right]\right| + \sup_{|x|\le2\sqrt{n}} (1+|x|)^{m}\left|\E\left[\phi^{\tp}_{m,n}(x)\right]-\phi^{q}_{m,n}(x)\right|\\
	&\le \E\left[\sup_{|x|\le2\sqrt{n}} (1+|x|)^{m}\left|\tf_n'(x)-\phi^{\tp}_{m,n}(x)\right|\right] +\co\big(n^{-({s}-2)/2}\big).
	\end{align*}
	
	Due to the form of $\tf_n'$ and uniform boundedness of $\phi^{\tp}_{m,n}$ (see \labelcref{eq.phi^tp,eq.tp-bound}),
	\begin{align*}
	\lefteqn{\sup_{|x|\le2\sqrt{n}}(1+|x|)^{m}\left|\tf_n'(x)-\phi^{\tp}_{m,n}(x)\right|}\quad\\*
	&=\sup_{|x|\le2\sqrt{n}}(1+|x|)^{m}\left|2^{-n}\sum_{i=1}^{2^n} \phi_{\beta_n}\Big(x-\tfrac{\sum \pm_i |X_j|}{V_n}\Big)-\phi^{\tp}_{m,n}(x)\right|\\
	&\le \sup_{|x|\le2\sqrt{n}} (2\sqrt{n} \, )^{m}\left((2\pi\beta_n)^{-1/2}+c(m)\right)= \cO\big(n^{m/2}\beta_n^{-1/2}\big).
	\end{align*}
	By \labelcref{eq.Mn-an} for $\tau_n:=n^{(m+s)/(2s)}\beta_n^{-1/(2s)}$,
	\begin{align*}
	\P(M_n\ge \tau_n)
	=\co\big(n^{-(m+s-2)/2}\beta_n^{1/2}\big).
	\end{align*}
	Let $\zeta>0$ such that $\P(|X_1|\le \zeta)<1$. Then,
	\[
	n^{m/2}\beta_n^{-1/2} \cdot \P (M_n \le \zeta) 
	= n^{m/2}\beta_n^{-1/2} \cdot \P(|X_1|\le \zeta)^n 
	= \co\big(n^{-({s}-2)/2}\big)
	\]
	in particular. Therefore, decomposing into $\{M_n \le \zeta \},\{\zeta < M_n< \tau_n\}$ and $\{M_n \ge \tau_n\}$,
	\begin{align*}
	&\E\left[\sup_{|x|\le2\sqrt{n}} (1+|x|)^{m}\left|\tf_n'(x)-\phi^{\tp}_{m,n}(x)\right|\right]\\*
	&\quad= \E\left[\1_{\{\zeta < M_n< \tau_n\}} \sup_{|x|\le2\sqrt{n}} (1+|x|)^{m}\left|\tf_n'(x)-\phi^{\tp}_{m,n}(x)\right| \right]
	+\co\big(n^{-({s}-2)/2}\big).
	\end{align*}
	
	Recall $B_n=V_n/M_n$. Applying \cref{p.dichte-p-c} ($M_n>\zeta$ guarantees $V_n^2>0$) yields
	\begin{align*}
	&\E\left[\1_{\{\zeta < M_n< \tau_n\}} \sup_{|x|\le2\sqrt{n}} (1+|x|)^{m}\left|\tf_n'(x)-\phi^{\tp}_{m,n}(x)\right| \right]
	\le c(m) \E\Big[\1_{\{\zeta < M_n< \tau_n\}} \\*
	&\qquad\cdot \Big(\tL_{m+1,n}
	+ \beta_n
	+ J_{0,2}
	+ J_{m,2}
	+ \beta_n^{-(m+2)/2} n^{m-\log(n)/4}
	+ e^{-B_n^2/4}\Big)
	\Big],
	\end{align*}
	where $J_{k,2}=\int_{B_n\le|t|< \nu_n} \big|\ddtk\tphi_{T_n'}(t)\big|\dt$ for $k=0,\dots,m$ and $\nu_n=\beta_n^{-1/2}\cdot \log(n)$.
	
	For the first summand,
	\begin{equation*}
	\E[\tL_{m+1,n}] = \co\big(n^{-({s}-2)/2}\big).
	\end{equation*}
	by \cref{l.E.tL} with $m+1>{s}$.
	Clearly, for the second summand $\beta_n=\co\big(n^{-({s}-2)/2}\big)$ and for the fifth summand $\beta_n^{-(m+2)/2} n^{m-\log(n)/4}=\co\big(n^{-({s}-2)/2}\big)$ as $\beta_n$ decreases in polynomial order and $n^{-\log(n)}$ is smaller than $n^{-k}$ for every $k>0$ and $n$ sufficiently large.
	
	Now we examine the expectation of $J_{k,2}$. Notice $\beta_n|t|< \beta_n\nu_n=\sqrt{\beta_n}\cdot \log(n)\le 1$ in the domain of the integral for $n$ sufficiently large.
	For $B_n\le t< \nu_n$ and $a\ge0$, this yields
	\begin{align*}
	\Big|\ddT{a}\exp(-\tfrac12 \beta_n t^2)\Big|
	\le c(a) \exp(-\tfrac12 \beta_n t^2)
	\le c(a).
	\end{align*}
	Thus by \labelcref{eq.phi_n'} and the Leibniz formula,
	\begin{align*}
	J_{k,2}
	&=\int_{B_n\le|t|< \nu_n} \Big|\sum_{l=0}^k\binom{k}{l}\Big(\ddtl \prod \cos\big(tV_n^{-1}|X_j|\big)\Big) \Big(\ddT{k-l}\exp(-\tfrac12\beta_nt^2)\Big)\Big|\dt\\
	&\le c(k)\sum_{l=0}^{k}\int_{B_n\le|t|< \nu_n} \Big|\ddtl \prod \cos\big(tV_n^{-1}|X_j|\big)\Big| \Big|\ddT{k-l}\exp(-\tfrac12\beta_nt^2)\Big|\dt\\
	&\le c(k)\sum_{l=0}^{k}\int_{B_n\le|t|< \nu_n} \Big|\ddtl \prod \cos\big(tV_n^{-1}|X_j|\big)\Big| \dt.
	\end{align*}	
	As $\nu_n=\beta_n^{-1/2}\cdot \log(n)$ and $\tau_n=n^{(m+s)/(2s)}\beta_n^{-1/(2s)}$ with $\beta_n=\co\big(n^{-({s}-2)/2}\big)$ grow in polynomial order, by \cref{p.E.I} we get
	\begin{align*}
	\E\big[\1_{\{\zeta < M_n< \tau_n\}} J_{k,2}\big]
	&\le c(k)\sum_{l=0}^{k}\E\bigg[\1_{\{\zeta < M_n< \tau_n\}} \int_{B_n\le|t|< \nu_n} \Big|\ddtl \prod \cos\big(tV_n^{-1}|X_j|\big)\Big|\dt \bigg]\\
	&= \co\big( n^{-({s}-2)/2} \, (\log n)^{({s}+\max\{2k,1\})/2} \big)
	\end{align*}
	for $k=0,\dots,m$.	
	Regarding the sixth summand, \cref{l.exp.B_n} yields
	\begin{align*}
	\E\left[ e^{-B_n^2/4}\right]
	= \co\big( n^{-({s}-2)/2} \, (\log n)^{{s}/2} \big).
	\end{align*}
	
	Therefore \labelcref{eq.p-dichte-p} holds and \cref{p.dichte-p} is proved.
\end{proof}

\cref{p.dichte-p} provides a non-uniform bound for $f_n'$. We will now deconvolute the densities to achieve a similar bound for $f_n$.
The H\"older type condition $(iii)$ of the following result seems hard to verify, but it will turn out to be satisfied by the conditions of \cref{t.llt}.

\begin{proposition}\label{p.dichte-general}
	Assume that the following conditions are satisfied:
	\begin{enumerate}[$(i)$]
		\item $X_1$ is symmetric, the distribution of $X_1$ is non-singular, and $\E|X_1|^{s}<\infty$ for some ${s}\ge2$.
		\item \cref{c.density-ex} is satisfied and there exists $N_1\in\N$ such that there exist versions of $f_{n}$
		that are bounded uniformly over all $n\ge N_1$.
		\item There exist $a\in(0,1],b,\alpha>0$ (that are allowed to depend on $s$), $N_2\in\N$ and a sequence $(r_n)$ with $r_n\to 0$ such that for some $m\in\N$ with $2\le m \le {s}$,
		\begin{equation}\label{eq.density-Hoelder-1}
		(1+|x|)^{m}|f_n(x)-f_n(x-h)|
		\le \alpha \, (1+|x|)^{m} |h|^{a} n^{b} + r_n
		\end{equation}
		holds for all $n\ge N_2$, $|h|<1$ and for all $|x|\le2\sqrt{n}$.
	\end{enumerate}
	Then,
	\begin{equation}\label{eq.p-dichte-general}
	\begin{split}
	&\sup_{x\in\R} \, (1+|x|)^{m}|f_n(x)-\phi^{q}_{m,n}(x)|\\*
	&\quad= \co\big( n^{-({s}-2)/2} \, (\log n)^{({s}+2\lfloor {s} \rfloor)/2} \big) 
	+ \cO\big( n^{-\lceil\frac{m-1}2\rceil} \big)
	+ r_n.
	\end{split}
	\end{equation}
\end{proposition}

\begin{proof}
	Let $n\ge N\vee N_1\vee N_2$ with $N$ from \cref{c.density-ex}. In this proof we decompose the left-hand side of \labelcref{eq.p-dichte-general} such that
	\begin{equation}\label{eq.f-split}
	\begin{split}
	\lefteqn{\sup_{x\in\R} \, (1+|x|)^{m}|f_n(x)-\phi^{q}_{m,n}(x)|}\quad\\*
	&\le\sup_{x\in\R} \, (1+|x|)^{m}|f_n(x)-f_n'(x)|
	+\sup_{x\in\R} \, (1+|x|)^{\lfloor {s} \rfloor}|f_n'(x)-\phi^{q}_{\lfloor s \rfloor,n}(x)|\\
	&\quad +\sup_{x\in\R} \, (1+|x|)^{m}|\phi^{q}_{\lfloor {s} \rfloor,n}(x)-\phi^{q}_{m,n}(x)|.
	\end{split}
	\end{equation}
	
	Now we fix $\beta_n= \co\big( n^{-(2b+m+{s}-2)/a} \big)$ (polynomial in $n$), which satisfies the assumptions of \cref{p.dichte-p}. Thus, the second summand of the right-hand side in \labelcref{eq.f-split} is of order $\co\big( n^{-({s}-2)/2} \, (\log n)^{({s}+2\lfloor {s} \rfloor)/2} \big)$.
	
	By \labelcref{eq.phi^q-def}, the third summand equals
	\begin{equation*}
	\sup_{x\in\R} \, (1+|x|)^{m}\Big|\sum_{r=m-1}^{\lfloor {s} \rfloor-2}q_{r}(x)n^{-r/2}\Big|
	\le c({s}) n^{-\lceil\frac{m-1}2\rceil}
	\end{equation*}
	for all $x\in\R$. Similar to \labelcref{eq.phi^q-bound}, the bound is due to the factor $\phi$ in all expansion terms $q_{r}$ and the fact $q_r=0$ for uneven $r$.
	
	As a last step, we need to bound $\sup_{x\in\R} (1+|x|)^{m}|f_n(x)-f_n'(x)|$. By \cref{r.T_n}, $|T_n|\le\sqrt{n}$ and thus $f_n(x)=0$ for all $|x|>\sqrt{n}$. Then, \labelcref{eq.fn'-bound} implies
	\begin{align*}
	\sup_{|x|>2\sqrt{n}}(1+|x|)^{m}|f_n(x)-f_n'(x)|
	=\sup_{|x|>2\sqrt{n}}(1+|x|)^{m}f_n'(x)
	\le \co\big(n^{-({s}-2)/2}\big).
	\end{align*}
	Hence, from now on we only have to consider the supremum over $|x|\le 2\sqrt{n}$. For ${\epsilon_n\sim \cN(0,\beta_n)}$ with density $\phi_{\beta_n}$, we get
	\begin{align}\label{eq.fn-fn'}
	\lefteqn{\sup_{|x|\le 2\sqrt{n}}(1+|x|)^{m}|f_n(x)-f_n'(x)|}\quad\nonumber\\
	&=\sup_{|x|\le 2\sqrt{n}}(1+|x|)^{m}\Big|\int_{-\infty}^{\infty} \big(f_n(x)-f_n(x-y)\big) \phi_{\beta_n}(y)\dy\Big|\\
	&=\sup_{|x|\le 2\sqrt{n}}(1+|x|)^{m}\Big|\E\big[f_n(x)-f_n(x-\epsilon_n)\big]\Big|\nonumber\\
	&\le\E\Big[\sup_{|x|\le 2\sqrt{n}}(1+|x|)^{m}\big|f_n(x)-f_n(x-\epsilon_n)\big|\Big].\nonumber
	\end{align}
	By \labelcref{eq.density-Hoelder-1}, we get for $|\epsilon_n|<1$
	\begin{align*}
	\sup_{|x|\le 2\sqrt{n}}(1+|x|)^{m}|f_n(x)-f_n(x-\epsilon_n)|
	&\le \sup_{|x|\le 2\sqrt{n}} \alpha \, (1+|x|)^{m} |\epsilon_n|^{a} n^{b} + r_n\\
	&= c(s) \sup_{|x|\le 2\sqrt{n}} (1+|x|)^{m} |\epsilon_n|^{a} n^{b} + r_n.
	\end{align*}
	By uniform boundedness of $f_n$ for $|\epsilon_n|\ge1$,
	\begin{align*}
	\sup_{|x|\le 2\sqrt{n}}(1+|x|)^{m} \big|f_n(x)-f_n(x-\epsilon_n)\big|
	&\le \sup_{|x|\le 2\sqrt{n}}(1+|x|)^{m} \cdot 2\sup_{n\ge N_1}\|f_n\|_{\sup}\\
	&\le c(s) \sup_{|x|\le 2\sqrt{n}}(1+|x|)^{m} |\epsilon_n|^{a} n^{b} + r_n.
	\end{align*}
	Therefore, noting that $\varepsilon_n\sim\mathcal{N}(0,\beta_n)$, we deduce from \labelcref{eq.fn-fn'}
	\begin{align*}
	\sup_{|x|\le 2\sqrt{n}}(1+|x|)^{m}|f_n(x)-f_n'(x)|
	&\le\E\Big[\sup_{|x|\le 2\sqrt{n}}(1+|x|)^{m}\big|f_n(x)-f_n(x-\epsilon_n)\big|\Big]\\
	&\le \E\Big[c(s) \sup_{|x|\le 2\sqrt{n}}(1+|x|)^{m} |\epsilon_n|^{a} n^{b} + r_n\Big]\\
	&\le c(s) n^{(2b+m)/2} \big(\E\epsilon_n^2\big)^{a/2} + r_n\\
	&= \co\big( n^{-({s}-2)/2} \big) + r_n.
	\end{align*}
\end{proof}

\subsection{The smooth function model}\label{s.density.smooth}

In this subsection, we briefly explain the implications of the smooth function model for self-normalized sums, which we use in the proof of \cref{t.llt}. For a more detailed introduction, see \cite[Sections 2.4 and 2.8]{Hal92edgeworth}. Let $Y_j:=(X_j,X_j^2)$ with mean $\mu_Y=(0,1)$ and covariance matrix
\begin{align*}
\Sigma:= \Cov(Y)=
\begin{pmatrix}
1 & \mu_3\\
\mu_3 & \mu_4-1
\end{pmatrix}.
\end{align*}
Assume that \cref{c.density-cf} holds and that $\mu_4<\infty$. This implies that $(X_1,X_1^2)$ is non-degenerate and thus $\Sigma$ is positive definite, see \cref{r.non-deg}. Let $Y^{(i)}$ denote the $i$-th coordinate of $Y$ and $\overline{Y}:=\tfrac1n \sum Y_j$.

If $\E |X_1|^{2m}<\infty$ holds for some $m\in\N$, $m\ge3$, by \cite[Theorem 19.2]{BR76normal}, for $n$ sufficiently large, the density of $n^{1/2}(\overline{Y}-\mu_Y)$, $g_n$ exists, is bounded and admits the expansion
\begin{equation}\label{eq.BR-19.2}
\sup_{y\in\R^2} (1+\|y\|)^{m}\Big|g_n(y)-\sum_{r=0}^{m-2}n^{-r/2}\pi_r(y)\phi_{\Sigma}(y)\Big| =\co\big(n^{-(m-2)/2}\big),
\end{equation}
where $\pi_r$ is a bivariate polynomial of degree $3r$ (which includes moments up to order $2r+4$) and $\phi_{\Sigma}$ denotes the 2-dimensional normal density with parameters 0 and $\Sigma$.

Next, we derive a density $f_n$ of $T_n$ in terms of $g_n$.
\begin{align*}
\P(T_n\le a)
&=\P\Big( n^{1/2} \, \overline{Y^{(1)}} \, \Big(\overline{Y^{(2)}}\Big)^{-1/2} \le a\Big)\\
&=\P\bigg( \Big(n^{1/2} \, \overline{Y^{(1)}}\Big) \, \Big(n^{-1/2}n^{1/2}\Big(\overline{Y^{(2)}}-1\Big)+1\Big)^{-1/2} \le a\bigg)\\
&=\int_{\R^2} \1\Big\{ u_1 \, \big(n^{-1/2}u_2+1\big)^{-1/2} \le a , \, u_2>-n^{1/2}\Big\} g_n(u) \du\\
&=\int_{-n^{1/2}}^\infty \int_{-\infty}^{a \, (n^{-1/2}u_2+1)^{1/2}} g_n(u_1,u_2) \du_1\du_2.
\intertext{By substituting $u_1=x\big(n^{-1/2}u_2+1\big)^{1/2}$, applying Fubini's theorem, and substituting ${u_2=z-n^{1/2}}$,}
\P(T_n\le a)
&=\int_{-n^{1/2}}^\infty \int_{-\infty}^{a} g_n\big(x \, (n^{-1/2}u_2+1)^{1/2},u_2\big) \big(n^{-1/2}u_2+1\big)^{1/2} \dx \du_2\\
&=\int_{-\infty}^{a} \int_{-n^{1/2}}^\infty g_n\big(x \, (n^{-1/2}u_2+1)^{1/2},u_2\big) \big(n^{-1/2}u_2+1\big)^{1/2} \du_2 \dx\\
&=\int_{-\infty}^{a} \int_{0}^\infty g_n\big(x \, (n^{-1/2}z)^{1/2},z-n^{1/2}\big) \big(n^{-1/2}z\big)^{1/2} \dz \dx\\
&=\int_{-\infty}^{a} \int_{0}^\infty g_n\big(\gamma_{x,n}(z)\big) \big(n^{-1/2}z\big)^{1/2} \dz \dx
\end{align*}
where $\gamma_{x,n}(z)=\big(x(n^{-1/2}z)^{1/2},z-n^{1/2}\big)$ and thus
\begin{align}\label{eq.fn-int-gn}
f_n(x)
= \int_{0}^\infty g_n\big(\gamma_{x,n}(z)\big) \big(n^{-1/2}z\big)^{1/2} \dz.
\end{align}

\begin{remark}
	In this context, Equation (2.67) on page 79 in \cite{Hal92edgeworth} states that 
	\begin{equation}\label{eq.fn-int-gn-Hall}
	f_n(x) = \int_{S(n,x)} g_n(y) \dy
	\end{equation}
	with
	$S(n,x):
	=\big\{\gamma_{x,n}(z)\in\R^2\colon z\in(0,\infty)\big\}$.
	However, the integral in \labelcref{eq.fn-int-gn-Hall} can neither be understood as an integral with respect to the two-dimensional Lebesgue measure (as it vanishes in this case) nor as a classical curve integral for the parametrization of $S(n,x)$ by $z\mapsto\gamma_{x,n}(z)=\big(x(n^{-1/2}z)^{1/2},z-n^{1/2}\big)$, namely
	\begin{align*}
	\int_{S(n,x)} g_n\,
	=\int_0^\infty g_n(\gamma_{x,n}(z)) \|\dot\gamma_{x,n}(z)\| \dz.
	\end{align*}
	Indeed, adopting the latter strategy for the uniform distribution on the lower triangle of the unit square
	$
	\bigcup_{x\in[0,1]} S(x)$
	with
	\begin{align*}
	S(x)=\big\{\gamma_{x,n}(t)\in[0,1]\times[0,1] \colon t\in[0,1-x]\big\} \quad \text{and} \quad \gamma_{x,n}(t)=(x+t,t),
	\end{align*}
	we do not obtain a density:
	\begin{align*}
	\int_0^1 \Big( \int_{ S(x)} 2 \Big) \dx
	=\int_0^1 \int_0^{1-x} 2\, \|\dot\gamma_{x,n}(t)\| \dt \dx
	=\sqrt{2} 
	\ne 1.
	\end{align*}

\end{remark}


%



\subsection{Proof of Theorem \ref{t.llt}}\label{s.density.main}

In the following remark, we bound polynomials that appear in the proof of \cref{t.llt}.

\begin{remark}\label{r.polynome}
	In the following derivations, we will often consider polynomials of the form $\pi_r\big(\gamma_{x,n}(z)\big)$, where $\pi_r$ is a bivariate polynomial of degree $3r$, which includes moments of $X_1$ up to order $(2r+4)$ and $\gamma_{x,n}(z)=\big(x(n^{-1/2}z)^{1/2},z-n^{1/2}\big)$. Integration of these polynomials leads to cumbersome calculations which we shall avoid by the subsequent upper bounds.
	Let $\pi_{r,1}$ denote the function which arises from $\pi_r$ by replacing all coefficients and both arguments by their absolute values. Then
	\begin{equation*}
	\big|\pi_r\big(\gamma_{x,n}(z)\big)\big|=\big|\pi_r\big(x(n^{-1/2}z)^{1/2},z-n^{1/2}\big)\big|\le\pi_{r,1}\big(x(n^{-1/2}z)^{1/2},z-n^{1/2}\big).
	\end{equation*}
	Now, $\pi_{r,1}$ is symmetric in both arguments and thus non-decreasing for growing absolute values of the variables. Hence, replacing the factors $n^{-1/4}$ by 1 and $|x|$ by $(|x|\vee 1)$ yields an upper bound. This yields
	\begin{align*}
	\pi_{r,1}\big(x(n^{-1/2}z)^{1/2},z-n^{1/2}\big)
	&\le \pi_{r,1}\big((|x|\vee 1)z^{1/2},z-n^{1/2}\big)\\
	&\le (|x|\vee 1)^{3r}\pi_{r,1}\big(z+1,z-n^{1/2}\big)
	\end{align*}
	because the degree of the polynomial is at most $3r$. Next,
	\begin{equation*}
	\pi_{r,1}\big(z+1,z-n^{1/2}\big)
	\le c(r) n^{3r/2}\hat{\pi}_{r}(|z|)
	\end{equation*}
	for some suitable univariate polynomial $\hat{\pi}_{r}$ of degree $3r$.
	If $|x|<2\sqrt{n}$, we have
	\begin{equation*}
	\big|\pi_r\big(\gamma_{x,n}(z)\big)\big|
	\le c(r) n^{3r/2}(|x|\vee 1)^{3r}\hat{\pi}_{r}(|z|)
	\le c(r) n^{3r} \hat{\pi}_{r}(|z|).
	\end{equation*}
\end{remark}

%
Now we are in the position to prove the main theorem.

\begin{proof}[Proof of \cref{t.llt}]
	Depending on which moment assumption in \cref{t.llt} is satisfied, let ${w}=m$ or ${w}=m+1$. In any case, ${w}\in\N$, ${w}\ge 3$ and $\E |X_1|^{2{w}}<\infty$.
	\cref{c.density-cf} and $\E |X_1|^{2{w}}<\infty$ yield that there exists $N_1\in \N$ such that \labelcref{eq.fn-int-gn,eq.BR-19.2} hold for all $n\ge N_1$. Throughout the remainder of the proof assume $n\ge N_1$. Now \labelcref{eq.fn-int-gn} implies \cref{c.density-ex}. In order to apply \cref{p.dichte-general}, it remains to show that $f_n$ are uniformly bounded over $n\ge N_1$ and that the H\"older-type condition $(iii)$ holds. We start with the latter.
	\labelcref{eq.fn-int-gn} and \labelcref{eq.BR-19.2} together with the required notation as introduced in \cref{s.density.smooth} yield
	\begin{equation}\label{eq.f_n-von-g_n}
	\begin{split}
	f_n(x)
	&= \int_{0}^\infty g_n\big(\gamma_{x,n}(z)\big) \big(n^{-1/2}z\big)^{1/2} \dz\\
	&=\int_0^\infty \Big(\sum_{r=0}^{{w}-2}n^{-r/2}\pi_r(\gamma_{x,n}(z))\phi_{\Sigma}(\gamma_{x,n}(z)) \\
	&\qquad\qquad+ \big(1+\|\gamma_{x,n}(z)\|\big)^{-{w}} R_n(\gamma_{x,n}(z))\Big) \big(n^{-1/2}z\big)^{1/2} \dz,
	\end{split}
	\end{equation}
	where $\pi_r$ is a bivariate polynomial of degree $3r$, which includes moments up to order $2r+4$ and $\phi_{\Sigma}$ denotes the 2-dimensional normal density with parameters 0 and
	\begin{align*}
	\Sigma=
	\begin{pmatrix}
	1\enskip & 0\\
	0\enskip & \mu_4-1
	\end{pmatrix}.
	\end{align*}
	\cref{c.density-cf} implies that $X_1^2$ is non-degenerate and thus $\Sigma$ is positive definite. Additionally, $\gamma_{x,n}(z)=\big(x(n^{-1/2}z)^{1/2},z-n^{1/2}\big)$ and
	\begin{equation}\label{eq.def-Rn'}
	R_n':=\sup_{y\in\R^2} |R_n(y)|=\co\big(n^{-({w}-2)/2}\big).
	\end{equation}
	Then by the triangle-inequality, 
	\begin{align}\label{eq.f_n-f_n-Z}
	\lefteqn{|f_n(x)-f_n(x-h)|}\quad\\*
	%
	&=\bigg|\int_0^\infty \Big(\sum_{r=0}^{{w}-2}n^{-r/2}\pi_r(\gamma_{x,n}(z))\phi_{\Sigma}(\gamma_{x,n}(z)) \big(n^{-1/2}z\big)^{1/2} \nonumber\\
	&\qquad+ \big(1+\|\gamma_{x,n}(z)\|\big)^{-{w}} R_n(\gamma_{x,n}(z)) \big(n^{-1/2}z\big)^{1/2} \Big) \nonumber\\
	&\quad- \Big(\sum_{r=0}^{{w}-2}n^{-r/2}\pi_r(\gamma_{x-h,n}(z))\phi_{\Sigma}(\gamma_{x-h,n}(z)) \big(n^{-1/2}z\big)^{1/2} \nonumber\\
	&\qquad+ \big(1+\|\gamma_{x-h,n}(z)\|\big)^{-{w}} R_n(\gamma_{x-h,n}(z)) \big(n^{-1/2}z\big)^{1/2}\Big)\dz \bigg|\nonumber\\
	%
	&\le \sum_{r=0}^{{w}-2}n^{-r/2}\int_0^\infty \big|\pi_r(\gamma_{x,n}(z))-\pi_r(\gamma_{x-h,n}(z))\big| \, \phi_{\Sigma}(\gamma_{x,n}(z)) \, \big(n^{-1/2}z\big)^{1/2}\dz\nonumber\\
	&\quad+ \sum_{r=0}^{{w}-2}n^{-r/2}\int_0^\infty \big|\pi_r(\gamma_{x-h,n}(z))\big| \, \big| \phi_{\Sigma}(\gamma_{x,n}(z))-\phi_{\Sigma}(\gamma_{x-h,n}(z)) \big| \, \big(n^{-1/2}z\big)^{1/2}\dz\nonumber\\
	&\quad+R_n'\int_0^\infty \big(1+\|\gamma_{x,n}(z)\|\big)^{-{w}} \big(n^{-1/2}z\big)^{1/2}+ \big(1+\|\gamma_{x-h,n}(z)\|\big)^{-{w}} \big(n^{-1/2}z\big)^{1/2}\dz\nonumber\\
	&=: K_1+K_2+K_3.\nonumber
	\end{align}
	
	Recall that we are in the case of $|x|<2\sqrt{n}$, $|h|<1$. All following formulas will be deduced uniformly in this regime.
	\cref{r.polynome} provides the bound
	\[
	\big|\pi_r\big(\gamma_{x,n}(z)\big)\big|
	\le c(r) n^{3r} \hat{\pi}_{r}(|z|)
	\]
	for $z\in\R$. Here $\hat{\pi}_{r}$ being a polynomial of degree $3r$ whose coefficients depend on the coefficients of $\pi_r$ and $r$. Due to
	\begin{equation*}
	\begin{split}
	\big|\pi_r\big(\gamma_{x-h,n}(z)\big)\big|
	\le c(r) n^{3r/2}(|x|+1)^{3r}\hat{\pi}_{r}(|z|)
	\le c(r) n^{3r} \hat{\pi}_{r}(|z|),
	\end{split}
	\end{equation*}
	the bound also holds if we replace $x$ by $x-h$.
	Note that
	\begin{align*}
	\phi_{\Sigma}(\gamma_{x,n}(z))=\phi\big(x(n^{-1/2}z)^{1/2}\big)\,\phi_{\mu_4-1}\big(z-n^{1/2}\big).
	\end{align*}
	We also need to bound integrals of the form
	\begin{align*}
	\int_0^\infty \hat{\pi}_{r}(z) \, z \, \phi_{\mu_4-1}\big(z-n^{1/2}\big) \dz
	&\le \int_{-\infty}^\infty \hat{\pi}_{r}(z) \, |z| \, \phi_{\mu_4-1}\big(z-n^{1/2}\big) \dz\\
	&\le \E\Big[\hat{\pi}_{r}(W_n) |W_n| \Big],
	\end{align*}
	where $W_n\sim\cN(n^{1/2},\mu_4-1)$ is a normally distributed random variable with mean $n^{1/2}$ and variance $\mu_4-1$. As the expected value is bounded up to a constant by the sum over the first $(3r+1)$ absolute moments of $W_n$,
	\begin{equation}\label{int-pi-bound}
	\int_0^\infty \hat{\pi}_{r}(z) \, z \, \phi_{\mu_4-1}\big(z-n^{1/2}\big) \dz
	\le c(r)n^{(3r+1)/2}. 
	\end{equation}
	Now we expand
	\begin{align*}
	K_1
	%
	%
	&=\sum_{r=0}^{{w}-2}n^{-r/2}\int_0^\infty 
	\phi\big(x(n^{-1/2}z)^{1/2}\big)\,\phi_{\mu_4-1}\big(z-n^{1/2}\big)
	\big(n^{-1/2}z\big)^{1/2}\\
	&\qquad \cdot \Big| \pi_r\big(x(n^{-1/2}z)^{1/2},z-n^{1/2}\big) - \pi_r\big((x-h)(n^{-1/2}z)^{1/2},z-n^{1/2}\big) \Big|\dz.
	\end{align*}
	The polynomial $\pi_r\big(x(n^{-1/2}z)^{1/2},z-n^{1/2}\big)$ is a part of $\pi_r\big((x-h)(n^{-1/2}z)^{1/2},\allowbreak z-n^{1/2}\big)$ in the sense that every summand of the first polynomial is also a summand of the second polynomial. The remaining summands of $\pi_r\big((x-h)(n^{-1/2}z)^{1/2},z-n^{1/2}\big)$ all feature the factor $h(n^{-1/2}z)^{1/2}$. Therefore, $h(n^{-1/2}z)^{1/2}$ is a common factor of every summand of $\pi_r\big(x(n^{-1/2}z)^{1/2},\allowbreak z-n^{1/2}\big) - \pi_r\big((x-h)(n^{-1/2}z)^{1/2},z-n^{1/2}\big)$ and thus
	\begin{align*}
	&\big|\pi_r\big(x(n^{-1/2}z)^{1/2},z-n^{1/2}\big) - \pi_r\big((x-h)(n^{-1/2}z)^{1/2},z-n^{1/2}\big)\big|\\
	&\quad\le |h|(n^{-1/2}z)^{1/2} \cdot \pi_{r,2}\big((|x|\vee |x-h|\vee 1)(n^{-1/2}|z|)^{1/2},|z-n^{1/2}|\big).
	\end{align*}
	for a suitable polynomial $\pi_{r,2}$ with non-negative coefficients and degree at most $3r$.
	Similar in spirit to \cref{r.polynome}, we bound
	\[
	\pi_{r,2}\big((|x|\vee |x-h|\vee 1)(n^{-1/2}|z|)^{1/2},|z-n^{1/2}|\big)
	\le c(r) n^{3r} \hat{\pi}_{r}(z)
	\]
	with $\hat{\pi}_{r}$ being a polynomial of degree $3r$ whose coefficients depend on the coefficients of $\pi_r$ and $r$. Using \labelcref{int-pi-bound}, the expression above is bounded by
	\begin{align*}
	&\sum_{r=0}^{{w}-2}n^{-r/2}\int_0^\infty \phi_{\mu_4-1}\big(z-n^{1/2}\big) \big(n^{-1/2}z\big)^{1/2}
	\cdot |h|(n^{-1/2}z)^{1/2} c(r) n^{3r} \widehat{\pi}_{r}(z) \dz\\*
	&\quad\le |h|c(w) \sum_{r=0}^{{w}-2}n^{(5r-1)/2}\int_0^\infty \phi_{\mu_4-1}\big(z-n^{1/2}\big) \, z \,
	\widehat{\pi}_{r}(z) \dz\\
	&\quad\le|h| c({w}) \sum_{r=0}^{{w}-2}n^{(5r-1)/2}n^{(3r+1)/2}\\
	&\quad\le|h| c({w}) n^{4{w}-8}.
	\end{align*}
	We conclude that
	\begin{equation}\label{eq.k1}
	K_1 \le |h| c({w}) n^{4{w}-8}.
	\end{equation}
	
	For $K_2$, we use that $\phi$ is Lipschitz continuous with Lipschitz constant $L>0$ and get
	\begin{align}\label{eq.k2}
	K_2
	%
	%
	&=\sum_{r=0}^{{w}-2}n^{-r/2}\int_0^\infty 
	\big|\pi_r(\gamma_{x-h,n}(z))\big|
	\big(n^{-1/2}z\big)^{1/2}
	\phi_{\mu_4-1}\big(z-n^{1/2}\big)\nonumber\\
	&\qquad\qquad\qquad\quad \cdot \Big| \phi\big(x(n^{-1/2}z)^{1/2}\big) - \phi\big((x-h)(n^{-1/2}z)^{1/2}\big) \Big|\dz\nonumber\\
	&\le\sum_{r=0}^{{w}-2}n^{-r/2}\int_0^\infty 
	c(r) n^{3r} \hat{\pi}_{r}(z)
	\big(n^{-1/2}z\big)^{1/2}
	\phi_{\mu_4-1}\big(z-n^{1/2}\big)\\*
	&\qquad\qquad\qquad\quad \cdot L |h|(n^{-1/2}z)^{1/2} \dz\nonumber\\
	&\le|h|c({w})\sum_{r=0}^{{w}-2}n^{(5r-1)/2}\int_0^\infty 
	\hat{\pi}_{r}(z)
	\, z \,
	\phi_{\mu_4-1}\big(z-n^{1/2}\big) \dz\nonumber\\
	%
	%
	&\le|h| c({w}) n^{4{w}-8}.\nonumber
	\end{align}
	
	We are left with
	\begin{equation*}
	K_3
	= R_n'\int_0^\infty \big(1+\|\gamma_{x,n}(z)\|\big)^{-{w}} \big(n^{-1/2}z\big)^{1/2}+ \big(1+\|\gamma_{x-h,n}(z)\|\big)^{-{w}} \big(n^{-1/2}z\big)^{1/2}\dz
	\end{equation*}
	where it is obviously impossible to attain a bound in $h$. Therefore, we also have to take the factor $(1+|x|)^m$ from \labelcref{eq.density-Hoelder-1} into account and find a bound in $n$. We only examine the first summand as the second one can be handled in a similar fashion. Thus,
	\begin{align*}
	&(1+|x|)^m\int_0^\infty \big(1+\|\gamma_{x,n}(z)\|\big)^{-{w}} \big(n^{-1/2}z\big)^{1/2} \dz\\*
	&\quad\leq c(m) \int_0^\infty \sqrt{(1+x^2)^{m} \big(1+n^{-1/2}zx^2+(z-n^{1/2})^2\big)^{-{w}} \big(n^{-1/2}z\big) }\dz\\
	&\quad= c(m) \int_0^\infty \sqrt{ \bigg( \frac{1+x^2}{1+n^{-1/2}zx^2+(z-n^{1/2})^2} \bigg)^{m}}\\*
	&\quad\qquad\qquad\quad \cdot \sqrt{\big(1+n^{-1/2}zx^2+(z-n^{1/2})^2\big)^{-({w}-m)} \big(n^{-1/2}z\big) }\dz.
	\end{align*}
	Recalling $w\ge m$ and $|x|<2\sqrt{n}$, we split this integral into three regions that we treat separately. First,
	\begin{align*}
	&\int_0^{\sqrt{n}/2} \sqrt{ \bigg( \frac{1+x^2}{1+n^{-1/2}zx^2+(z-n^{1/2})^2} \bigg)^{m}}\\*
	&\qquad\quad\cdot \sqrt{\big(1+n^{-1/2}zx^2+(z-n^{1/2})^2\big)^{-({w}-m)} \big(n^{-1/2}z\big) }\dz\\*
	&\quad\le \int_0^{\sqrt{n}/2} \sqrt{ \bigg( \frac{1+x^2}{(\frac{1}{2}\sqrt{n}-n^{1/2})^2} \bigg)^{m} \big((\tfrac{1}{2}\sqrt{n}-n^{1/2})^2\big)^{-({w}-m)} \cdot \tfrac12 }\dz\\
	&\quad\le \int_0^{\sqrt{n}/2} \sqrt{ \bigg( \frac{1+4n}{n/4} \bigg)^{m} \big(n/4\big)^{-({w}-m)} }\dz\\
	&\quad\le c(w)n^{(m+1-{w})/2}
	\end{align*}
	and second by \cref{l.int-a/2},
	\begin{align*}
	&\int_{\frac{1}{2}\sqrt{n}}^{2\sqrt{n}} \sqrt{ \bigg( \frac{1+x^2}{1+n^{-1/2}zx^2+(z-n^{1/2})^2} \bigg)^{m}}\\*
	&\qquad\quad\cdot \sqrt{\big(1+n^{-1/2}zx^2+(z-n^{1/2})^2\big)^{-({w}-m)} \big(n^{-1/2}z\big) }\dz\\*
	&\quad\le \int_{\frac{1}{2}\sqrt{n}}^{2\sqrt{n}} \sqrt{ \bigg( \frac{1+x^2}{1+x^2/2} \bigg)^{m} \big(1+(z-n^{1/2})^2\big)^{-({w}-m)} \cdot 2 }\dz\\
	&\quad\le c(m) \int_{\frac{1}{2}\sqrt{n}}^{2\sqrt{n}} \big(1+(z-n^{1/2})^2\big)^{-({w}-m)/2} \dz\\
	&\quad\le \begin{cases}
	c(w) \sqrt{n} , &\text{ if } {w}=m,\\
	c(w) \log n , &\text{ if } {w}=m+1.
	\end{cases}
	\end{align*}
	Finally,
	\begin{align*}
	&\int_{2\sqrt{n}}^\infty \sqrt{ \bigg( \frac{1+x^2}{1+n^{-1/2}zx^2+(z-n^{1/2})^2} \bigg)^{m}}\\* &\qquad\quad\cdot\sqrt{\big(1+n^{-1/2}zx^2+(z-n^{1/2})^2\big)^{-({w}-m)} \big(n^{-1/2}z\big) }\dz\\*
	&\quad\le \int_{2\sqrt{n}}^\infty \sqrt{ (1+4n)^{m} \big((z-n^{1/2})^2\big)^{-{w}} \big(n^{-1/2}z\big) }\dz\\
	&\quad\le c(m) n^{m/2} n^{-1/4} \int_{2\sqrt{n}}^\infty \sqrt{ \big((z-z/2)^2\big)^{-{w}} \cdot z}\dz\\
	&\quad\le c(m) n^{m/2} n^{-1/4} \int_{2\sqrt{n}}^\infty z^{-(2{w}-1)/2} \dz\\
	&\quad\le c(w) n^{(m+1-{w})/2}.
	\end{align*}
	This also holds if $x$ is replaced by $x-h$ and thus combining the last four equations and \labelcref{eq.def-Rn'}
	\begin{align}\label{eq.k3}
	(1+|x|)^m K_3
	&= (1+|x|)^m R_n' \int_0^\infty \big(1+\|\gamma_{x,n}(z)\|\big)^{-{w}} \big(n^{-1/2}z\big)^{1/2}\nonumber\\*
	&\qquad\qquad\qquad\qquad+ \big(1+\|\gamma_{x-h,n}(z)\|\big)^{-{w}} \big(n^{-1/2}z\big)^{1/2}\dz\nonumber\\
	&\le R_n' \begin{cases}
	c(w) \sqrt{n} , &\text{ if } {w}=m,\\
	c(w) \log n , &\text{ if } {w}=m+1,
	\end{cases}\\&
	\le \begin{cases}
	\co\big(n^{-(m-3)/2}\big), &\text{ if } {w}=m,\nonumber\\
	\co\big(n^{-(m-1)/2} \, \log n \big), &\text{ if } {w}=m+1,
	\end{cases}\nonumber\\*
	&=:r_n\nonumber
	\end{align}
	uniformly for all $|x|<2\sqrt{n}$, $|h|<1$. In summary, we get by \labelcref{eq.f_n-f_n-Z,eq.k1,eq.k2,eq.k3}, 
	\begin{equation*}
	(1+|x|)^{m}|f_n(x)-f_n(x-h)|
	\le c(w) (1+|x|)^{m} |h| n^{4{w}-8} + r_n.
	\end{equation*}
	Hence, setting $a=1$ and $b=4{w}-8$, the H\"older-type condition $(iii)$ in \cref{p.dichte-general} is satisfied.
	It remains to derive a uniform bound on $f_n$. For this aim we subsequently use the identity for the normal density
	\begin{align*}
	\phi_{\sigma^2}(x)=\phi_{\sigma^2/2}(x)\,\phi_{\sigma^2/2}(x) \sqrt{2/(\pi\sigma^2)}
	\end{align*}
	for $x\in\R$ and $\sigma^2>0$.
	By \labelcref{eq.k3,eq.f_n-von-g_n}, it suffices to bound
	\begin{align*}
	\lefteqn{\int_0^\infty \sum_{r=0}^{{w}-2}n^{-r/2}\pi_r(\gamma_{x,n}(z))\phi_{\Sigma}(\gamma_{x,n}(z)) \big(n^{-1/2}z\big)^{1/2} \dz}\quad\\*
	&= \sqrt{2/(\pi(\mu_4-1))} \int_0^\infty \sum_{r=0}^{{w}-2}n^{-r/2}\pi_r(\gamma_{x,n}(z)) \, \phi\big(x(n^{-1/2}z)^{1/2}\big) \, \phi_{(\mu_4-1)/2}\big(z-n^{1/2}\big)  \\*
	&\qquad \qquad \qquad \qquad \qquad \qquad \cdot \phi_{(\mu_4-1)/2}\big(z-n^{1/2}\big) \, \big(n^{-1/2}z\big)^{1/2} \dz\\
	&= \sqrt{2/(\pi(\mu_4-1))} \int_0^\infty \sum_{r=0}^{{w}-2}n^{-r/2}\pi_r(\gamma_{x,n}(z)) \phi_{\Sigma_2}(\gamma_{x,n}(z)) \\*
	&\qquad \qquad \qquad \qquad \qquad \qquad \cdot \phi_{(\mu_4-1)/2}\big(z-n^{1/2}\big) \, \big(n^{-1/2}z\big)^{1/2} \dz\\
	&\le c(w) n^{-1/4}\int_0^\infty \phi_{(\mu_4-1)/2}\big(z-n^{1/2}\big) \, z^{1/2} \dz
	\end{align*}
	where
	\begin{align*}
	\Sigma_2=
	\begin{pmatrix}
	1\enskip & 0\\
	0\enskip & (\mu_4-1)/2
	\end{pmatrix}.
	\end{align*}
	The integral is divided into two parts that are evaluated separately. First,
	\begin{align*}
	c(w) n^{-1/4}\int_0^{2n^{1/2}} \phi_{(\mu_4-1)/2}\big(z-n^{1/2}\big) \, z^{1/2} \dz
	&\le c(w) \int_0^{2n^{1/2}} \phi_{(\mu_4-1)/2}\big(z-n^{1/2}\big)  \dz\\
	&\le c(w)
	\end{align*}
	and second,
	\begin{align*}
	\lefteqn{c(w) n^{-1/4}\int_{2n^{1/2}}^\infty \phi_{(\mu_4-1)/2}\big(z-n^{1/2}\big) \, z^{1/2} \dz}\quad\\
	&= c(w) n^{-1/4}\int_{n^{1/2}}^\infty \phi_{(\mu_4-1)/2}(z) \, \big(z+n^{1/2}\big)^{1/2} \dz\\
	&\le  c(w) n^{-1/4}\int_{n^{1/2}}^\infty \phi_{(\mu_4-1)/2}(z) \, z^{2} \dz \le  c(w).
	\end{align*}
	Thus, $f_n$ is bounded uniformly over all $n\ge N_1$ and \cref{p.dichte-general} yields
	\begin{equation*}
	\sup_{x\in\R} \, (1+|x|)^{m}|f_n(x)-\phi^{q}_{m,n}(x)|
	= \co\big( n^{-({s}-2)/2} \, (\log n)^{({s}+2\lfloor {s} \rfloor)/2} \big) 
	+ \cO\big( n^{- \lceil \frac{m-1}2 \rceil} \big)
	+ r_n
	\end{equation*}
	for both $m =w-1$ and $m=w$, and for all ${s}$ with $m \le {s} \le 2{w}$. We choose ${s}=m+2$ to finally get
	\begin{align*}
	\sup_{x\in\R} \, (1+|x|)^{m}|f_n(x)-\phi^{q}_{m,n}(x)|
	&= \begin{cases}
	\co\big(n^{-(m-3)/2}\big), &\text{ if } {w}=m,\\
	\co\big(n^{-(m-1)/2} \, \log n \big), &\text{ if } {w}=m+1.
	\end{cases}
	\end{align*}
\end{proof}

\section{Rate of convergence and Edgeworth-type expansion in the entropic central limit theorem for self-normalized sums}\label{ch.entropy}

The main result of this section is the following theorem.

\begin{theorem}\label{t.entropy-general}
	Assume that $X_1$ is symmetric, the distribution of $X_1$ is non-singular, \cref{c.density-cf} is satisfied and $\E|X_1|^{2m}<\infty$ for some $m\in\N$, $m\ge3$. Then
	\begin{equation}\label{eq.t-entropy-general}
	D(T_n)=\frac{c_2}{n^2}+\dots+\frac{c_{\lfloor (m-2)/2 \rfloor}}{n^{\lfloor (m-2)/2 \rfloor}}+ \co\big((n\log n)^{-(m-2)/2}\,\log n\big),
	\end{equation}
	where
	\begin{align*}
	c_l=\sum_{k=2}^{2l} \frac{(-1)^k}{k(k-1)} \sum \int_{\R} \frac{q_{r_1}(x)\dots q_{r_k}(x)}{\phi(x)^{k-1}} \dx, \quad l\in\N,
	\end{align*}
	with the inner sum running over all positive integers $r_1,\dots,r_k$ such that $r_1+\dots+r_k=2l$.
\end{theorem}

Together with \cref{p.prop}, \cref{t.entropy-general} implies \cref{t.entropy}. Before starting with the proof of \cref{t.entropy-general}, let us record the following observation.

\begin{remark}\label{r.E-T_n-normalized}
	The classical statistic $Z_n$ is normalized in the sense that $\E Z_n=0$ and $\Var Z_n=1$. The $t$-statistic is not normalized (see \cite[p. 72f.]{Hal92edgeworth}). In general, the self-normalized sum is also not normalized (see \cref{r.E-T_n^2}). If $X_1$ is symmetric however, it is normalized because
	\begin{equation*}
	\E T_n
	=\E \Big[\tE \big[ S_n/V_n \big]\Big]
	=\E \Big[V_n^{-1}\sum \tE \big[ X_j \big] \Big]
	=0
	\end{equation*}
	and
	\begin{equation*}
	\E T_n^2
	=\E \Big[\tE \big[ S_n^2/V_n^2 \big]\Big]
	=\E \Big[V_n^{-2}\Big(\sum \tE \big[ X_j^2 \big] + 2\sum_{j=1}^{n-1}\sum_{k=j+1}^{n} \tE \big[ X_jX_k \big]\Big) \Big]
	=1.
	\end{equation*}
\end{remark}

In the following proof, we adapt some ideas from \cite{BCG13}.

\begin{proof}[Proof of \cref{t.entropy-general}]
	By \cref{t.llt}, the densities $f_n$ exist and are uniformly bounded by some constant $M$ for $n\ge N$. Throughout the remainder of the proof assume $n\ge N$. 
	Since $X_1$ is symmetric, $T_n$ is normalized (see \cref{r.E-T_n-normalized}). Therefore, its relative entropy, defined in \labelcref{eq.def-rel-entropy}, includes the standard normal density $\phi_{0,1}=\phi$.
	
	Set $\Delta_{m,n}=(n\log n)^{-(m-2)/2}\,\log n$.
	Let $A_n\ge1$, then we split up the integral from the definition
	\begin{align}\label{eq.D-ints}
	D(T_n)
	= \int_{|x|\le A_n} f_n(x)\log \frac{f_n(x)}{\phi(x)}\dx 
	+ \int_{|x|>A_n} f_n(x)\log \frac{f_n(x)}{\phi(x)}\dx.
	\end{align}
	
	For the second integral, we get the upper bound
	\begin{align*}
	\int_{|x|>A_n} f_n(x)\log \frac{f_n(x)}{\phi(x)}\dx
	&\le\int_{|x|>A_n} f_n(x)\log \frac{M}{\phi(x)}\dx\\
	&=\log \big(M\sqrt{2\pi}\big)\int_{|x|>A_n} f_n(x)\dx+\tfrac12 \int_{|x|>A_n} x^2 f_n(x)\dx\\
	&\leq\Big(\log \big(M\sqrt{2\pi}\big)A_n^{-2}+\tfrac12\Big) \int_{|x|>A_n} x^2 f_n(x)\dx.
	\end{align*}
	By means of the inequality $u\log u\ge u-1$ for $u>0$, we get the lower bound
	\begin{align*}
	\int_{|x|>A_n} f_n(x)\log \frac{f_n(x)}{\phi(x)}\dx
	&=\int_{|x|>A_n} \phi(x) \frac{f_n(x)}{\phi(x)} \log \frac{f_n(x)}{\phi(x)}\dx\\
	&\ge\int_{|x|>A_n} \phi(x) \Big(\frac{f_n(x)}{\phi(x)} -1\Big)\dx\\
	&=\int_{|x|>A_n} \big( f_n(x)-\phi(x) \big) \dx\\
	&\ge - \P\{|Z|>A_n\}.
	\end{align*}
	Therefore,
	\begin{align}\label{eq.entropy>A_n}
	\bigg|\int_{|x|>A_n} f_n(x)\log \frac{f_n(x)}{\phi(x)}\dx\bigg|
	\le C \int_{|x|>A_n} x^2 f_n(x)\dx
	+\P\{|Z|>A_n\},
	\end{align}
	where $C=\log \big(M\sqrt{2\pi}\big)+1/2$. Before further evaluation of this term, we specify $A_n$ with the intent to balance both integrals in \labelcref{eq.D-ints}.
	Set
	\begin{align*}
	A_n=\sqrt{(m-2)\log n + (m-3)\log\log n + \rho_n},
	\end{align*}
	where $\rho_n\to \infty$ sufficiently slow and $\rho_n\le\log n$. $(\rho_n)$ is needed to achieve the order $\co$ and not only $\cO$ in the remainder in \labelcref{eq.t-entropy-general}. It will be defined at a later stage, depending on the remainder in \cref{t.llt-red}. For any $k\ge0$, the bound
	\begin{equation}\label{eq.A_n-exp}
	\begin{split}
	A_n^k \, e^{-A_n^{2}/2}
	&= \big((m-2)\log n + (m-3)\log\log n + \rho_n\big)^{k/2}\\*
	&\quad \cdot n^{-(m-2)/2} (\log n)^{-(m-3)/2} \, e^{-\rho_n/2}\\
	&=\co\big(n^{-(m-2)/2} (\log n)^{-(m-3-k)/2}\big)
	\end{split}
	\end{equation}
	is frequently needed.
	
	Recall that $|T_n|\le\sqrt{n}$ by \cref{r.T_n}. We return to \labelcref{eq.entropy>A_n} and derive by partial integration for Stieltjes integrals (see e.g. \cite[Theorem 21.67 (iv)]{HS-Analysis}) using the symmetry of $T_n$,
	\begin{align}\label{eq.entropy>A_n-expanded}
	\int_{|x|>A_n} x^2 f_n(x)\dx
	&= 2\int_{A_n}^{\sqrt{n}} x^2 \dd F_n(x)\nonumber\\
	&= 2\Big(\sqrt{n}^2 (F_n(\sqrt{n} \, ))-A_n^2 (F_n(A_n))-\int_{A_n}^{\sqrt{n}} F_n(x) 2x \dx \Big)\\
	&= 2A_n^2 \big(1-F_n(A_n)\big)+4\int_{A_n}^{\infty} x \big( 1-F_n(x)\big)\dx\nonumber 
	\end{align}
	
	By \cref{t.clt} (with $s=m+1$),
	\begin{equation*}
	F_n(x) = \Phi^{Q}_{m+1,n}(x) + (1+|x|)^{-(m+1)} R_n(x)
	\end{equation*}
	for all $x\in\R$ and $n$ sufficiently large. Here, $\sup_{x} |R_n(x)|=\co\big( n^{-(m-1)/2} \, (\log n)^{(3m+3)/2} \big)$.
	
	Therefore, we can replace $A_n^2\big(1-F_n(A_n)\big)$ with
	\begin{equation*}
	A_n^2\big(1-\Phi^{Q}_{m+1,n}(A_n)\big)
	\end{equation*}
	with an error of magnitude
	\begin{align*}
	A_n^2(1+A_n)^{-(m+1)}R_n(A_n)
	=\co\big( n^{-(m-1)/2} \, (\log n)^{(2m+4)/2} \big).
	\end{align*}
	Recall the definition
	\begin{align*}
	\Phi^{Q}_{m+1,n}(x)=\Phi(x)+\sum_{r=1}^{m-1}Q_{r}(x)n^{-r/2}
	\end{align*}
	from \labelcref{eq.Phi^Q-def}. For uneven $r$, $Q_r$ vanishes and for even $r$, $Q_{r}$ has the useful form of $\phi$ multiplied with a polynomial of degree $2r-1$. The coefficients of $Q_{r}$ are functions of the moments $\mu_3,\dots,\mu_{r+2}$ and can thus be bounded by $c(r)$. Now by \labelcref{eq.A_n-exp},
	\begin{align*}
	\Big|A_n^2\sum_{r=1}^{m-1}Q_{r}(A_n)n^{-r/2}\Big|
	\le A_n^2 \, c(m) \, A_n^{2m-3} \phi(A_n) n^{-1/2}
	=\co\big(n^{-(m-1)/2} (\log n)^{(m+2)/2}\big).
	\end{align*}
	For the last part, we use $1-\Phi(x)\le \phi(x)/x$ (for $x>0$) and \labelcref{eq.A_n-exp} to achieve
	\begin{align*}
	A_n^2\big(1-\Phi(A_n)\big)
	\le A_n \phi(A_n)
	=\co(\Delta_{m,n})
	\end{align*}
	and therefore
	\begin{align*}
	A_n^2\big(1-F_n(A_n)\big)
	=\co(\Delta_{m,n}).
	\end{align*}
	
	Similarly, in the second part of \labelcref{eq.entropy>A_n-expanded}, we can replace $\int_{A_n}^{\infty} x \big( 1-F_n(x)\big)\dx$ with
	\begin{align*}
	\int_{A_n}^{\infty} x \big( 1-\Phi^{Q}_{m+1,n}(x)\big)\dx
	\end{align*}
	with an error of at most
	\begin{align*}
	\int_{A_n}^{\infty} x (1+|x|)^{-(m+1)} R_n(x)\dx
	=\co\big( n^{-(m-1)/2} \, (\log n)^{(2m+4)/2} \big).
	\end{align*}
	As above, using $\alpha=\frac{1}{2(2m-2)}$, \cref{l.int-t-exp} (with $\beta=1$, $\nu_n=n^\alpha$) and \labelcref{eq.A_n-exp},
	\begin{align}\label{eq.int-Q}
	\lefteqn{\Big|\int_{A_n}^{\infty} x\sum_{r=1}^{m-1}Q_{r}(x)n^{-r/2} \dx\Big|}\quad\nonumber\\*
	&\leq c(m) n^{-1/2} \Big(\int_{A_n}^{n^\alpha} x^{2m-2} \phi(x) \dx + \int_{n^\alpha}^{\infty} x^{2m-2} \phi(x) \dx \Big)\nonumber\\
	&\le c(m) \int_{A_n}^{n^\alpha} \phi(x) \dx + c(m) \exp\big(-\tfrac14 n^{2\alpha}\big)\\
	&\le c(m) \phi(A_n) + c(m) \exp\big(-\tfrac14 n^{2\alpha}\big)\nonumber\\
	&=\co\big(n^{-(m-2)/2} (\log n)^{-(m-3)/2}\big)\nonumber
	\end{align}
	and by \labelcref{eq.A_n-exp},
	\begin{align*}
	\int_{A_n}^{\infty} x \big( 1-\Phi(x)\big)\dx
	&\le  \int_{A_n}^{\infty} \phi(x) \dx\le \phi(A_n)
	=\co\big(n^{-(m-2)/2} (\log n)^{-(m-3)/2}\big)
	\end{align*}
	and therefore
	\begin{align*}
	\int_{A_n}^{\infty} x \big( 1-F_n(x)\big)\dx
	=\co\big(n^{-(m-2)/2} (\log n)^{-(m-3)/2}\big).
	\end{align*}
	
	Hence by \labelcref{eq.entropy>A_n-expanded},
	\begin{align}\label{eq.int-x^2f_n}
	\int_{|x|>A_n} x^2 f_n(x)\dx
	=\co(\Delta_{m,n})
	\end{align}
	and by inserting
	\begin{align}\label{eq.P-Z>A_n}
	\P\{|Z|>A_n\}
	=2\big(1-\Phi(A_n)\big)
	\le 2 \phi(A_n)
	=\co\big(n^{-(m-2)/2} (\log n)^{-(m-3)/2}\big)
	\end{align}
	into \labelcref{eq.entropy>A_n}, we obtain 
	\begin{align*}
	\bigg|\int_{|x|>A_n} f_n(x)\log \frac{f_n(x)}{\phi(x)}\dx\bigg|
	=\co(\Delta_{m,n}).
	\end{align*}
	
	In order to evaluate $D(T_n)$, we thus only have to examine (with $L(u):=u\log(u)$)
	\begin{equation}\label{eq.entropy<A_n}
	\begin{split}
	\int_{|x|\le A_n} f_n(x)\log \frac{f_n(x)}{\phi(x)}\dx
	&= \int_{|x|\le A_n} L\Big(\frac{f_n(x)}{\phi(x)}\Big)\phi(x) \dx\\
	&= \int_{|x|\le A_n} L\big(1+u_m(x)+v_n(x)\big)\phi(x) \dx,
	\end{split}
	\end{equation}
	where
	\begin{align*}
	u_m(x)=\frac{\phi^{q}_{m-1,n}(x)-\phi(x)}{\phi(x)}
	\qquad \text{and}\qquad
	v_n(x)=\frac{f_n(x)-\phi^{q}_{m-1,n}(x)}{\phi(x)}.
	\end{align*}
	By \cref{t.llt-red},
	\begin{equation}\label{eq.entropy-llt}
	\begin{split}
	\big| f_n(x) - \phi^{q}_{m-1,n}(x)\big| 
	&\le (1+|x|)^{-(m-1)} R_n(x) \\
	&\le (1+|x|)^{-(m-1)} n^{-(m-2)/2} \log(n) r_n
	\end{split}
	\end{equation}
	for all $x\in\R$ and $n$ sufficiently large. Here, $r_n=n^{(m-2)/2} (\log n)^{-1} \, \sup_{x} |R_n(x)| =\co(1)$.
	
	There exists $C_m>0$, such that $y\mapsto (1+y)^{m-1}\phi(y)$ is decreasing for $y\ge C_m$ and thus its inverse is increasing. Additionally, for $|x|\le C_m$, we bound
	\[
	(1+|x|)^{-(m-1)}\phi(x)^{-1}\le c(m).
	\]
	Therefore, we estimate
	\[
	(1+|x|)^{-(m-1)}\phi(x)^{-1}\le (1+A_n)^{-(m-1)}\phi(A_n)^{-1}
	\]
	for all $|x|\le A_n$ and for $n$ sufficiently large. Hence,
	\begin{align}\label{eq.vn-bound}
	|v_n(x)|
	&\le (1+|x|)^{-(m-1)} \phi(x)^{-1} n^{-(m-2)/2} \log(n) r_n\nonumber\\
	&\le c(m) A_n^{-(m-1)} e^{A_n^2/2} \, n^{-(m-2)/2} \log(n) r_n\nonumber\\
	&= c(m) \big((m-2)\log n + (m-3)\log\log n + \rho_n\big)^{-(m-1)/2} \big(\log n\big)^{(m-3)/2} \\
	&\quad \cdot e^{\rho_n/2} \log(n) r_n\nonumber\\
	&\le c(m) e^{\rho_n/2} r_n.\nonumber
	\end{align}
	Next, we choose the sequence $(\rho_n)$ in a way that the last expression goes to 0. Thus, ${|v_n(x)|<1/4}$ holds for $|x|\le A_n$ and $n$ sufficiently large.
	
	Recall $\phi^{q}_{m-1,n}(x)=\phi(x)+\sum_{r=1}^{m-3}q_{r}(x)n^{-r/2}$ from \labelcref{eq.phi^q-def}. For even $r$, $q_{r}$ factorizes as $\phi$ multiplied with an even polynomial of degree $2r$. The coefficients of $q_{r}$ are functions of the moments $\mu_3,\dots,\mu_{r+2}$ and can thus be bounded by $c(r)$. As above, $q_{r}=0$ for uneven $r\in\N$. So
	\begin{equation}\label{eq.u_m}
	\begin{split}
	|u_m(x)|
	=\frac{\big|\phi^{q}_{m-1,n}(x)-\phi(x)\big|}{\phi(x)}
	\le c(m)\big(1+|x|^{2m-6}\big)n^{-1/2}
	\le c(m)A_n^{2m-6} n^{-1/2},
	\end{split}
	\end{equation}
	where the last inequality only holds for $|x|\le A_n$. In particular $|u_m(x)|<1/4$ holds for $|x|\le A_n$ and $n$ sufficiently large.
	Next, by Lemma \ref{Taylor_L},
	\begin{equation*}
	L(1+u+v)=L(1+u)+v+\theta_1 u v+\theta_2 v^2
	\end{equation*}
	for $|u|\le 1/4, |v|\le1/4$ and $|\theta_j|\le2$ depending on $u$ and $v$.
		Inserting $u=u_m(x)$ and $v=v_n(x)$, shows how to eliminate $v_n(x)$ from \labelcref{eq.entropy<A_n}. This produces an error not exceeding
	\[
	|D_1|+2D_2+2D_3,
	\]
	where
	\begin{align*}
	D_1
	&=\int_{|x|\le A_n}\big(f_n(x)-\phi^{q}_{m-1,n}(x)\big)\dx,\\
	D_2
	&=\int_{|x|\le A_n}\big|u_m(x)\big|\big|f_n(x)-\phi^{q}_{m-1,n}(x)\big|\dx \\
	\intertext{and}
	D_3
	&=\int_{|x|\le A_n}\frac{\big(f_n(x)-\phi^{q}_{m-1,n}(x)\big)^2}{\phi(x)}\dx.
	\end{align*}
	
	For even $r\in\N$, the function $q_{r}$ factorizes as $\phi$ multiplied with a sum of Hermite polynomials $H_k$ for $k>0$ (see \labelcref{eq.tp,eq.E-tp-q,eq.q-2}). Now
	\begin{align*}
	\int_\R \phi(x)H_k(x)\dx=\int_\R \phi(x)H_k(x)H_0(x)\dx=0
	\end{align*}
	for all $k>0$ yields
	\begin{align}\label{eq.int-q=0}
	\int_\R q_{r}(x)\dx=0
	\end{align}
	for all even $r=1,\dots,m-3$. As $q_{r}=0$ for uneven $r\in\N$,
	\begin{align*}
	\int_\R\ \big(f_n(x)-\phi^{q}_{m-1,n}(x)\big)\dx=0.
	\end{align*}
	Thus,
	\begin{equation*}
	|D_1|
	=\Big|\int_{|x|>A_n}\big(f_n(x)-\phi^{q}_{m-1,n}(x)\big)\dx\Big|
	\le \int_{|x|>A_n}f_n(x)\dx + \int_{|x|>A_n}\big|\phi^{q}_{m-1,n}(x)\big|\dx,
	\end{equation*}
	where the first integral is of order $\co(\Delta_{m,n})$ by \labelcref{eq.int-x^2f_n}. Using \labelcref{eq.u_m} and \labelcref{eq.P-Z>A_n}, the second integral is bounded by
	\begin{equation}\label{eq.int1}
	\begin{split}
	&\int_{|x|>A_n}\big|\phi^{q}_{m-1,n}(x)-\phi(x)\big|\dx +\int_{|x|>A_n}\phi(x)\dx\\
	&\quad \le c(m)\int_{|x|>A_n}|x|^{2m-6} n^{-1/2} \phi(x)\dx +\P\{|Z|>A_n\}\\
	&\quad =\co\big(n^{-(m-2)/2} (\log n)^{-(m-3)/2}\big).
	\end{split}
	\end{equation}
	For $m>3$, we have used \labelcref{eq.int-Q} with $\alpha=\frac{1}{2(2m-6)}$ and for $m=3$ the order is true because in that case the order of the first term is smaller than the second one.
	Therefore, ${|D_1|=\co(\Delta_{m,n})}$. By \labelcref{eq.u_m} and \labelcref{eq.entropy-llt},
	\begin{align*}
	D_2
	&\le c(m) \int_{|x|\le A_n} A_n^{2m-6} n^{-1/2} (1+|x|)^{-(m-1)} n^{-(m-2)/2} \log(n) r_n \dx\\
	&\le c(m) A_n^{2m-5} n^{-(m-1)/2} \log(n) r_n\\
	&=\co(n^{-(m-1)/2} (\log n)^{(2m-3)/2}).
	\end{align*}
	Similarly as in \labelcref{eq.vn-bound}, we obtain for all $|x|\le A_n$
	\begin{align*}
	D_3
	&\le \int_{|x|\le A_n} (1+|x|)^{-2(m-1)} \phi(x)^{-1} n^{-(m-2)} (\log n)^2 r_n^2 \dx\\
	&\le c(m) A_n^{-(2m-3)} \phi(A_n)^{-1} n^{-(m-2)} (\log n)^2 r_n^2 \\
	&= c(m) \big((m-2)\log n + (m-3)\log\log n + \rho_n\big)^{-(2m-3)/2} \big(\log n\big)^{(m-3)/2} \, e^{\rho_n/2} \\
	&\quad \cdot n^{-(m-2)/2} (\log n)^2 r_n^2 \\
	&\le c(m) (\log n)^{-(m-4)/2} n^{-(m-2)/2} r_n\\
	&=\co(\Delta_{m,n}).
	\end{align*}
	Thus, eliminating $v_n(x)$ from \labelcref{eq.entropy<A_n} leads to
	\begin{equation*}
	D(T_n)= \int_{|x|\le A_n} L\big(1+u_m(x)\big)\phi(x) \dx + \co(\Delta_{m,n}).
	\end{equation*}
	
	For $m=3$, $u_m(x)=0$ and \labelcref{eq.t-entropy-general} holds. Hence, assume $m\ge4$ from now on. By Taylor expansion around $u=0$,
	\begin{align*}
	L(1+u) 
	&= u + \sum_{k=2}^{m-2} \frac{(-1)^k}{k(k-1)}u^k+\theta u^{m-1}
	\end{align*}
	for some $\theta \ge 0$, depending on $u$ and $m$ that can be bounded by $c(m)$ for all $|u|\le1/4$.
	%
	%
	Inserting $u_m$ yields
	\begin{align*}
	\lefteqn{\int_{|x|\le A_n} L\big(1+u_m(x)\big)\phi(x) \dx}\quad\\
	&\le\int_{|x|\le A_n} \phi^{q}_{m-1,n}(x)-\phi(x) \dx 
	+ \sum_{k=2}^{m-2} \frac{(-1)^k}{k(k-1)} \int_{|x|\le A_n} u_m(x)^k\phi(x) \dx \\
	&\quad+ c(m) \int_{|x|\le A_n} \big|u_m(x)\big|^{m-1} \phi(x) \dx\\
	&=:E_1+E_2+E_3
	\end{align*}
	for $n$ large enough. By \labelcref{eq.int-q=0} and \labelcref{eq.int1},
	\begin{align*}
	\lefteqn{|E_1|=\bigg| \int_{|x|\le A_n} \phi^{q}_{m-1,n}(x)-\phi(x) \dx \bigg|
	= \bigg| \int_{|x|> A_n} \phi^{q}_{m-1,n}(x)-\phi(x) \dx \bigg|}\quad\\*
	&\le \int_{|x|> A_n} \big|\phi^{q}_{m-1,n}(x)-\phi(x)\big| \dx
	=\co\big(n^{-(m-2)/2} (\log n)^{-(m-3)/2}\big)
	\end{align*}
	and by \labelcref{eq.u_m}
	\begin{align*}
	E_3
	&\le \int_\R \big|u_m(x)\big|^{m-1} \phi(x) \dx\\
	&\le c(m) n^{-(m-1)/2}\int_\R \big(1+|x|^{2m-6}\big)^{m-1} \phi(x) \dx=\cO\big(n^{-(m-1)/2}\big).
	\end{align*}
	By \labelcref{eq.u_m} and \labelcref{eq.int-Q} with $\alpha=\frac{1}{2(2m-6)}$, the integral in $E_2$ can be extended to the whole real line at the expense of an error not exceeding
	\begin{align*}
	\lefteqn{\sum_{k=2}^{m-2} \frac{(-1)^k}{k(k-1)} \int_{|x|> A_n} u_m(x)^k\phi(x) \dx}\quad\\
	&\le c(m) \sum_{k=2}^{m-2} \frac{(-1)^k}{k(k-1)} n^{-k/2} \int_{|x|> A_n} \big(1+|x|^{2m-6}\big)^k \phi(x) \dx\\
	&=\co\big(n^{-(m-1)/2} (\log n)^{-(m-3)/2}\big).
	\end{align*}
	In summary,
	\begin{align*}
	D(T_n)
	&= \int_{|x|\le A_n} L\big(1+u_m(x)\big)\phi(x) \dx + \co(\Delta_{m,n})\\*
	&= \sum_{k=2}^{m-2} \frac{(-1)^k}{k(k-1)} \int_{\R} \frac{\big(\phi^{q}_{m-1,n}(x)-\phi(x)\big)^k}{\phi(x)^{k-1}} \dx + \co(\Delta_{m,n}).
	\end{align*}
	
	Recalling $\phi^{q}_{m-1,n}(x)-\phi(x)=\sum_{r=1}^{m-3}q_{r}(x)n^{-r/2}$, we get
	\begin{align*}
	\big(\phi^{q}_{m-1,n}(x)-\phi(x)\big)^k
	=\sum_{l=1}^{k(m-3)} n^{-l/2} \sum q_{r_1}(x)\dots q_{r_k}(x)
	\end{align*}
	where the inner sum runs over all positive integers $r_1,\dots,r_k\le m-3$ with $r_1+\dots+r_k=l$. So
	\begin{align*}
	D(T_n)= \sum_{k=2}^{m-2} \frac{(-1)^k}{k(k-1)} \sum_{l=1}^{k(m-3)} n^{-l/2} \sum \int_{\R} \frac{q_{r_1}(x)\dots q_{r_k}(x)}{\phi(x)^{k-1}} \dx + \co(\Delta_{m,n}).
	\end{align*}
	Here, all summands with uneven $l=r_1+\dots+r_k$ vanish as $q_{r}=0$ for uneven $r$.
	Additionally, all summands with $l>m-2$ will be absorbed by the remainder $\co(\Delta_{m,n})$ which yields
	\begin{equation}\label{eq.D(Tn)-finish}
	\begin{split}
	D(T_n)
	&= \sum_{k=2}^{m-2} \frac{(-1)^k}{k(k-1)} \sum_{l=1,~l \text{ even}}^{m-2} n^{-l/2} \sum \int_{\R} \frac{q_{r_1}(x)\dots q_{r_k}(x)}{\phi(x)^{k-1}} \dx + \co(\Delta_{m,n})\\*
	&= \sum_{k=2}^{m-2} \frac{(-1)^k}{k(k-1)} \sum_{l=1}^{\lfloor (m-2)/2 \rfloor} n^{-l} \sum \int_{\R} \frac{q_{r_1}(x)\dots q_{r_k}(x)}{\phi(x)^{k-1}} \dx + \co(\Delta_{m,n})
	\end{split}
	\end{equation}
	where the inner sum runs over all positive integers $r_1,\dots,r_k\le m-3$ such that $r_1+\dots+r_k=l$ in the first line and $r_1+\dots+r_k=2l$ in the second line. Define
	\begin{align*}
	c_l
	&=\sum_{k=2}^{m-2} \frac{(-1)^k}{k(k-1)} \sum \int_{\R} \frac{q_{r_1}(x)\dots q_{r_k}(x)}{\phi(x)^{k-1}} \dx\\*
	&=\sum_{k=2}^{2l} \frac{(-1)^k}{k(k-1)} \sum \int_{\R} \frac{q_{r_1}(x)\dots q_{r_k}(x)}{\phi(x)^{k-1}} \dx.
	\end{align*}
	Here, the inner sum in the first line runs over all positive integers $r_1,\dots,r_k\le m-3$ such that $r_1+\dots+r_k=2l$ and the the inner sum in the second line runs over all positive integers $r_1,\dots,r_k$ such that $r_1+\dots+r_k=2l$. The fact $2\le k\le 2l\le m-2$ implies the second identity.
	Note that
	\begin{equation*}
	c_1
	= \frac{(-1)^2}{2} \sum_{r_1+r_2=2} \int_{\R} \frac{q_{r_1}(x) q_{r_2}(x)}{\phi(x)} \dx
	= \frac{1}{2} \int_{\R} \frac{q_1(x) q_1(x)}{\phi(x)} \dx
	=0
	\end{equation*}
	and thus, \cref{t.entropy-general} follows from \labelcref{eq.D(Tn)-finish}.
\end{proof}

\section{Rate of convergence for Edgeworth expansions in the central limit theorem in total variation distance for self-normalized sums}\label{ch.TV}

The main goal of this section is to prove \cref{t.TV} by examining the total variation distance between $F_n$ and $\Phi^{Q}_{m,n}$. By \labelcref{eq.TV-L1}, this is equivalent to bounding the $L^1$ norm of the difference of the corresponding densities.

\begin{theorem}\label{t.dichte-L1}
	Assume that $X_1$ is symmetric, the distribution of $X_1$ is non-singular, \cref{c.density-cf} is satisfied and $\E|X_1|^{2m}<\infty$ for some $m\in\N$, $m\ge3$. Then there exists $N\in\N$ such that for all $n\ge N$, the statistics $T_{n}$ have densities $f_{n}$ that satisfy
	\begin{equation}\label{eq.t-dichte-L1}
	\|f_n-\phi^{q}_{m,n}\|_{L^1}= \co\big(n^{-(m-2)/2}\big).
	\end{equation}
\end{theorem}

Together with \cref{p.prop} and \labelcref{eq.TV-L1}, \cref{t.dichte-L1} implies \cref{t.TV}. On our route to proving \cref{t.dichte-L1}, we first show the following proposition. For any $h\in\R$, let $f(\lbullet+h)$ denote the function $x\mapsto f(x+h)$.

\begin{proposition}\label{p.dichte-Lp}
	Assume that $X_1$ is symmetric, the distribution of $X_1$ is non-singular and $\E|X_1|^{s}<\infty$ for some ${s}\ge2$. Additionally, assume that \cref{c.density-ex} is satisfied and there exists $p\in[1,\infty)$ and $\alpha\in\R$ with $0<\alpha\le c(s)$ such that $\|f_n\|_{L^p}\le \exp(n^\alpha)$ for all $n$ sufficiently large. Let $m=\lfloor {s} \rfloor$ and $a_n=\co(1)$ with polynomial order. Then
	\begin{equation}\label{eq.p-dichte-Lp}
	\|f_n-\phi^{q}_{m,n}\|_{L^p} \le \sup_{|h|\le a_n} \big\|f_n(\lbullet)-f_n(\lbullet+h)\big\|_{L^p} + \co\big( n^{-({s}-2)/2} \, (\log n)^{({s}+2m)/2} \big).
	\end{equation}
\end{proposition}
\begin{proof}
	Set $\beta_n= a_n^2\cdot n^{-\max\{2\alpha,(s-2)/2\}}$ and recall $f_n'$ from \labelcref{eq.f_n'-def}. Now, \cref{p.dichte-p} yields
	\begin{equation*}
	\sup_{x\in\R} \, (1+|x|)^{m}|f_n'(x)-\phi^{q}_{m,n}(x)|= r_n
	\end{equation*}
	for $r_n=\co\big( n^{-({s}-2)/2} \, (\log n)^{({s}+2m)/2} \big)$.
	By Minkowski's inequality and \cref{t.llt},
	\begin{align*}
	\|f_n-\phi^{q}_{m,n}\|_{L^p}
	&\le \|f_n-f_n'\|_{L^p} + \|f_n'-\phi^{q}_{m,n}\|_{L^p}\\
	&\le \|f_n-f_n'\|_{L^p} + \|(1+|\lbullet|)^{-m}\|_{L^p} \cdot r_n\\
	&= \|f_n-f_n'\|_{L^p} + \co\big( n^{-({s}-2)/2} \, (\log n)^{({s}+2m)/2} \big).
	\end{align*}
	Next, for $b_n:=\sqrt{\beta_n}\cdot n^\alpha \le a_n$ and by \cite[4.13(1)]{Alt},
	\begin{align*}
	\|f_n-f_n'\|_{L^p}
	&\le \Big\|\int_{-\infty}^{\infty} \big(f_n(\lbullet)-f_n(\lbullet-y)\big) \phi_{\beta_n}(y) \1_{\{|y|\le b_n\}} \dy\Big\|_{L^p}\\
	&\quad+\Big\|\int_{-\infty}^{\infty} \big(f_n(\lbullet)-f_n(\lbullet-y)\big) \phi_{\beta_n}(y) \1_{\{|y|> b_n\}} \dy\Big\|_{L^p}\\
	&\le \big\|\phi_{\beta_n}(\lbullet)\1_{\{|\lbullet|\le b_n\}}\big\|_{L^1} \cdot \sup_{|h|\le b_n} \big\|f_n(\lbullet)-f_n(\lbullet+h)\big\|_{L^p}\\
	&\quad + \big\|\phi_{\beta_n}(\lbullet)\1_{\{|\lbullet|> b_n\}}\big\|_{L^1} \cdot \sup_{|h|> b_n} \big\|f_n(\lbullet)-f_n(\lbullet+h)\big\|_{L^p}.
	\end{align*}
	As $1-\Phi(x)\le \phi(x)/x$, for $n$ sufficiently large, the expression above is bounded by
	\begin{align*}
	\lefteqn{ \sup_{|h|\le b_n} \big\|f_n(\lbullet)-f_n(\lbullet+h)\big\|_{L^p}
	+ 2\big(1-\Phi\big(b_n/\sqrt{\beta_n}\,\big)\big)\cdot 2 \|f_n\|_{L^p}}\quad\\
	&= \sup_{|h|\le b_n} \big\|f_n(\lbullet)-f_n(\lbullet+h)\big\|_{L^p}
	+ 2\big(1-\Phi(n^\alpha)\big)\cdot 2 \|f_n\|_{L^p}\\
	&\le \sup_{|h|\le a_n} \big\|f_n(\lbullet)-f_n(\lbullet+h)\big\|_{L^p} + 4 \,n^{-\alpha} \exp(-n^{2\alpha}/2+n^\alpha) \\
	&= \sup_{|h|\le a_n} \big\|f_n(\lbullet)-f_n(\lbullet+h)\big\|_{L^p} + \co\big( n^{-({s}-2)/2} \, (\log n)^{({s}+2m)/2} \big)
	\end{align*}
	in particular.
\end{proof}

The term
\begin{align*}
\sup_{|h|\le a_n} \|f_n(\lbullet)-f_n(\lbullet+h)\|_{L^p}
\end{align*}
from \labelcref{eq.p-dichte-Lp} is known as the integral (or $L^p$) modulus of continuity. There exists plentiful literature on bounds of this term if the Fourier transform of $f_n$ satisfies various kinds of tail bounds (see e.g. \cite{Cli91,GT12,Tit48}).

Next, we derive a bound on the $L^1$ modulus of continuity of $f_n$.

\begin{proposition}\label{p.L1-modulus}
	Assume that $X_1$ is symmetric, the distribution of $X_1$ is non-singular, \cref{c.density-cf} is satisfied and $\E|X_1|^{2m}<\infty$ for some $m\in\N$, $m\ge3$. Let $a_n=\co\big(n^{-(9m-17)/2}\big)$. Then there exists $N\in\N$ such that for all $n\ge N$, the statistics $T_{n}$ have densities $f_{n}$ that satisfy
	\begin{align}\label{eq.p-L1-modulus}
	\sup_{|h|\le a_n} \|f_n(\lbullet)-f_n(\lbullet+h)\|_{L^1} = \co\big(n^{-(m-2)/2}\big).
	\end{align}
\end{proposition}

\begin{proof}
	Analogously to the proof of \cref{t.llt}, \cref{c.density-cf} and $\E |X_1|^{2m}<\infty$ yield that there exists $N\in \N$ such that \labelcref{eq.fn-int-gn,eq.BR-19.2} hold for all $n\ge N$. Throughout the remainder of the proof assume $n\ge N$. Now \labelcref{eq.fn-int-gn} implies that the self-normalized sums $T_{n}$ have densities $f_n$. We expand
	\[
	|f_n(x)-f_n(x-h)|
	\le K_1+K_2+K_3
	\]
	as in \labelcref{eq.f_n-f_n-Z}, where the terms $K_1$, $K_2$ and $K_3$ are introduced. This time, the $2m$-th moment is assumed to be finite. Therefore with $w=m$, \labelcref{eq.k1,eq.k2} yield
	\[
	K_1+K_2\le |h| c(m) n^{4m-8},
	\]
	and thus
	\begin{align*}
	|f_n(x)-f_n(x-h)|
	&\le |h| c(m) n^{4m-8}
	+ K_3\\
	&= |h| c(m) n^{4m-8}\\
	&\quad+ R_n' \int_0^\infty \big(1+\|\gamma_{x,n}(z)\|\big)^{-m} \big(n^{-1/2}z\big)^{1/2}\\
	&\qquad\qquad\quad + \big(1+\|\gamma_{x-h,n}(z)\|\big)^{-m} \big(n^{-1/2}z\big)^{1/2}\dz
	\end{align*}
	for all $|x|<2\sqrt{n}$ and $|h|<1$. Here, $R_n'=\co\big(n^{-(m-2)/2}\big)$ and $\gamma_{x,n}(z)=\big(x(n^{-1/2}z)^{1/2}, \allowbreak z-n^{1/2}\big)$. By \cref{r.T_n},
	\begin{equation}\label{eq.L1-modulus-zerlegung}
	\begin{split}
	\lefteqn{\sup_{|h|\le a_n} \|f_n(\lbullet)-f_n(\lbullet+h)\|_{L^1}}\quad\\
	&\le\sup_{|h|\le a_n} \int_{-2\sqrt{n}}^{2\sqrt{n}} |h| c(m) n^{4m-8}\dx\\
	&\quad + \sup_{|h|\le a_n} \int_{-2\sqrt{n}}^{2\sqrt{n}} R_n' \int_0^\infty \big(1+\|\gamma_{x,n}(z)\|\big)^{-m} \big(n^{-1/2}z\big)^{1/2}\\
	&\qquad\qquad\qquad\qquad\qquad+ \big(1+\|\gamma_{x-h,n}(z)\|\big)^{-m} \big(n^{-1/2}z\big)^{1/2}\dz\dx.
	\end{split}
	\end{equation}
	Regarding the first summand,
	\begin{equation}\label{eq.L1-K3}
	\sup_{|h|\le a_n} \int_{-2\sqrt{n}}^{2\sqrt{n}} |h| c(m) n^{4m-8}\dx
	\le c(m) a_n \, n^{(8m-15)/2}\\
	= \co\big(n^{-(m-2)/2}\big).
	\end{equation}
	For the second summand, it is obviously impossible to attain a bound in $h$. Therefore, we have to find a bound in $n$.
	For $m\ge3$, by Fubini's theorem
	\begin{align}\label{eq.int-Z4}
	\lefteqn{\int_{-2\sqrt{n}}^{2\sqrt{n}} \int_0^\infty \big(1+\|\gamma_{x,n}(z)\|\big)^{-m} \big(n^{-1/2}z\big)^{1/2} \dz \dx}\quad\nonumber\\
	&\le \int_{-2\sqrt{n}}^{2\sqrt{n}} \int_0^\infty \big(1+\|\gamma_{x,n}(z)\|^2\big)^{-3/2} \big(n^{-1/2}z\big)^{1/2} \dz \dx\\
	&= \int_0^\infty \big(n^{-1/2}z\big)^{1/2} \int_{-2\sqrt{n}}^{2\sqrt{n}} \big(1+n^{-1/2}zx^2+(z-n^{1/2})^2\big)^{-3/2} \dx \dz.\nonumber
	\end{align}
	Applying \cref{l.int-ax^2+b} gives
	\begin{equation}\label{eq.inner-int}
	\begin{split}
	\lefteqn{\int_{-2\sqrt{n}}^{2\sqrt{n}} \big(1+n^{-1/2}zx^2+(z-n^{1/2})^2\big)^{-3/2} \dx}\quad\\
	&= 2\big(1+4n^{1/2}z+(z-n^{1/2})^2\big)^{-1/2} (1+(z-n^{1/2})^2)^{-1} 2 n^{1/2}
	\end{split}
	\end{equation}
	such that we have to examine
	\begin{align*}
	\int_0^\infty \big(n^{-1/2}z\big)^{1/2} \big(1+4n^{1/2}z+(z-n^{1/2})^2\big)^{-1/2} (1+(z-n^{1/2})^2)^{-1} n^{1/2} \dz.
	\end{align*}
	We split this integral into three regions that we treat separately. First,
	\begin{align*}
	&\int_0^{\sqrt{n}/2} \big(n^{-1/2}z\big)^{1/2} \big(1+4n^{1/2}z+(z-n^{1/2})^2\big)^{-1/2} (1+(z-n^{1/2})^2)^{-1} n^{1/2} \dz\\
	&\quad\le c(0) \int_0^{\sqrt{n}/2} 1 \cdot n^{-1/2} \, n^{-1} \, n^{1/2} \dz
	\le c(0) n^{-1/2}.
	\end{align*}
	Second, by \cref{l.int-a/2}
	\begin{align*}
	&\int_{\frac{1}{2}\sqrt{n}}^{2\sqrt{n}} \big(n^{-1/2}z\big)^{1/2} \big(1+4n^{1/2}z+(z-n^{1/2})^2\big)^{-1/2} (1+(z-n^{1/2})^2)^{-1} n^{1/2} \dz\\
	&\quad\le c(0) \int_{\frac{1}{2}\sqrt{n}}^{2\sqrt{n}} \sqrt{2} \cdot n^{-1/2} \, (1+(z-n^{1/2})^2)^{-1} \, n^{1/2} \dz\\
	&\quad= c(0) \int_{\frac{1}{2}\sqrt{n}}^{2\sqrt{n}} (1+(z-n^{1/2})^2)^{-1} \dz\le c(0).
	\end{align*}
	Finally,
	\begin{align*}
	&\int_{2\sqrt{n}}^\infty \big(n^{-1/2}z\big)^{1/2} \big(1+4n^{1/2}z+(z-n^{1/2})^2\big)^{-1/2} (1+(z-n^{1/2})^2)^{-1} n^{1/2} \dz\\
	&\quad\le c(0) \int_{2\sqrt{n}}^\infty n^{-1/4} z^{1/2} \big(z-n^{1/2}\big)^{-1} \, \big(z-n^{1/2}\big)^{-2} \, n^{1/2} \dz\\
	&\quad\le c(0) \, n^{1/4} \int_{2\sqrt{n}}^\infty z^{1/2} \big(z-z/2\big)^{-3} \dz\\
	&\quad= c(0) \, n^{1/4} \int_{2\sqrt{n}}^\infty z^{-5/2} \dz= c(0) \, n^{-1/2}.
	\end{align*}
	Combining \labelcref{eq.int-Z4}, \labelcref{eq.inner-int} and the last three formulas yields
	\begin{equation*}
	\int_{-2\sqrt{n}}^{2\sqrt{n}} \int_0^\infty \big(1+\|\gamma_{x,n}(z)\|\big)^{-m} \big(n^{-1/2}z\big)^{1/2} \dz \dx
	\le c(0).
	\end{equation*}
	This remains valid if $x$ is replaced by $x-h$ for any $|h|\le a_n$ and thus
	\begin{align*}
	& \sup_{|h|\le a_n} \int_{-2\sqrt{n}}^{2\sqrt{n}} R_n' \int_0^\infty \big(1+\|\gamma_{x,n}(z)\|\big)^{-m} \big(n^{-1/2}z\big)^{1/2}\\*
	&\qquad\qquad\qquad\qquad+ \big(1+\|\gamma_{x-h,n}(z)\|\big)^{-m} \big(n^{-1/2}z\big)^{1/2}\dz \dx\\
	&\quad\le R_n' \, c(0)
	=\co\big(n^{-(m-2)/2}\big),
	\end{align*}
	which yields \labelcref{eq.p-L1-modulus} when combined with \labelcref{eq.L1-modulus-zerlegung,eq.L1-K3}.
\end{proof}

\begin{proof}[Proof of \cref{t.dichte-L1}]
	\cref{t.llt} implies \cref{c.density-ex}. 
	Due to the form \labelcref{eq.phi^q-def} and the factor $\phi$ in all expansion terms $q_r$, $\|\phi^{q}_{m,n}\|_{L^p}\le c(m)$ for all $p\in[1,\infty)$. Thus by \labelcref{eq.fn-phi-Lp}, $\|f_n\|_{L^p}\le c(m)$ for all $p\in[1,\infty)$ and all $n$ sufficiently large.
	Set ${s}=m+1/2$ and $a_n=\co\big(n^{-(9m-16)/2}\big)$ with polynomial order.
	Now \cref{p.dichte-Lp,p.L1-modulus} imply \labelcref{eq.t-dichte-L1}.
\end{proof}

\begin{remark}
	The bounds in the $L^1$ limit theorem are by $\log n$ tighter than in the LLT. This is because the integral modulus of continuity used in the $L^1$ setting is easier to estimate than the pointwise H\"older-type bound which was the equivalent condition in the LLT setting. As this bound was used in the proof of the entropic CLT, the discrepancy of $\log n$ transfers there. The reason behind the factor $\log n$ is the appearance of different exponents in \cref{l.int-a/2}.
\end{remark}

\begin{appendix}

\section{Proofs, remarks and auxiliary lemmas}\label{app.proofs}

\subsection{Proof of Proposition \ref{p.E-tP-Q}}\label{app.proofs.pre}

\begin{remark}\label{r.V_n-lambda}
	We calculate the expected values of the $\tlambda_{l,n}$ appearing in $\tP_{2,n}$ and $\tP_{4,n}$ (see \labelcref{eq.tP-2+4}). For clear illustration, let all moments be finite in the following procedure. Additionally, we will explicitly write $\mu_2$ instead of its fixed value 1 as it helps to keep track of the order of moments. First, we need the expected values of the conditional cumulants divided by powers of $V_n$ for which we expand
	\begin{equation}\label{eq.V_n^-2}
	\begin{split}
	V_n^{-2}
	&=n^{-1}\Big(\mu_2+n^{-1} \Big(\sum (X_j^2-\mu_2)\Big)\Big)^{-1}\\
	&= n^{-1} \mu_2^{-1}
	- n^{-2} \mu_2^{-2} \Big(\sum (X_j^2-\mu_2)\Big) \\
	&\quad+ n^{-3} \mu_2^{-3} \Big(\sum (X_j^2-\mu_2)\Big)^2 
	+ \cO_p\big(n^{-3}\big)
	\end{split}
	\end{equation}
	therefore
	\begin{align*}
	V_n^{-4}
	&= n^{-2} \mu_2^{-2}
	- n^{-3} 2 \mu_2^{-3} \Big(\sum (X_j^2-\mu_2)\Big) 
	+ n^{-4} 3 \mu_2^{-4} \Big(\sum (X_j^2-\mu_2)\Big)^2
	+ \cO_p\big(n^{-4}\big),\\
	V_n^{-6}
	&= n^{-3} \mu_2^{-3}
	- n^{-4} 3 \mu_2^{-4} \Big(\sum (X_j^2-\mu_2)\Big) 
	+ n^{-5} 6 \mu_2^{-5} \Big(\sum (X_j^2-\mu_2)\Big)^2
	+ \cO_p\big(n^{-5}\big),\\
	V_n^{-8}
	&= n^{-4} \mu_2^{-4}
	- n^{-5} 4 \mu_2^{-5} \Big(\sum (X_j^2-\mu_2)\Big) 
	+ n^{-6} 10 \mu_2^{-6} \Big(\sum (X_j^2-\mu_2)\Big)^2
	+ \cO_p\big(n^{-6}\big).
	\end{align*}
	Due to $\E X_j^2=\mu_2$, the index of every factor $(X_j^2-\mu_2)$ within
	\[
	\Big(\sum (X_j^2-\mu_2)\Big)^{k}\sum|X_j|^{2 l}
	=\sum_{j_1,\dots,j_{k+1}=1}^n \big(X_{j_1}^2-\mu_2\big) \cdots \big( X_{j_k}^2-\mu_2 \big) |X_{j_{k+1}}|^{2 l}
	\]
	has to be equal to the index of another factor or the summand vanishes within the expectation. Thus, for $k\in\N$, the expectation of summands with $\big(\sum (X_j^2-\mu_2)\big)^{2k-1}$ and ${\big(\sum (X_j^2-\mu_2)\big)^{2k}}$ produce the same order of $n$. This yields
	\begin{align*}
	\E\Big[V_n^{-4}\sum|X_j|^4\Big]
	&=n^{-1} \mu_2^{-2} \mu_4
	- n^{-2} \mu_2^{-4} \big(2\mu_2\mu_6 +\mu_2^2\mu_4 - 3 \mu_4^2\big)
	+ \cO\big(n^{-3}\big),\\
	\E\Big[V_n^{-6}\sum|X_j|^6\Big]
	&=n^{-2} \mu_2^{-3} \mu_6
	- n^{-3} \mu_2^{-5} (3\mu_8\mu_2+3\mu_6\mu_2^2-6\mu_4\mu_6)
	+ \cO\big(n^{-4}\big),\\
	\E\Big[V_n^{-8}\Big(\sum|X_j|^4\Big)^2\Big]
	&=n^{-2} \mu_2^{-4} \mu_4^2
	- n^{-3} \mu_2^{-5} (8\mu_6\mu_4-7\mu_4^2\mu_2- \mu_8\mu_2)
	+ \cO\big(n^{-4}\big)
	\end{align*}
	such that by \labelcref{eq.tk-2r,eq.tlambda}
	\begin{align*}
	\E\bigg[\frac{\tlambda_{4,n}}{4!}\bigg]
	&= n^{-1} \big(- \tfrac{1}{12}\big) \mu_2^{-2} \mu_4
	+ n^{-2} \tfrac{1}{12} \mu_2^{-4} \big(2\mu_2\mu_6 +\mu_2^2\mu_4 - 3 \mu_4^2\big)
	+ \cO\big(n^{-3}\big),\\
	\E\bigg[\frac{\tlambda_{6,n}}{6!}\bigg]
	&= n^{-2} \tfrac{1}{45} \mu_2^{-3} \mu_6
	+ n^{-3} \big(- \tfrac{1}{45}\big) \mu_2^{-5} (3\mu_8\mu_2+3\mu_6\mu_2^2-6\mu_4\mu_6)
	+ \cO\big(n^{-4}\big),\\
	\E\bigg[\frac{1}{2} \Big(\frac{\tlambda_{4,n}}{4!}\Big)^2\bigg]
	&=n^{-2} \tfrac{1}{288} \mu_2^{-4} \mu_4^2
	+ n^{-3} \big(- \tfrac{1}{288}\big) \mu_2^{-5} (8\mu_6\mu_4-7\mu_4^2\mu_2- \mu_8\mu_2)
	+ \cO\big(n^{-4}\big).
	\end{align*}
\end{remark}

\begin{proof}[Proof of \cref{p.E-tP-Q}]
	The connection $\E\big[\Phi^{\tP}_{m,n}\big]$ and $\Phi^Q_{m,n}$ (the procedure for $\E\big[\phi^{\tp}_{m,n}\big]$ and $\phi^q_{m,n}$ is the same) is given in \labelcref{eq.Phi^tP}, \labelcref{eq.tP}, \labelcref{eq.Q-def} and \labelcref{eq.Phi^Q-def} and the corresponding computations are conducted in \cref{r.V_n-lambda}. However, we need to address the issue of moments higher than ${s}$ (as they could be infinite) and the factor $\exp(x^2/4)$. For ${s}<4$, there are no expansion terms in \labelcref{eq.E-tP-Q,eq.E-tp-q} and there is nothing to prove. So assume ${s}\ge4$.
	
	As in \cref{r.V_n-lambda}, we first evaluate the expectation of $V_n^{-2\alpha}\sum|X_j|^{2\alpha}$ for $\alpha\in\N$, ${2\le \alpha\le m/2}$. Define ${h}=2\lceil \frac{m-2\alpha+1}{2}\rceil$. Similar to \labelcref{eq.V_n^-2}, we perform a Taylor expansion of $V_n^{-2\alpha}$ up to order ${h}-1$ as follows
	\begin{equation}\label{eq.Exp-tL}
	\begin{split}
	V_n^{-2\alpha}\sum|X_j|^{2\alpha}
	&=\sum_{r=0}^{{h}-1} (-1)^r n^{-(\alpha+r)} t_{r,2\alpha} \Big(\sum (X_j^2-1)\Big)^{r}\Big(\sum|X_j|^{2\alpha}\Big)\\
	&\quad+(-1)^{{h}} n^{-(\alpha+{h})} t_{{h},2\alpha} \xi^{-(\alpha+{h})} \Big(\sum (X_j^2-1)\Big)^{{h}}\Big(\sum|X_j|^{2\alpha}\Big),
	\end{split}
	\end{equation}
	with (random) intermediate value $\xi$ between $1$ and $n^{-1} V_n^2$. Here, $t_{r,2\alpha}$ are the constant factors which appear in this Taylor expansion and are bounded in absolute value by $c(m)$.
	
	First, we will show that the sum of all terms where $X_\cdot$ appears with a power higher than ${s}$ is of order $\co \big(n^{-({s}-2)/2}\big)$. The same is shown for the remainder of the Taylor expansion including the $\xi$.
	
	Note that by \labelcref{eq.Lle1,eq.Mn-delta,eq.Vn-bound-1/2}, we can restrict the expectation of the terms in \labelcref{eq.Exp-tL} to the event
	\begin{equation}\label{eq.event}
	A_{n,1}\cap A_{n,2} \quad \text{where} \quad A_{n,1}:=\{M_n< \delta_n \sqrt{n} \, \} \quad\text{and} \quad A_{n,2}:=\{V_n^2 > \tfrac{n}{2}\}
	\end{equation}
	with an error of order $\co \big(n^{-({s}-2)/2}\big)$. The sequence $(\delta_n)$ is introduced in \cref{r.moment}.
	Therefore,
	\begin{equation}\label{eq.E-expanded}
	\begin{split}
	\lefteqn{\E\Big[ V_n^{-2\alpha}\sum|X_j|^{2\alpha} \Big]}\quad\\
	&=\E\Big[ \1_{A_{n,1}\cap A_{n,2}} V_n^{-2\alpha}\Big(\sum|X_j|^{2\alpha}\Big) \Big] + \co \big(n^{-({s}-2)/2}\big)\\
	&=\E\bigg[ \1_{A_{n,1}\cap A_{n,2}} \sum_{r=0}^{{h}-1} (-1)^r n^{-(\alpha+r)} t_{r,2\alpha} \Big(\sum (X_j^2-1)\Big)^{r}\Big(\sum|X_j|^{2\alpha}\Big)\bigg]\\
	&\quad+\E\bigg[ \1_{A_{n,1}\cap A_{n,2}} (-1)^{{h}} n^{-(\alpha+{h})} t_{{h},2\alpha} \xi^{-(\alpha+{h})} \Big(\sum (X_j^2-1)\Big)^{{h}}\Big(\sum|X_j|^{2\alpha}\Big) \bigg] \\
	&\quad+ \co \big(n^{-({s}-2)/2}\big).
	\end{split}
	\end{equation}
	
	Note that by \labelcref{eq.event}
	\begin{align*}
	\lefteqn{\bigg|\E\bigg[\1_{A_{n,1}\cap A_{n,2}^c} \sum_{r=0}^{{h}-1} (-1)^r n^{-(\alpha+r)} t_{r,2\alpha} \Big(\sum (X_j^2-1)\Big)^{r}\Big(\sum|X_j|^{2\alpha}\Big)\bigg]\bigg|}\quad\\*
	&\le c(h) \sum_{r=0}^{{h}-1} \E\Big[\1_{A_{n,1}\cap A_{n,2}^c} n^{-(\alpha+r)} \big(V_n^2+n\big)^{r} V_n^2 (\delta_n \sqrt{n} \, )^{2\alpha-2}\Big]\\
	&\le c(h) \sum_{r=0}^{{h}-1} \E\Big[\1_{A_{n,2}^c} n^{-(r+1)} \big(\tfrac{n}{2}+n\big)^{r} \tfrac{n}{2}\Big] \\
	&\le c(h) \P\big(A_{n,2}^c\big)\\
	&= \co \big(n^{-({s}-2)/2}\big)
	\end{align*}
	and thus,
	\begin{equation}\label{eq.A_{n,2}-raus}
	\begin{split}
	&\E\bigg[ \1_{A_{n,1}\cap A_{n,2}} \sum_{r=0}^{{h}-1} (-1)^r n^{-(\alpha+r)} t_{r,2\alpha} \Big(\sum (X_j^2-1)\Big)^{r}\Big(\sum|X_j|^{2\alpha}\Big)\bigg]\\
	&= \sum_{r=0}^{{h}-1} (-1)^r n^{-(\alpha+r)} t_{r,2\alpha} \, \E\Big[ \1_{A_{n,1}} \Big(\sum (X_j^2-1)\Big)^{r}\Big(\sum|X_j|^{2\alpha}\Big)\Big]
	+ \co \big(n^{-({s}-2)/2}\big).
	\end{split}
	\end{equation}
	
	Initially, we examine
	\[
	n^{-(\alpha+r)} \, \E\Big[ \1_{A_{n,1}} \Big(\sum (X_j^2-1)\Big)^{r}\Big(\sum|X_j|^{2\alpha}\Big)\Big]
	\]
	for $r=0,\dots,m$.
	
	
	
	We expand
	\[
	\Big(\sum (X_j^2-1)\Big)^{r}
	=\sum_{i_1+\dots+i_{n}=r} \binom{r}{i_1,\dots,i_n} \prod (X_j^2-1)^{i_{j}}.
	\]
	Here the sum extends over all $n$-tuples $(i_1,...,i_n)$ of non-negative integers with $\sum_{j=1}^ni_j=r$ and 
	\[
	{r \choose i_{1},\ldots ,i_{n}}={\frac {r!}{i_{1}!\,i_{2}!\cdots i_{n}!}}.
	\]
	Since $X_1,\dots,X_n$ are i.i.d.,
	\begin{align*}
	\E\bigg[ {r \choose i_{\pi(1)},\ldots ,i_{\pi(n)}}\prod (X_j^2-1)^{i_{\pi(j)}} \bigg]
	=	\E\bigg[ {r \choose i_{1},\ldots ,i_{n}}\prod (X_j^2-1)^{i_{j}}	\bigg]
	\end{align*}
	for all (non-random) permutations $\pi$ from $\{1,\dots,n\}$ to itself.
	
	For $l=1,\dots,r$ let $k_l$ be the number of appearances of $l$ in the tuple $(i_1,...,i_n)$ (so $\sum_{l=1}^{r}lk_l =\sum_{j=1}^ni_j=r$). Then for each tuple $(k_1,\dots,k_r)$, there exist ${n \choose k_{1},\ldots ,k_{r},(n-\sum_{l=1}^{r} k_l)}$ tuples $(i_1,...,i_n)$ such that $k_l$ is the number of appearances of $l$ in the tuple $(i_1,...,i_n)$. So instead of summing over all $n$-tuples $(i_1,...,i_n)$ with $\sum_{j=1}^ni_j=r$, we can sum over all $r$-tuples $(k_1,...,k_r)$ of non-negative integers with $\sum_{l=1}^{r}lk_l =r$ and multiply each summand by ${n \choose k_{1},\ldots ,k_{r},(n-\sum_{l=1}^{r} k_l)}$. We perform the just described summation by using the sum $\sum_{*(k_\cdot,r,\cdot)}$ and $u(k_\cdot)$ which are explained in \cref{ch.pre}. Additionally, we convert
	\[
	{r \choose i_{1},\ldots ,i_{n}}={\frac {r!}{i_{1}!\,i_{2}!\cdots i_{n}!}}=\frac{r!}{1!^{k_1}\cdots r!^{k_r}}.
	\]
	
	Now for any $(k_1,\dots,k_r)$, we identify that element of the corresponding set
	\[
	\{(i_{\pi(1)},...,i_{\pi(n)}): \pi \text{ permutations from }\{1,\dots,n\}\text{ to itself} \}
	\]
	for which
	\begin{align*}
	i_1&=\dots=i_{k_1}=1,\\
	i_{k_1+1}&=\dots=i_{k_1+k_2}=2,\\
	&\vdots\\
	i_{\sum_{l=1}^{r-1}k_l+1}&=\dots=i_{\sum_{l=1}^rk_l}=r,\\
	i_{\sum_{l=1}^rk_l+1}&=\dots=i_n=0.
	\end{align*}
	When summing over $(k_1,\dots,k_r)$, this will be the representative, evaluated within the expectation. This reordering yields
	\begin{align*}
	&\E\left[\Big(\sum (X_j^2-1)\Big)^{r} \Big(\sum|X_j|^{2\alpha}\Big) \right]\\
	&\quad=\E\left[\sum_{i_1+\dots+i_{n}=r} \binom{r}{i_1,\dots,i_n} \Big(\prod (X_j^2-1)^{i_{j}}\Big) \Big(\sum|X_j|^{2\alpha}\Big) \right]\\
	&\quad=\E\Bigg[\sum_{*(k_\cdot,r,\cdot)} {n \choose k_{1},\ldots ,k_{r},(n-u(k_\cdot))} \frac{r!}{1!^{k_1}\cdots r!^{k_r}}\\
	&\quad\qquad\cdot\bigg( \prod_{l=1}^{r} \prod_{a=1}^{k_l}\Big(X_{\sum_{b=1}^{l-1}k_b+a}^2-1 \Big)^{l} \bigg) \Big(\sum|X_j|^{2\alpha}\Big) \Bigg].
	\end{align*}
	As $M_n$ is invariant under permutation, the last set of formulas remains valid if $\1_{A_{n,1}}$ is present in the expectation. So
	\begin{equation}\label{eq.expanded}
	\begin{split}
	&n^{-(\alpha+r)} \, \E\Big[ \1_{A_{n,1}} \Big(\sum (X_j^2-1)\Big)^{r}\Big(\sum|X_j|^{2\alpha}\Big)\Big]\\
	&\quad= n^{-(\alpha+r)} \sum_{*(k_\cdot,r,\cdot)} {n \choose k_{1},\ldots ,k_{r},(n-u(k_\cdot))} \frac{r!}{1!^{k_1}\cdots r!^{k_r}}\\
	&\quad\qquad\cdot \sum \E\bigg[ \1_{A_{n,1}} \bigg( \prod_{l=1}^{r} \prod_{a=1}^{k_l}\Big(X_{\sum_{b=1}^{l-1}k_b+a}^2-1 \Big)^{l} \bigg) |X_j|^{2\alpha} \bigg]
	\end{split}
	\end{equation}
	for all $r=0,\dots,m$.
	
	Below, we distinguish between the values of $j$, which only coincide with the index of another factor for $j\le u(k_\cdot)$, and get
	\begin{equation}\label{eq.j-split}
	\begin{split}
	&\sum \E\bigg[ \1_{A_{n,1}} \bigg( \prod_{l=1}^{r} \prod_{a=1}^{k_l}\Big(X_{\sum_{b=1}^{l-1}k_b+a}^2-1 \Big)^{l} \bigg) |X_j|^{2\alpha} \bigg]\\
	&\quad= \sum_{j=u(k_\cdot)+1}^{n} \E\bigg[ \1_{A_{n,1}} \bigg( \prod_{l=1}^{r} \prod_{a=1}^{k_l}\Big(X_{\sum_{b=1}^{l-1}k_b+a}^2-1 \Big)^{l} \bigg) |X_j|^{2\alpha} \bigg]\\
	&\quad\quad+ \sum_{j=1}^{u(k_\cdot)} \E\bigg[ \1_{A_{n,1}} \bigg( \prod_{l=1}^{r} \prod_{a=1}^{k_l}\Big(X_{\sum_{b=1}^{l-1}k_b+a}^2-1 \Big)^{l} \bigg) |X_j|^{2\alpha} \bigg].
	\end{split}
	\end{equation}
	
	Next, we use that $X_1,\dots,X_n$ are i.i.d. and
	\begin{equation}\label{eq.A_{n,1}-exp}
	A_{n,1}=\{|X_1|<\delta_n \sqrt{n} \, \} \cap \dots \cap \{|X_n|<\delta_n \sqrt{n} \, \}.
	\end{equation}
	Additionally, we introduce the notation
	\[
	\E_{\delta_n,n}[\,\cdot\,]:=\E[\1_{\{|X_1|<\delta_n \sqrt{n} \, \}}\,\cdot\,].
	\]
	For the upper part of the sum this yields
	\begin{align*}
	&\sum_{j=u(k_\cdot)+1}^{n} \E\bigg[ \1_{A_{n,1}} \bigg( \prod_{l=1}^{r} \prod_{a=1}^{k_l}\Big(X_{\sum_{b=1}^{l-1}k_b+a}^2-1 \Big)^{l} \bigg) |X_j|^{2\alpha} \bigg]\\*
	&= \big(n-u(k_\cdot)\big) \bigg( \prod_{l=1}^{r}\Big(\E_{\delta_n,n}\Big[ \big(X_1^2-1 \big)^{l} \Big]\Big)^{k_l} \bigg) \E_{\delta_n,n}\Big[ |X_1|^{2\alpha} \Big] \P\big(\{|X_1|<\delta_n \sqrt{n} \, \}\big)^{n-u(k_\cdot)-1}.
	\end{align*}
	
	Assume that $k_l>0$ for at least one $l \ge \lfloor s/2\rfloor+1$ (i.e. $2l>{s}$). Furthermore, in the following formula we use the fact $2\le\alpha\le m/2$. As
	\[
	{n \choose k_{1},\ldots ,k_{r},(n-u(k_\cdot))}\le c(r) n^{u(k_\cdot)} 
	\qquad
	\text{and} 
	\qquad
	\big(n-u(k_\cdot)\big)\le n,
	\]
	these summands of $\sum_{*(k_\cdot,r,\cdot)}$ in \labelcref{eq.expanded} can be bounded by
	\begin{align*}
	&c(r) n^{-\alpha-r+u(k_\cdot)+1}
	\bigg( \prod_{l=1}^{r}\Big(\E_{\delta_n,n}\Big[ \big(X_1^2-1 \big)^{l} \Big]\Big)^{k_l} \bigg)\\
	& \qquad \cdot \E_{\delta_n,n}\Big[ |X_1|^{2\alpha} \Big] \P\big(\{|X_1|<\delta_n \sqrt{n} \, \}\big)^{n-u(k_\cdot)-1} \\
	&=c(r) n^{1-\alpha}
	\bigg( \prod_{l=1}^{r}\Big(n^{1-l}\E_{\delta_n,n}\Big[ \big(X_1^2-1 \big)^{l} \Big]\Big)^{k_l} \bigg)\\
	& \qquad \cdot \E_{\delta_n,n}\Big[ |X_1|^{2\alpha} \Big] \P\big(\{|X_1|<\delta_n \sqrt{n} \, \}\big)^{n-u(k_\cdot)-1} \\
	&=c(r) n^{1-\alpha}
	\bigg( \prod_{l=1}^{\lfloor s/2\rfloor}\Big(n^{1-l}\E_{\delta_n,n}\Big[ \big(X_1^2-1 \big)^{l} \Big]\Big)^{k_l} \bigg) \bigg( \prod_{l=\lfloor s/2\rfloor+1}^{r}\Big(n^{1-l}\E_{\delta_n,n}\Big[ \big(X_1^2-1 \big)^{l} \Big]\Big)^{k_l} \bigg) \\
	& \qquad \cdot \E_{\delta_n,n}\Big[ |X_1|^{2\alpha} \Big] \P\big(\{|X_1|<\delta_n \sqrt{n} \, \}\big)^{n-u(k_\cdot)-1} \\
	&\le c(r) n^{1-\alpha}
	\bigg( \prod_{l=1}^{\lfloor s/2\rfloor}\Big(n^{1-l}\E_{\delta_n,n}\Big[ \big(X_1^2-1 \big)^{l} \Big]\Big)^{k_l} \bigg) \\
	& \qquad \cdot \bigg( \prod_{l=\lfloor s/2\rfloor+1}^{r}\Big(n^{1-\lfloor s/2\rfloor}\delta_n^{2(l-\lfloor s/2\rfloor)} \E_{\delta_n,n}\Big[ \big(X_1^2-1 \big)^{\lfloor s/2\rfloor} \Big]\Big)^{k_l} \bigg)\\
	& \qquad \cdot \E_{\delta_n,n}\Big[ |X_1|^{2\alpha} \Big] \P\big(\{|X_1|<\delta_n \sqrt{n} \, \}\big)^{n-u(k_\cdot)-1} \\
	&\le c(r) n^{1-\alpha+1-\lfloor s/2\rfloor} \delta_n^2 \\
	&= \co \big(n^{-({s}-2)/2}\big).
	\end{align*}
	
	Thus, the upper part of the sum reduces to
	\begin{equation}\label{eq.j-ge}
	\begin{split}
	& n^{-(\alpha+r)} \sum_{*(k_\cdot,r,\cdot)} {n \choose k_{1},\ldots ,k_{r},(n-u(k_\cdot))} \frac{r!}{1!^{k_1}\cdots r!^{k_r}} \big(n-u(k_\cdot)\big) \1_{\{k_{\lfloor s/2\rfloor+1}=\dots=k_r=0\}}\\
	&\qquad \cdot \bigg( \prod_{l=1}^{r}\Big(\E_{\delta_n,n}\Big[ \big(X_1^2-1 \big)^{l} \Big]\Big)^{k_l} \bigg) \E_{\delta_n,n}\Big[ |X_1|^{2\alpha} \Big] \P\big(\{|X_1|<\delta_n \sqrt{n} \, \}\big)^{n-u(k_\cdot)-1}\\
	&\quad+\co \big(n^{-({s}-2)/2}\big).
	\end{split}
	\end{equation}
	
	Now to the lower part of the sum in \labelcref{eq.j-split}. We treat the summands separately. For any $j$, define $l^*=l^*(j)\in\{1,\dots,r\}$ as the power of the term $(X_{j}^2-1 )$ within the products in \labelcref{eq.j-split}. in turn, for any $l^*$, let $j^*$ be the smallest index with the power $l^*$, i.e. $j^*=j^*(l^*):=\argmin\{j:l^*(j)=l^*\}$. With this assignment $j^*=\sum_{b=1}^{{l^*}-1}k_b+1$.
	
	We use that $X_1,\dots,X_n$ are i.i.d. and \labelcref{eq.A_{n,1}-exp}, which yields
	\begin{align*}
	\lefteqn{\E\bigg[ \1_{A_{n,1}} \bigg( \prod_{l=1}^{r} \prod_{a=1}^{k_l}\Big(X_{\sum_{b=1}^{l-1}k_b+a}^2-1 \Big)^{l} \bigg) |X_j|^{2\alpha} \bigg]}\quad\\*
	&=\E\bigg[ \1_{A_{n,1}} \bigg( \prod_{l=1}^{r} \prod_{a=1}^{k_l}\Big(X_{\sum_{b=1}^{l-1}k_b+a}^2-1 \Big)^{l} \bigg) |X_{\sum_{b=1}^{{l^*}-1}k_b+1}|^{2\alpha} \bigg]\\
	&= \bigg( \prod_{l=1 , ~ l \ne l^*}^{r}\Big(\E_{\delta_n,n}\Big[ \big(X_1^2-1 \big)^{l} \Big]\Big)^{k_l} \bigg) \Big(\E_{\delta_n,n}\Big[ \big(X_1^2-1 \big)^{l^*} \Big]\Big)^{k_{l^*}-1}\\
	& \qquad \cdot \E_{\delta_n,n}\Big[ \big(X_1^2-1 \big)^{l^*}|X_1|^{2\alpha} \Big] \P\big(\{|X_1|<\delta_n \sqrt{n} \, \}\big)^{n-u(k_\cdot)}.
	\end{align*}
	
	Let $2\alpha+2l^*>{s}$. These summands of $\sum_{*(k_\cdot,r,\cdot)}$ in \labelcref{eq.expanded} can be bounded by
	\begin{align*}
	&c(r) n^{-\alpha-r+u(k_\cdot)}
	\bigg( \prod_{l=1 , ~ l \ne l^*}^{r}\Big(\E_{\delta_n,n}\Big[ \big(X_1^2-1 \big)^{l} \Big]\Big)^{k_l} \bigg) \Big(\E_{\delta_n,n}\Big[ \big(X_1^2-1 \big)^{l^*} \Big]\Big)^{k_{l^*}-1}\\*
	& \qquad \cdot \E_{\delta_n,n}\Big[ \big(X_1^2-1 \big)^{l^*}|X_1|^{2\alpha} \Big] \P\big(\{|X_1|<\delta_n \sqrt{n} \, \}\big)^{n-u(k_\cdot)} \\
	&=c(r) n^{-\alpha}
	\bigg( \prod_{l=1 , ~ l \ne l^*}^{r}\Big(n^{1-l}\E_{\delta_n,n}\Big[ \big(X_1^2-1 \big)^{l} \Big]\Big)^{k_l} \bigg) \Big(n^{1-l^*}\E_{\delta_n,n}\Big[ \big(X_1^2-1 \big)^{l^*} \Big]\Big)^{k_{l^*}-1}\\
	& \qquad \cdot n^{1-l^*} \E_{\delta_n,n}\Big[ \big(X_1^2-1 \big)^{l^*}|X_1|^{2\alpha} \Big] \P\big(\{|X_1|<\delta_n \sqrt{n} \, \}\big)^{n-u(k_\cdot)} \\
	&\le c(r) n^{-\alpha}
	\bigg( \prod_{l=1 , ~ l \ne l^*}^{r}\Big(\delta_n^{2(l-1)}\E_{\delta_n,n}\Big[ \big(X_1^2-1 \big) \Big]\Big)^{k_l} \bigg) \Big(\delta_n^{2(l^*-1)}\E_{\delta_n,n}\Big[ \big(X_1^2-1 \big) \Big]\Big)^{k_{l^*}-1}\\
	& \qquad \cdot n^{1-l^*} \E_{\delta_n,n}\Big[ \big(|X_1|^{2l^*+2\alpha}+|X_1|^{2\alpha} \big) \Big] \P\big(\{|X_1|<\delta_n \sqrt{n} \, \}\big)^{n-u(k_\cdot)} \\
	&\le c(r) n^{-\alpha}
	\bigg( \prod_{l=1 , ~ l \ne l^*}^{r}\Big(\delta_n^{2(l-1)}\E_{\delta_n,n}\Big[ \big(X_1^2-1 \big) \Big]\Big)^{k_l} \bigg) \Big(\delta_n^{2(l^*-1)}\E_{\delta_n,n}\Big[ \big(X_1^2-1 \big) \Big]\Big)^{k_{l^*}-1}\\
	& \qquad \cdot n^{1-l^*} (\delta_n \sqrt{n} \, )^{2l^*+2\alpha-{s}} \E_{\delta_n,n}\Big[ \big(|X_1|^{s}+|X_1|^{{s}-2l^*} \big) \Big] \P\big(\{|X_1|<\delta_n \sqrt{n} \, \}\big)^{n-u(k_\cdot)} \\
	&\le c(r) n^{-\alpha+1-l^*+l^*+\alpha-{s}/2} \delta_n^{2l^*+2\alpha-{s}} \\
	&= \co \big(n^{-({s}-2)/2} \big)
	\end{align*}
	since $\delta_n\to0$.
	
	If $2\alpha+2l^*\le{s}$, a similar procedure and bound as for the upper part of the sum in \labelcref{eq.j-split} apply. To this end, assume that $k_l>0$ for at least one $l \ge \lfloor s/2\rfloor+1$. Furthermore, in the following formula we use the fact $2\le\alpha\le m/2$. With $l^*$ as above, these summands of \labelcref{eq.expanded} can be bounded by
	\begin{align*}
	&c(r) n^{-\alpha-r+u(k_\cdot)}
	\bigg( \prod_{l=1 , ~ l \ne l^*}^{r}\Big(\E_{\delta_n,n}\Big[ \big(X_1^2-1 \big)^{l} \Big]\Big)^{k_l} \bigg) \Big(\E_{\delta_n,n}\Big[ \big(X_1^2-1 \big)^{l^*} \Big]\Big)^{k_{l^*}-1}\\*
	& \quad \cdot \E_{\delta_n,n}\Big[ \big(X_1^2-1 \big)^{l^*}|X_1|^{2\alpha} \Big] \P\big(\{|X_1|<\delta_n \sqrt{n} \, \}\big)^{n-u(k_\cdot)} \\*
	&=c(r) n^{-\alpha}
	\bigg( \prod_{l=1 , ~ l \ne l^*}^{r}\Big(n^{1-l}\E_{\delta_n,n}\Big[ \big(X_1^2-1 \big)^{l} \Big]\Big)^{k_l} \bigg) \Big(n^{1-l^*}\E_{\delta_n,n}\Big[ \big(X_1^2-1 \big)^{l^*} \Big]\Big)^{k_{l^*}-1}\\
	& \quad \cdot n^{1-l^*} \E_{\delta_n,n}\Big[ \big(X_1^2-1 \big)^{l^*}|X_1|^{2\alpha} \Big] \P\big(\{|X_1|<\delta_n \sqrt{n} \, \}\big)^{n-u(k_\cdot)} \\
	&=c(r) n^{-\alpha}
	\bigg( \prod_{l=1 , ~ l \ne l^*}^{\lfloor s/2\rfloor}\Big(n^{1-l}\E_{\delta_n,n}\Big[ \big(X_1^2-1 \big)^{l} \Big]\Big)^{k_l} \bigg) \Big(n^{1-l^*}\E_{\delta_n,n}\Big[ \big(X_1^2-1 \big)^{l^*} \Big]\Big)^{k_{l^*}-1}\\
	& \quad \cdot n^{1-l^*} \E_{\delta_n,n}\Big[ \big(X_1^2-1 \big)^{l^*}|X_1|^{2\alpha} \Big] \bigg( \prod_{l=\lfloor s/2\rfloor+1}^{r}\Big(n^{1-l}\E_{\delta_n,n}\Big[ \big(X_1^2-1 \big)^{l} \Big]\Big)^{k_l} \bigg)\\
	& \quad \cdot \P\big(\{|X_1|<\delta_n \sqrt{n} \, \}\big)^{n-u(k_\cdot)} \\
	&\le c(r) n^{-\alpha}
	\bigg( \prod_{l=1 , ~ l \ne l^*}^{\lfloor s/2\rfloor}\Big(n^{1-l}\E_{\delta_n,n}\Big[ \big(X_1^2-1 \big)^{l} \Big]\Big)^{k_l} \bigg) \Big(n^{1-l^*}\E_{\delta_n,n}\Big[ \big(X_1^2-1 \big)^{l^*} \Big]\Big)^{k_{l^*}-1}\\
	& \quad \cdot n^{1-l^*} \E_{\delta_n,n}\Big[ \big(X_1^2-1 \big)^{l^*}|X_1|^{2\alpha} \Big] \\
	&\quad \cdot \bigg( \prod_{l=\lfloor s/2\rfloor+1}^{r}\Big(n^{1-\lfloor s/2\rfloor}\delta_n^{2(l-\lfloor s/2\rfloor)} \E_{\delta_n,n}\Big[ \big(X_1^2-1 \big)^{\lfloor s/2\rfloor} \Big]\Big)^{k_l} \bigg) \P\big(\{|X_1|<\delta_n \sqrt{n} \, \}\big)^{n-u(k_\cdot)} \\
	&\le c(r) n^{-\alpha+1-l^*+1-\lfloor s/2\rfloor} \delta_n^2 \\
	&= \co \big(n^{-({s}-2)/2}\big).
	\end{align*}
	
	Summarizing, for the summands of $\sum_{*(k_\cdot,r,\cdot)}$ in \labelcref{eq.expanded} with $j\le u(k_\cdot)$ we get
	\begin{align}\label{eq.j-le}
	& n^{-(\alpha+r)} {n \choose k_{1},\ldots ,k_{r},(n-u(k_\cdot))} \frac{r!}{1!^{k_1}\cdots r!^{k_r}}\nonumber\\
	&\quad \qquad \cdot \E\bigg[ \1_{A_{n,1}} \bigg( \prod_{l=1}^{r} \prod_{a=1}^{k_l}\Big(X_{\sum_{b=1}^{l-1}k_b+a}^2-1 \Big)^{l} \bigg) |X_j|^{2\alpha} \bigg]\nonumber\\
	&\quad = n^{-(\alpha+r)} {n \choose k_{1},\ldots ,k_{r},(n-u(k_\cdot))} \frac{r!}{1!^{k_1}\cdots r!^{k_r}} \1_{\{k_{\lfloor s/2\rfloor+1}=\dots=k_r=0,\, 2\alpha+2l^*\le{s}\}}\\
	&\quad \qquad \cdot \bigg( \prod_{l=1 , ~ l \ne l^*}^{r}\Big(\E_{\delta_n,n}\Big[ \big(X_1^2-1 \big)^{l} \Big]\Big)^{k_l} \bigg) \Big(\E_{\delta_n,n}\Big[ \big(X_1^2-1 \big)^{l^*} \Big]\Big)^{k_{l^*}-1}\nonumber\\
	&\quad \qquad \cdot \E_{\delta_n,n}\Big[ \big(X_1^2-1 \big)^{l^*}|X_1|^{2\alpha} \Big] \P\big(\{|X_1|<\delta_n \sqrt{n} \, \}\big)^{n-u(k_\cdot)}\nonumber\\
	&\quad\quad+\co \big(n^{-({s}-2)/2}\big).\nonumber
	\end{align}
	
	Putting together \labelcref{eq.j-ge,eq.j-le,eq.j-split,eq.expanded}, we get
	\begin{align}\label{eq.together-indicator}
	&n^{-(\alpha+r)} \, \E\Big[ \1_{A_{n,1}} \Big(\sum (X_j^2-1)\Big)^{r}\Big(\sum|X_j|^{2\alpha}\Big)\Big]\\
	&=n^{-(\alpha+r)} \sum_{*(k_\cdot,r,\cdot)} {n \choose k_{1},\ldots ,k_{r},(n-u(k_\cdot))} \frac{r!}{1!^{k_1}\cdots r!^{k_r}} \big(n-u(k_\cdot)\big) \1_{\{k_{\lfloor s/2\rfloor+1}=\dots=k_r=0\}}\nonumber\\*
	&\qquad \cdot \bigg( \prod_{l=1}^{r}\Big(\E_{\delta_n,n}\Big[ \big(X_1^2-1 \big)^{l} \Big]\Big)^{k_l} \bigg) \E_{\delta_n,n}\Big[ |X_1|^{2\alpha} \Big] \P\big(\{|X_1|<\delta_n \sqrt{n} \, \}\big)^{n-u(k_\cdot)-1}\nonumber\\
	& \quad + n^{-(\alpha+r)} \sum_{*(k_\cdot,r,\cdot)} {n \choose k_{1},\ldots ,k_{r},(n-u(k_\cdot))} \frac{r!}{1!^{k_1}\cdots r!^{k_r}} \nonumber\\
	& \qquad \cdot \Bigg( \sum_{j=1}^{u(k_\cdot)} \1_{\{k_{\lfloor s/2\rfloor+1}=\dots=k_r=0,\, 2\alpha+2l^*\le{s}\}} \bigg( \prod_{l=1 , ~ l \ne l^*}^{r}\Big(\E_{\delta_n,n}\Big[ \big(X_1^2-1 \big)^{l} \Big]\Big)^{k_l} \bigg) \nonumber\\
	& \qquad \cdot \Big(\E_{\delta_n,n}\Big[ \big(X_1^2-1 \big)^{l^*} \Big]\Big)^{k_{l^*}-1} \E_{\delta_n,n}\Big[ \big(X_1^2-1 \big)^{l^*}|X_1|^{2\alpha} \Big] \P\big(\{|X_1|<\delta_n \sqrt{n} \, \}\big)^{n-u(k_\cdot)} \Bigg)\nonumber\\*
	& \quad +\co \big(n^{-({s}-2)/2}\big).\nonumber
	\end{align}
	
	Note that for $a=0,\dots, \lfloor s/2\rfloor$ and $b=0,\dots, \lfloor s/2\rfloor$ with $2a+2b\le{s}$, the expectations including the complementary indicator are bounded by
	\begin{align}\label{eq.X_1^a*X_1^b-ge}
	\lefteqn{\bigg|\E\Big[ \1_{\{|X_1|\ge\delta_n \sqrt{n} \, \}}\big(X_1^2-1 \big)^{a} |X_1|^{2b} \Big] \bigg|}\quad\nonumber\\*
	&\le \E\Big[ \1_{\{|X_1|\ge\delta_n \sqrt{n} \, \}}\big(|X_1|^{2a}+1 \big) |X_1|^{2b} \Big] \nonumber\\
	&\le \E\Big[ \1_{\{|X_1|\ge\delta_n \sqrt{n} \, \}} (\delta_n \sqrt{n} \, )^{-(s-2a-2b)} \big(|X_1|^s+|X_1|^{s-2a} \big) \Big] \\
	&\le c(0) n^{-(s/2-a-b)} \, \E\Big[ \1_{\{|X_1|\ge\delta_n \sqrt{n} \, \}} \delta_n ^{-s} |X_1|^s \Big] \nonumber\\
	&=\co\big(n^{-(s/2-a-b)}\big)\nonumber
	\end{align}
	by \labelcref{eq.delta_n-def}. In particular for $a=b=0$, $\P\big(\{|X_1|\ge\delta_n \sqrt{n} \, \}\big)=\co\big(n^{-s/2}\big)$.
	
	For any $l_0=0,\dots,\lfloor s/2\rfloor$, the presence of a factor $\E_{\delta_n,n}\big[ \big(X_1^2-1 \big)^{l_0} \big]$ in a summand of \labelcref{eq.together-indicator} implies $k_{l_0}\ge1$. Transitioning from $\E_{\delta_n,n}\big[ \big(X_1^2-1 \big)^{l_0} \big]$ to $\E\big[\big(X_1^2-1 \big)^{l_0} \big]$ in one of its appearances leads us to study the additional summand in \labelcref{eq.together-indicator}
	\begin{align*}
	& n^{-(\alpha+r)} {n \choose k_{1},\ldots ,k_{r},(n-u(k_\cdot))} \frac{r!}{1!^{k_1}\cdots r!^{k_r}} \big(n-u(k_\cdot)\big) \1_{\{k_{\lfloor s/2\rfloor+1}=\dots=k_r=0\}}\\
	&\qquad \cdot \bigg( \prod_{l=1,~l\ne l_0}^{r}\Big(\E_{\delta_n,n}\Big[ \big(X_1^2-1 \big)^{l} \Big]\Big)^{k_l} \bigg) \E_{\delta_n,n}\Big[ |X_1|^{2\alpha} \Big] \P\big(\{|X_1|<\delta_n \sqrt{n} \, \}\big)^{n-u(k_\cdot)-1}\\
	&\qquad \cdot \Big(\E_{\delta_n,n}\Big[ \big(X_1^2-1 \big)^{l_0} \Big]\Big)^{k_{l_0}-1}
	\bigg(- \E\Big[ \1_{\{|X_1|\ge\delta_n \sqrt{n} \, \}} \big(X_1^2-1 \big)^{l_0} \Big]\bigg)
	\end{align*}
	in the upper part and in the lower part correspondingly.
	Using \labelcref{eq.X_1^a*X_1^b-ge}, this summand is bounded in absolute value by
	\begin{align*}
	&c(m) n^{-(\alpha+r)} n^{u(k_\cdot)} n
	 \E\Big[ \1_{\{|X_1|\ge\delta_n \sqrt{n} \, \}} \big(X_1^2-1 \big)^{l_0} \Big]\\
	&\quad\le c(m) n^{-\alpha-r+u(k_\cdot)+1}
	\co\big(n^{-(s/2-l_0)}\big)\\
	&\quad= \co\big(n^{-\alpha-r+u(k_\cdot)+1-s/2+l_0}\big).
	\end{align*}
	Using the facts $\alpha\ge 2$ and $k_{l_0}\ge1$, we examine the power of $n$
	\begin{align*}
	-\alpha-r+u(k_\cdot)+1-s/2+l_0
	&\le -2-r+u(k_\cdot)+1-s/2+(l_0-1)k_{l_0}+1\\
	&= -s/2-r+(k_1+\dots+k_{l_0}+\dots+k_r)+(l_0-1)k_{l_0}\\
	&= -s/2-r+(k_1+\dots+l_0k_{l_0}+\dots+k_r)\\
	&\le -s/2.
	\end{align*}
	Thus, leaving out the indicator in the factor $\E_{\delta_n,n}\big[ \big(X_1^2-1 \big)^{l_0} \big]$ produces an error of order $\co\big(n^{-s/2}\big)$ in the upper part of \labelcref{eq.together-indicator}. In the lower part, the factor $n-u(k_\cdot)$ is missing and the order therefore even smaller.
	
	Repeating this argument consecutively (note that this requires $c(r)$ steps at most) until no $\E_{\delta_n,n}\big[ \big(X_1^2-1 \big)^{l} \big]$ is left anymore, we arrive at
	\begin{align*}
	&n^{-(\alpha+r)} \sum_{*(k_\cdot,r,\cdot)} {n \choose k_{1},\ldots ,k_{r},(n-u(k_\cdot))} \frac{r!}{1!^{k_1}\cdots r!^{k_r}} \big(n-u(k_\cdot)\big) \1_{\{k_{\lfloor s/2\rfloor+1}=\dots=k_r=0\}}\\*
	&\quad\qquad \cdot \bigg( \prod_{l=1}^{r}\Big(\E\Big[ \big(X_1^2-1 \big)^{l} \Big]\Big)^{k_l} \bigg) \E_{\delta_n,n}\Big[ |X_1|^{2\alpha} \Big] \P\big(\{|X_1|<\delta_n \sqrt{n} \, \}\big)^{n-u(k_\cdot)-1}\\
	&\quad \quad + n^{-(\alpha+r)} \sum_{*(k_\cdot,r,\cdot)} {n \choose k_{1},\ldots ,k_{r},(n-u(k_\cdot))} \frac{r!}{1!^{k_1}\cdots r!^{k_r}} \\
	&\quad \qquad \cdot \Bigg( \sum_{j=1}^{u(k_\cdot)} \1_{\{k_{\lfloor s/2\rfloor+1}=\dots=k_r=0,\, 2\alpha+2l^*\le{s}\}} \bigg( \prod_{l=1 , ~ l \ne l^*}^{r}\Big(\E\Big[ \big(X_1^2-1 \big)^{l} \Big]\Big)^{k_l} \bigg) \\
	&\quad \qquad \cdot \Big(\E\Big[ \big(X_1^2-1 \big)^{l^*} \Big]\Big)^{k_{l^*}-1} \E_{\delta_n,n}\Big[ \big(X_1^2-1 \big)^{l^*}|X_1|^{2\alpha} \Big] \P\big(\{|X_1|<\delta_n \sqrt{n} \, \}\big)^{n-u(k_\cdot)} \Bigg)\\*
	&\quad \quad +\co \big(n^{-({s}-2)/2}\big).
	\end{align*}
	
	Similar in spirit, we consecutively transition from $\P\big(\{|X_1|\ge\delta_n \sqrt{n} \, \}\big)$ to 1 at the cost of at most $\co \big(n^{-{s}/2}\big)$, from $\E_{\delta_n,n}\Big[ |X_1|^{2\alpha} \Big]$ to $\E\Big[ |X_1|^{2\alpha} \Big]$ at the cost of at most $\co \big(n^{-({s}-2)/2}\big)$ because
	\begin{align*}
	-\alpha-r+u(k_\cdot)+1-s/2+\alpha \le -(s-2)/2,
	\end{align*}
	and from $\E_{\delta_n,n}\Big[ \big(X_1^2-1 \big)^{l^*}|X_1|^{2\alpha} \Big]$ to $\E\Big[ \big(X_1^2-1 \big)^{l^*}|X_1|^{2\alpha} \Big]$ at the cost of at most $\co \big(n^{-({s}-2)/2}\big)$ because
	\begin{align*}
	-\alpha-r+u(k_\cdot)-s/2+{l^*}+\alpha
	&\le -s/2-r+(k_1+\dots+k_{l^*}+\dots+k_r)+({l^*}-1)k_{l^*}+1\\
	&= -s/2-r+(k_1+\dots+{l^*}k_{l^*}+\dots+k_r)+1\\
	&\le -(s-2)/2.
	\end{align*}
	
	As a consequence, we get from \labelcref{eq.together-indicator}
	\begin{align}\label{eq.together}
	\lefteqn{n^{-(\alpha+r)} \E\Big[ \1_{A_{n,1}} \Big(\sum (X_j^2-1)\Big)^{r}\Big(\sum|X_j|^{2\alpha}\Big)\Big]}\quad\nonumber\\*
	&=n^{-(\alpha+r)} \sum_{*(k_\cdot,r,\cdot)} {n \choose k_{1},\ldots ,k_{r},(n-u(k_\cdot))} \frac{r!}{1!^{k_1}\cdots r!^{k_r}} \nonumber\\
	&\qquad \cdot \big(n-u(k_\cdot)\big) \1_{\{k_{\lfloor s/2\rfloor+1}=\dots=k_r=0\}} \bigg( \prod_{l=1}^{r}\Big(\E\Big[ \big(X_1^2-1 \big)^{l} \Big]\Big)^{k_l} \bigg) \E\Big[ |X_1|^{2\alpha} \Big]\nonumber\\
	&\quad+ n^{-(\alpha+r)} \sum_{*(k_\cdot,r,\cdot)} {n \choose k_{1},\ldots ,k_{r},(n-u(k_\cdot))} \frac{r!}{1!^{k_1}\cdots r!^{k_r}} \\
	& \qquad \cdot \sum_{j=1}^{u(k_\cdot)} \1_{\{k_{\lfloor s/2\rfloor+1}=\dots=k_r=0,\, 2\alpha+2l^*\le{s}\}} \bigg( \prod_{l=1 , ~ l \ne l^*}^{r}\Big(\E\Big[ \big(X_1^2-1 \big)^{l} \Big]\Big)^{k_l} \bigg) \nonumber\\
	& \qquad \cdot \Big(\E\Big[ \big(X_1^2-1 \big)^{l^*} \Big]\Big)^{k_{l^*}-1} \E\Big[ \big(X_1^2-1 \big)^{l^*}|X_1|^{2\alpha} \Big]\nonumber\\*
	& \quad +\co \big(n^{-({s}-2)/2}\big).\nonumber
	\end{align}
	
	Now we return to the remainder of the Taylor expansion in \labelcref{eq.E-expanded} and shall prove that this term is of order $\co\big(n^{-({s}-2)/2}\big)$. Recall that $\xi$ is a (random) intermediate value between $1$ and $n^{-1} V_n^2$ and ${h}=2\lceil \frac{m-2\alpha+1}{2}\rceil$. Note that $\xi\ge 1/2$ on $A_{n,2}$ and since $2\le\alpha\le m/2$, $\1_{A_{n,2}}\xi^{-(\alpha+{h})}\le c(m)$. Together with $|t_{{h},2\alpha}|\le c(m)$, this yields
	\begin{align*}
	&\bigg|\E\Big[ \1_{A_{n,1}\cap A_{n,2}} (-1)^{{h}} n^{-(\alpha+{h})} t_{{h},2\alpha} \xi^{-(\alpha+{h})} \Big(\sum (X_j^2-1)\Big)^{{h}}\Big(\sum|X_j|^{2\alpha}\Big) \Big]\bigg| \\
	&\quad\le c(m) n^{-(\alpha+{h})} \, \E\Big[ \1_{A_{n,1}} \Big(\sum (X_j^2-1)\Big)^{{h}}\Big(\sum|X_j|^{2\alpha}\Big) \Big].
	\end{align*}
	If we trace through the arguments above, we see that they are valid for all $r=0,\dots,m$. Thus, all steps are also valid for the remainder of the Taylor expansion, where $r={h}$. By \labelcref{eq.together}, the expression above thus is equal to
	\begin{align}\label{eq.remainder-together}
	&c(m)\Bigg(n^{-(\alpha+r)} \sum_{*(k_\cdot,r,\cdot)} {n \choose k_{1},\ldots ,k_{r},(n-u(k_\cdot))} \frac{r!}{1!^{k_1}\cdots r!^{k_r}} \nonumber\\*
	&\qquad \cdot \big(n-u(k_\cdot)\big) \1_{\{k_{\lfloor s/2\rfloor+1}=\dots=k_r=0\}} \bigg( \prod_{l=1}^{r}\Big(\E\Big[ \big(X_1^2-1 \big)^{l} \Big]\Big)^{k_l} \bigg) \E\Big[ |X_1|^{2\alpha} \Big]\nonumber\\
	&\quad + n^{-(\alpha+r)} \sum_{*(k_\cdot,r,\cdot)} {n \choose k_{1},\ldots ,k_{r},(n-u(k_\cdot))} \frac{r!}{1!^{k_1}\cdots r!^{k_r}} \\
	&\qquad \cdot \sum_{j=1}^{u(k_\cdot)} \1_{\{k_{\lfloor s/2\rfloor+1}=\dots=k_r=0,\, 2\alpha+2l^*\le{s}\}} \bigg( \prod_{l=1 , ~ l \ne l^*}^{r}\Big(\E\Big[ \big(X_1^2-1 \big)^{l} \Big]\Big)^{k_l} \bigg) \nonumber\\
	&\qquad \cdot \Big(\E\Big[ \big(X_1^2-1 \big)^{l^*} \Big]\Big)^{k_{l^*}-1} \E\Big[ \big(X_1^2-1 \big)^{l^*}|X_1|^{2\alpha} \Big]\Bigg)\nonumber\\*
	&\quad +\co \big(n^{-({s}-2)/2}\big)\nonumber
	\end{align}
	with $r={h}$.
	Note that $\E X_1^2=1$. Thus in the first part, all summands with $k_1>0$ vanish and in the second part, all summands with $k_1>1$ vanish. This results in
	\[
	u(k_\cdot)=k_1+k_2+\dots+k_r=\tfrac12(2k_2+\dots+2k_r)\le r/2
	\]
	for the first part and
	\[
	u(k_\cdot)=k_1+k_2+\dots+k_r\le1+\tfrac12(2k_2+\dots+2k_r)\le 1+r/2
	\]
	for the second part.
	
	Now we want to determine the order in $n$ of the terms in \labelcref{eq.remainder-together}. Note that all appearing moments are finite. Using $r={h}\ge m-2\alpha+1$, we can focus on the power of $n$ which is in the first part
	\begin{align*}
	-\alpha-r+u(k_\cdot)+1
	&\le -\alpha-r+r/2+1
	=-\alpha-r/2+1\\
	&\le-\alpha-(m-2\alpha+1)/2+1
	=-(m-1)/2
	\end{align*}
	and in the second part
	\begin{align*}
	-\alpha-r+u(k_\cdot)
	\le -\alpha-r+1+r/2
	=-\alpha-r/2+1
	\le-(m-1)/2.
	\end{align*}
	Thus, all terms (at most $c(m)$ many) from \labelcref{eq.remainder-together} are of order $\co \big(n^{-({s}-2)/2}\big)$.
	Ultimately, merging \labelcref{eq.E-expanded,eq.A_{n,2}-raus,eq.together,eq.remainder-together}, we arrive at
	\begin{align}\label{eq.all}
	\lefteqn{\E\Big[ V_n^{-2\alpha}\sum|X_j|^{2\alpha} \Big]}\quad\nonumber\\
	&=\sum_{r=0}^{{h}-1} (-1)^r n^{-(\alpha+r)} t_{r,2\alpha} \sum_{*(k_\cdot,r,\cdot)} {n \choose k_{1},\ldots ,k_{r},(n-u(k_\cdot))} \frac{r!}{1!^{k_1}\cdots r!^{k_r}} \nonumber\\
	&\qquad \cdot \big(n-u(k_\cdot)\big) \1_{\{k_{\lfloor s/2\rfloor+1}=\dots=k_r=0\}} \bigg( \prod_{l=1}^{r}\Big(\E\Big[ \big(X_1^2-1 \big)^{l} \Big]\Big)^{k_l} \bigg) \E\Big[ |X_1|^{2\alpha} \Big]\nonumber\\
	&\quad+ \sum_{r=0}^{{h}-1} (-1)^r n^{-(\alpha+r)} t_{r,2\alpha} \sum_{*(k_\cdot,r,\cdot)} {n \choose k_{1},\ldots ,k_{r},(n-u(k_\cdot))} \frac{r!}{1!^{k_1}\cdots r!^{k_r}} \\
	&\qquad \cdot \sum_{j=1}^{u(k_\cdot)} \1_{\{k_{\lfloor s/2\rfloor+1}=\dots=k_r=0,\, 2\alpha+2l^*\le{s}\}} \bigg( \prod_{l=1 , ~ l \ne l^*}^{r}\Big(\E\Big[ \big(X_1^2-1 \big)^{l} \Big]\Big)^{k_l} \bigg) \nonumber\\
	&\qquad \cdot \Big(\E\Big[ \big(X_1^2-1 \big)^{l^*} \Big]\Big)^{k_{l^*}-1} \E\Big[ \big(X_1^2-1 \big)^{l^*}|X_1|^{2\alpha} \Big]\nonumber\\
	&\quad +\co \big(n^{-({s}-2)/2}\big).\nonumber
	\end{align}
	
	All expectations in \labelcref{eq.all} can be calculated explicitly as no power of $X_1$ is greater than $s$. Using $u(k_\cdot)\le r, \alpha\in\N$ and $\alpha\ge 2$, this results in an expansion in terms of orders $n^{-1},n^{-2},\dots,n^{-\lfloor m-2/2\rfloor}$ with moments up to order $m$ and a remainder of order $\co \big(n^{-({s}-2)/2}\big)$. By \labelcref{eq.tk-2r,eq.tlambda} for $k\ge4$, the expectation of $\tlambda_{k,n}$ and its powers admit the same expansion up to the respective coefficient from \labelcref{eq.tk-2r}. In \cref{r.V_n-lambda}, this expansion is performed exemplarily for $k=4$ and $k=6$. By \labelcref{eq.Phi^tP} and \labelcref{eq.tP}, $\E\big[\Phi^{\tP}_{m,n}(x)\big]$ admits a similar representation. By the definition of $Q_r$ in \labelcref{eq.Q-def}, the expansion terms of order $n^{-r/2}$ are assigned to their respective $Q_r$ (or to the remainder $\co \big(n^{-({s}-2)/2}\big)$). Note that $Q_{r}=0$ for uneven $r$. That is, all the non-remainder terms in $$\E\big[\Phi^{\tP}_{m,n}(x)\big]$$ are exactly the summands appearing in $\Phi^Q_{m,n}$, defined in \labelcref{eq.Phi^Q-def}, and therefore cancel each other out in \labelcref{eq.E-tP-Q,eq.E-tp-q}.
	What still needs to be examined is the dependence of the remainder on the argument $x$. In $\tP_{k,n}$ (as well as in $\tp_{k,n}$), the $\tlambda_{k,n}$ are multiplied by $\phi$ and a Hermite polynomial of order $l\le2m$. For $l\le2m$, we can bound
	\[
	\sup_{x} \exp(x^2/4)\phi(x)H_l(x)\le c(l).
	\]
	As this non-random factor remains unaffected by the expectation, it also appears in the remainder such that \labelcref{eq.E-tP-Q,eq.E-tp-q} hold.
\end{proof}

\subsection{Further proofs for Section \ref{ch.distr}}\label{app.proofs.distr}

\begin{proof}[Proof of \cref{l.E.mom}]
We prove this lemma by an approach similar to the one used in the classical proof of Markov's inequality. Without loss of generality let $a_1=\max\{a_1,\dots,a_l\}$. By using $M_{n,l}=\max_{j= l+1,l+2,\dots,n} |X_j|$ and $j_0\in\{1,\dots,n\}$ being the smallest index with $|X_{j_0}|=M_n$, we see
\begin{align*}
\lefteqn{ \E\Big[\1_{\big\{M_n\ge \sqrt{\frac{n}{\log(n)\, \eta}}\big\}} |X_1|^{a_1} \cdots |X_l|^{a_l}\Big]}\quad\\
&= \sum_{j=1}^l \E\Big[ \1_{\{j_0=j\}}\1_{\big\{M_n\ge \sqrt{\frac{n}{\log(n)\, \eta}}\big\}} |X_1|^{a_1} \cdots |X_l|^{a_l}\Big]\\
&\quad+ \E\Big[ \sum_{j=l+1}^n \1_{\{j_0=j\}}\1_{\big\{M_n\ge \sqrt{\frac{n}{\log(n)\, \eta}}\big\}} |X_1|^{a_1} \cdots |X_l|^{a_l}\Big]\\
&\le l \, \E\Big[\1_{\{j_0=1\}}\1_{\big\{M_n\ge \sqrt{\frac{n}{\log(n)\, \eta}}\big\}} |X_1|^{a_1} \cdots |X_l|^{a_l}\Big]\\
&\quad+ \E\Big[\1_{\{j_0>l\}}\1_{\big\{M_n\ge \sqrt{\frac{n}{\log(n)\, \eta}}\big\}} |X_1|^{a_1} \cdots |X_l|^{a_l}\Big]\\
&\le \, l \, \big(\tfrac{n}{\log(n)\, \eta}\big)^{-({s}-a_1)/2} \, \E\Big[\1_{\{j_0=1\}}\1_{\big\{M_n\ge \sqrt{\frac{n}{\log(n)\, \eta}}\big\}} |X_1|^{{s}} |X_2|^{a_2} \cdots |X_l|^{a_l}\Big]\\
&\quad+ \E\Big[\1_{\{j_0>l\}}\1_{\big\{M_{n,l}\ge \sqrt{\frac{n}{\log(n)\, \eta}}\big\}} |X_1|^{a_1} \cdots |X_l|^{a_l}\Big]\\
&\le \, l \, \big(\tfrac{n}{\log(n)\, \eta}\big)^{-({s}-a_1)/2} \, \E\Big[ |X_1|^{{s}} |X_2|^{a_2} \cdots |X_l|^{a_l}\Big]\\
&\quad+ \E\Big[\1_{\big\{M_{n,l}\ge \sqrt{\frac{n}{\log(n)\, \eta}}\big\}} |X_1|^{a_1} \cdots |X_l|^{a_l}\Big]\\
&= l \, \eta^{({s}-a_1)/2}\, (\log n)^{({s}-a_1)/2} n^{-({s}-a_1)/2} \E |X_1|^{{s}} \E |X_1|^{a_2} \cdots \E |X_1|^{a_l}\\
&\quad+ \P\big(M_{n,l}\ge \sqrt{\tfrac{n}{\log(n)\, \eta}} \, \big) \E |X_1|^{a_1} \cdots \E |X_1|^{a_l}\\
& = \co\Big( n^{-({s}-\max\{2,a_1\})/2} \, (\log n)^{{s}/2} \Big),
\end{align*}
where we used \labelcref{eq.Mn-an} in the last step.
\end{proof}

\begin{proof}[Proof of \cref{l.E.tL}]
Using $(\delta_n)$ from \cref{r.moment}, we get by \labelcref{eq.Lle1}, \labelcref{eq.Vn-bound-1/2} and \labelcref{eq.Mn-delta}
\begin{align*}
\E\left[ \tL_{k,n}\right]
&\le\E\left[ \1_{\{V_n^2 \le \frac n2\}}\right]\\
&\quad+\E\left[ \1_{\{M_n\ge \delta_n \sqrt{n} \, \}}\right]\\
&\quad+\E\left[ \1_{\{V_n^2 > n/2, M_n< \delta_n \sqrt{n} \, \}} V_n^{-k}M_n^{k-{s}}\sum|X_j|^{{s}} \right]\\
&\le\co\big(n^{-({s}-2)/2}\big)\\
&\quad+ \co\big(n^{-({s}-2)/2}\big)\\
&\quad+ 2^{k/2} n^{-({s}-2)/2} \delta_n^{k-{s}} \E\big[ |X_1|^{{s}}\big]\\
&= \co\big(n^{-({s}-2)/2}\big)
\end{align*}
since $\delta_n\to0$.
\end{proof}

\begin{proof}[Proof of \cref{l.exp.B_n}]
We split up the expectation and get by \labelcref{eq.Vn-bound-1/2,eq.Mn-an}
\begin{align*}
\E\left[ \exp\big(-\eta\,B_n^2\big)\right]
&\le\E\left[ \1_{\{V_n^2 \le \frac n2\}}\right]\\
&\quad+\E\left[ \1_{\big\{M_n\ge \sqrt{\frac{n\,\eta}{\log(n)\, {s}}}\big\}}\right]\\
&\quad+\E\left[ \1_{\big\{V_n^2 > n/2, M_n< \sqrt{\frac{n\,\eta}{\log(n)\, {s}}}\big\}} \exp\big(-\eta\,V_n^2\,M_n^{-2}\big) \right]\\
&\le\P\left(V_n^2 \le \tfrac n2\right)\\
&\quad+\P\left(M_n\ge \sqrt{\tfrac{n\,\eta}{\log(n)\, {s}}}\,\right)\\
&\quad+\E\left[ \1_{\big\{V_n^2 > n/2, M_n< \sqrt{\frac{n\,\eta}{\log(n)\, {s}}}\big\}} \exp\big(-\eta\,\tfrac n2\,\tfrac{\log(n)\,s}{n\,\eta}\big)\right]\\
&\le\co\big(n^{-({s}-2)/2}\big)\\
&\quad+\co\big( n^{-({s}-2)/2} \, (\log n)^{{s}/2} (\eta\wedge 1)^{-s/2} \big)\\
&\quad+n^{-{s}/2}\\
&= \co\big( n^{-({s}-2)/2} \, (\log n)^{{s}/2} (\eta\wedge 1)^{-s/2} \big).
\end{align*}
\end{proof}

\subsection{Further proofs for Section \ref{ch.density}}\label{app.proofs.density}

\begin{proof}[Proof of \cref{p.prop}]
Take $f$ to be the density of $X_1$, $Y_j:=(X_j,X_j^2)$, $B:=B_1\times B_2:=[a_1,b_1]\times[a_2,b_2]$ with $a_1<b_1$, $a_2<b_2$ in $\R$. Then by substitution and Fubini's theorem,
\begin{align*}
\lefteqn{\bP\big(Y_1+Y_2+Y_3 \in B\big)}\quad\\*
&=\bP\Big(X_1+X_2+X_3 \in B_1, X_1^2+X_2^2+X_3^2 \in B_2\Big)\\
&=2\int_\R\int_\R\int_\R f(x)f(y)f(z)\1\big\{x+y+z\in B_1, x^2+y^2+z^2\in B_2, x\le y\big\}\dx\dy\dz\\
&=2\int_\R\int_\R\int_\R f(s-y-z)f(y)f(z)\\*
&\qquad\qquad\quad\cdot\1\big\{s\in B_1, (s-y-z)^2+y^2+z^2 \in B_2, (s-z)/2 \le y\big\}\ds\dy\dz\\
&=2\int_{B_1}\int_\R\int_\R f(s-y-z)f(y)f(z)\\*
&\qquad\qquad\quad\cdot\1\big\{ (s-y-z)^2+y^2+z^2 \in B_2, (s-z)/2 \le y\big\}\dy\dz\ds.
\end{align*}
Define
\[
v_{s,z}(y):=(s-y-z)^2+y^2+z^2=\tfrac12 \Big( \big(2y-(s-z)\big)^2+s^2+3z^2-2sz\Big).
\]
Note that
\[
v'_{s,z}(y) = 2 (2y-s+z).
\]
As we only consider the regime $(s-z)/2 \le y$, the inverse is
\[
v_{s,z}^{-1}(v)=\tfrac12 \Big(s-z+\sqrt{2v-s^2-3z^2+2sz}\Big)=\tfrac12 \Big(s-z+\tfrac{1}{\sqrt{3}} \sqrt{6v-2s^2-(3z-s)^2}\Big).
\]
Thus the expression above is equal to
\begin{align*}
&2\int_{B_1}\int_\R\int_{(s-z)/2}^\infty \frac{2(2v_{s,z}^{-1}(v_{s,z}(y))-s+z)}{2(2v_{s,z}^{-1}(v_{s,z}(y))-s+z)}f(s-v_{s,z}^{-1}(v_{s,z}(y))-z)f(v_{s,z}^{-1}(v_{s,z}(y)))f(z)\\*
&\qquad\qquad\qquad\quad\cdot\1\big\{ v_{s,z}(y) \in B_2 \big\}\dy\dz\ds\\*
&=2\int_{B_1}\int_\R\int_{v_{s,z}((s-z)/2)}^\infty \frac{1}{2(2v_{s,z}^{-1}(v)-s+z)} f\big(s-v_{s,z}^{-1}(v)-z\big) f(v_{s,z}^{-1}(v)) f(z) \\*
&\qquad\qquad\qquad\qquad\qquad\cdot\1\big\{ v \in B_2\big\}\dv\dz\ds\\
&=\int_{B_1}\int_{B_2}\int_\R \frac{1}{2v_{s,z}^{-1}(v)-s+z} f\big(s-v_{s,z}^{-1}(v)-z\big) f(v_{s,z}^{-1}(v)) f(z)\\*
&\qquad\qquad\qquad\cdot\1\big\{ v_{s,z}\big((s-z)/2\big) \le v \big\}\dz\dv\ds.
\end{align*}
Hence the random variable $(Y_1+Y_2+Y_3)$ has the density
\[
\int_\R \frac{1}{2v_{s,z}^{-1}(v)-s+z} f\big(s-v_{s,z}^{-1}(v)-z\big) f(v_{s,z}^{-1}(v)) f(z) \1\big\{ v_{s,z}\big((s-z)/2\big) \le v \big\}\dz
\]
for $s,v\in\R$.
Note that
\[
v_{s,z}((s-z)/2) \le v \quad \Leftrightarrow \quad 6v-2s^2-(3z-s)^2 \ge 0 \quad \Rightarrow \quad 6v-2s^2 \ge 0.
\]
We examine the integral
\begin{align*}
\lefteqn{\int_\R \frac{1}{2v_{s,z}^{-1}(v)-s+z} \1\big\{ v_{s,z}((s-z)/2) \le v \big\}\dz}\quad\\*
&=\int_\R \frac{\sqrt{3}}{\sqrt{6v-2s^2-(3z-s)^2}} \1\big\{ \tfrac12 \big( s^2+3z^2-2sz\big) \le v \big\}\dz\\
&=\int_\R \frac{\sqrt{3}}{\sqrt{6v-2s^2-(3z-s)^2}} \1\big\{ \tfrac13 \big(s-\sqrt{6v-2s^2} \big) \le z \le \tfrac13 \big(s+\sqrt{6v-2s^2}\big) \big\}\dz\\
&=\int_{\frac13 (s-\sqrt{6v-2s^2})}^{\frac13 (s+\sqrt{6v-2s^2})} \frac{\sqrt{3}}{\sqrt{6v-2s^2-(3z-s)^2}}\dz\\
&=\frac{1}{\sqrt{3}}\int_{-\sqrt{6v-2s^2}}^{\sqrt{6v-2s^2}} \frac{1}{\sqrt{6v-2s^2-w^2}}\dw=\frac{\pi}{\sqrt{3}}
\end{align*}
by \cref{l.int-a-w^2}.
Therefore for all bounded $f$, $(\sum X_j,\sum X_j^2)$ has a bounded density for $n=3$ (or equivalently, for all $n\ge3$). Thus by \cite[Theorem 19.1]{BR76normal}, \cref{c.density-cf} is fulfilled.
\end{proof}

\begin{proof}[Proof of \cref{l.lemma4'}]
	By the triangle inequality and \cref{p.lemma4}, we get for ${k=0,\dots,m}$
	\begin{align*}
	\lefteqn{\Big|\ddtk \Big(\tphi_{T_n'}(t)-\tphi_{\Phi^{\tP}_{m,n}}(t)\Big)\Big|}\quad\\
	&\le \Big|\ddtk \Big(\tphi_{T_n'}(t)-\tphi_{T_n}(t)\Big)\Big|+ c(m)\tL_{m+1,n}e^{-t^2/6}\big(|t|^{m+1-k}+|t|^{3m-1+k}\big),
	\end{align*}
	where we expand the first summand by the Leibniz rule as
	\begin{align*}
	\Big|\ddtk \Big(\tphi_{T_n'}(t)-\tphi_{T_n}(t)\Big)\Big|
	&\le \sum_{l=0}^{k-1} \binom{k}{l}\Big|\ddtl\tphi_{T_n}(t)\Big|\Big|\ddT{k-l}\exp(-\tfrac12 \beta_n t^2)\Big|\\
	&\quad+\Big|\ddtk\tphi_{T_n}(t)\Big|\Big(1-\exp(-\tfrac12 \beta_n t^2)\Big).
	\end{align*}
	
	By \cref{p.lemma4}, \labelcref{eq.ddtl-tphi} and \labelcref{eq.Lle1},
	\begin{align*}
	\Big|\ddtl\tphi_{T_n}(t)\Big|
	&\le \Big|\ddtl\tphi_{\Phi^{\tP}_{m,n}}(t)\Big|+c(m)\tL_{m+1,n}e^{-t^2/6}\big(|t|^{m+1-l}+|t|^{3m-1+l}\big)\\
	&\le c(m)e^{-t^2/2}\big(1+|t|^{2m-4+l}\big)+c(m)e^{-t^2/6}\big(|t|^{m+1-l}+|t|^{3m-1+l}\big)\\
	&\le c(m)e^{-t^2/6}\big(1+|t|^{3m-1+l}\big)
	\end{align*}
	for $l=0,\dots,m$. Additionally for $k>l$,
	\[\Big|\ddT{k-l}\exp(-\tfrac12 \beta_n t^2)\Big|
	\le c(k-l) \beta_n (1+|t|^{k-l}) \exp(-\tfrac12 \beta_n t^2)
	\]
	and
	\[\Big(1-\exp(-\tfrac12 \beta_n t^2)\Big)\le\tfrac12 \beta_n t^2.\]
	Putting everything together yields
	\begin{align*}
	\lefteqn{\Big|\ddtk \Big(\tphi_{T_n'}(t)-\tphi_{T_n}(t)\Big)\Big|}\quad\\*
	&\le \sum_{l=0}^{k-1} \binom{k}{l}\Big|\ddtl\tphi_{T_n}(t)\Big|\Big|\ddT{k-l}\exp(-\tfrac12 \beta_n t^2)\Big|\\
	&\quad+\Big|\ddtk\tphi_{T_n}(t)\Big|\Big(1-\exp(-\tfrac12 \beta_n t^2)\Big)\\
	&\le \sum_{l=0}^{k-1} \binom{k}{l}\Big(c(m)e^{-t^2/6}\big(1+|t|^{3m-1+l}\big)\Big)\Big(c(k-l) \beta_n (1+|t|^{k-l}) \exp(-\tfrac12 \beta_n t^2)\Big)\\
	&\quad+\Big(c(m)e^{-t^2/6}\big(1+|t|^{3m-1+k}\big)\Big)\Big(\tfrac12 \beta_n t^2\Big)\\
	&\le c(m) \beta_n e^{-t^2/6} \big(1+|t|^{3m+1+k}\big),
	\end{align*}
	which proves \labelcref{eq.lemma4'}.
\end{proof}

\subsection{Further remarks}\label{app.proofs.further}

\begin{remark}\label{r.non-deg}
	Let $\Sigma$ be the covariance matrix of $(X_1,X_1^2)$ and assume $\mu_4<\infty$. Then the following equivalences hold:\\
	$(X_1,X_1^2)$ is non-degenerate\\
	$\Leftrightarrow$ $\Var\big(a_1 X_1+a_2 X_1^2)>0$ for all $a=(a_1,a_2)\in\R^2$ with $a\ne0$\\
	$\Leftrightarrow$ $a^T \Sigma a >0$ for all $a\in\R^2$ with $a\ne0$\\
	$\Leftrightarrow$ $\Sigma$ is positive definite.
\end{remark}

\begin{remark}\label{r.E-T_n^2}
In this remark, we do neither assume that $X_1$ is symmetric, nor that $\mu_2=1$. Under these conditions, we investigate if $T_n$ is normalized, expanding \cref{r.E-T_n-normalized}. For clear illustration however, let all moments be finite in the following procedure. By \labelcref{eq.V_n^-2} we expand
\begin{equation}\label{eq.Tn-exp}
\begin{split}
T_n^2
&=n^{-1} \mu_2^{-1} S_n^2
- n^{-2} \mu_2^{-2} S_n^2 \Big(\sum (X_j^2-\mu_2)\Big)\\
&\quad+ n^{-3} \mu_2^{-3} S_n^2 \Big(\sum (X_j^2-\mu_2)\Big)^2
+ \cO_p\big(n^{-2}\big).
\end{split}
\end{equation}
Due to $\mu_1=0$ and $\E X_j^2=\mu_2$, the index of every factor within
\[
\Big(\sum X_j\Big)^{2}\Big(\sum (X_j^2-\mu_2)\Big)^{k}
=\sum_{j_1,\dots,j_{k+2}=1}^n X_{j_1} X_{j_2} \big(X_{j_3}^2-\mu_2\big) \cdots \big( X_{j_{k+2}}^2-\mu_2 \big)
\]
has to be equal to the index of another factor or the summand vanishes within the expectation. We take the expectation of each summand of \labelcref{eq.Tn-exp} separately
\begin{align*}
\E S_n^2
&= n \mu_2,\\
\E \Big[ S_n^2 \Big(\sum (X_j^2-\mu_2)\Big) \Big]
&= n (\mu_4-\mu_2^2),\\
\E \Big[ S_n^2 \Big(\sum (X_j^2-\mu_2)\Big)^2 \Big]
&= n^2 \big(\mu_2\mu_4-\mu_2^3+2\mu_3^2\big)
+ \cO(n)
\end{align*}
which leads to
\begin{align*}
\E T_n^2
&= 1 + n^{-1} 2 \mu_2^{-3} \mu_3^2 + \cO\big(n^{-2}\big).
\end{align*}
\end{remark}

\subsection{Auxiliary results}\label{app.lemmas}

\begin{lemma}\label{l.int-t-exp}
	Let $\nu_n\ge1$, $k=0,1,\dots$ and $0<\beta<4$. Then
	\begin{equation*}
	\int_{\nu_n}^\infty t^k \exp\big(- \beta t^2 /2 \big)\dt
	\le c(k)\beta^{-(k+2)/2} \exp\big(- \beta \nu_n^2 /4 \big).
	\end{equation*}
\end{lemma}
\begin{proof}
	We examine
	\begin{align*}
	\int_{\nu_n}^\infty t^k \exp(- \beta t^2 /2)\dt
	&\le \int_{\nu_n}^\infty t \, (t^2)^{\lceil\frac{k-1}{2}\rceil} \exp(- \beta t^2 /2)\dt\\
	&= \frac{\lceil\frac{k-1}{2}\rceil!}{(\beta/4 )^{\lceil\frac{k-1}{2}\rceil}} \int_{\nu_n}^\infty t \, \frac{(\beta t^2 /4)^{\lceil\frac{k-1}{2}\rceil}}{\lceil\frac{k-1}{2}\rceil!} \exp(- \beta t^2 /2)\dt\\
	&\le \frac{k!}{(\beta/4 )^{k/2}} \int_{\nu_n}^\infty t \Big(\sum_{l=0}^\infty \frac{(\beta t^2 /4)^{l}}{l!} \Big) \exp(- \beta t^2 /2)\dt\\
	&= \frac{k!\,2^k}{\beta^{k/2}} \int_{\nu_n}^\infty t \, \exp(\beta t^2 /4) \exp(- \beta t^2 /2)\dt\\
	&= \frac{k!\,2^{k+1}}{\beta^{(k+2)/2}} \exp( \beta \nu_n^2 /4).
	\end{align*}
\end{proof}

\begin{lemma}\label{l.ana1}
	Let $f:\R\to\R$ be a measurable function with $0\le f(x)\le1$ for all $x\in\R$ and let $\mu$ be a probability measure on $\R$. Assume that $\mu(\{x:f(x)<1\})>0$. Then
	\[
	\int_\R f(x)\dmu < 1.
	\]
\end{lemma}
\begin{proof}
	For $n\in\N$ we define $A_n:=\{x:f(x)<1-1/n\}$, which are increasing for ${n\to \infty}$. By continuity from below,
	\[
	\lim_{n\to\infty}\mu(A_n)=\mu\big(\{x:f(x)<1\}\big)>0.
	\]
	Thus, there exists an $N\in\N$ such that $\mu(A_N)>0$ and therefore,
	\begin{align*}
	\int_\R f(x)\dmu
	=\int_{A_N} f(x)\dmu + \int_{A_N^c} f(x)\dmu
	<(1-\tfrac{1}{N}) \mu(A_N) + \mu(A_N^c)
	=1-\tfrac{1}{N} \mu(A_N)<1.
	\end{align*}
\end{proof}

\begin{lemma}\label{l.int-a/2}
Let $n\ge2$. Then
\begin{align*}
\int_{\frac{1}{2}\sqrt{n}}^{2\sqrt{n}} \big(1+(z-n^{1/2})^2\big)^{-a/2} \dz
\le \begin{cases}
c(0) \sqrt{n} , &\text{ if } a=0,\\
c(0) \log n , &\text{ if } a=1,\\
c(0) , &\text{ if } a=2.
\end{cases}
\end{align*}
\end{lemma}
\begin{proof}
For $a=0$, the statement is trivial. Otherwise,
\begin{align*}
\int_{\frac{1}{2}\sqrt{n}}^{2\sqrt{n}} \big(1+(z-n^{1/2})^2\big)^{-a/2} \dz
=\int_{-\frac{1}{2}\sqrt{n}}^{\sqrt{n}} \big(1+z^2\big)^{-a/2} \dz
\le2\int_{0}^{\sqrt{n}} \big(1+z^2\big)^{-a/2} \dz.
\end{align*}
Now for $a=1$ and the inverse hyperbolic sine $\arsinh$,
\begin{align*}
\int_{0}^{\sqrt{n}} \big(1+z^2\big)^{-1/2} \dz
= \arsinh(\sqrt{n})
= \log(\sqrt{n} + \sqrt{1 + n})
\le c(0) \log(n),
\end{align*}
and for $a=2$
\begin{align*}
\int_{0}^{\sqrt{n}} \big(1+z^2\big)^{-1} \dz
=\arctan(\sqrt{n})
\le \tfrac \pi 2.
\end{align*}
\end{proof}

\begin{lemma}\label{Taylor_L}
For $L(x)=x\log x$ with $x>0$, 
	\begin{equation*}
	L(1+u+v)=L(1+u)+v+\theta_1 u v+\theta_2 v^2
	\end{equation*}
	for $|u|\le 1/4, |v|\le1/4$ and $|\theta_j|\le2$ depending on $u$ and $v$.
\end{lemma}

\begin{proof}	
	By Taylor expansion around $v=0$, 
	\begin{align*}
	(1+u)\log(1+u+v)
	&=(1+u)\Big(\log(1+u+0) + v (1+u+0)^{-1} \\*
	&\qquad\qquad+ \tfrac{v^2}{2} (-1) (1+u+\xi_1v)^{-2} \Big)\\
	&=(1+u)\log(1+u) + v + v^2 (1+u) (-\tfrac12) (1+u+\xi_1v)^{-2} 
	\end{align*}
	with $\xi_1\in[0,1]$ depending on $u$ and $v$ and by Taylor expansion around $u+v=0$, 
	\begin{align*}
	v\log(1+u+v)
	&=v\Big(\log(1+0) + (u+v) (1+\xi_2(u+v))^{-1} \Big)\\*
	&= (uv+v^2) (1+\xi_2(u+v))^{-1}
	\end{align*}
	with $\xi_2\in[0,1]$ depending on $u$ and $v$.
	Here,
	\begin{align*}
	(1+\xi_2(u+v))^{-1}
	&\le (1-1/2)^{-1}
	=2,\ \ \ 
	(1+\xi_2(u+v))^{-1}
	\ge (1+1/2)^{-1}
	= 2/3
	\end{align*}
	and
	\begin{align*}
	(1+u) \tfrac{1}{2} (1+u+\xi_1v)^{-2}
	&\le (1+u) \tfrac{1}{2} (1+u-1/4)^{-2}
	\le (3/4) \tfrac{1}{2} (1/2)^{-2}
	= 3/2,\\*
	(1+u) \tfrac{1}{2} (1+u+\xi_1v)^{-2}
	&\ge (1+u) \tfrac{1}{2} (1+u+1/4)^{-2}
	\ge (5/4) \tfrac{1}{2} (3/2)^{-2}
	= 5/18.
	\end{align*}
	Consequently,
	$
	(1+\xi_2(u+v))^{-1}-(1+u) \tfrac{1}{2} (1+u+\xi_1v)^{-2}
	\in [-5/6,31/18]$.

\end{proof}

\begin{lemma}\label{l.int-ax^2+b}
	Let $a,b,l>0$. Then
	\begin{align*}
	\int_{-l}^l (ax^2+b)^{-3/2}\dx = 2\big(al^2+b\big)^{-1/2}\,b^{-1}\,l.
	\end{align*}
\end{lemma}
\begin{proof}
	\begin{align*}
	\int_{-l}^l (ax^2+b)^{-3/2}\dx
	&=2b^{-3/2} \int_{0}^l (\tfrac ab x^2+1)^{-3/2}\dx\\
	&=2b^{-3/2}\sqrt{b/a} \int_{0}^{l\sqrt{a/b}} (x^2+1)^{-3/2}\dx\\
	&=2b^{-3/2}\sqrt{b/a} \, \frac{l\sqrt{a/b}}{\sqrt{l^2a/b+1}} = 2\big(al^2+b\big)^{-1/2}\,b^{-1}\,l.
	\end{align*}
\end{proof}

\begin{lemma}\label{l.int-a-w^2}
	Let $a > 0$. Then
	\begin{align*}
	\int_{-\sqrt{a}}^{\sqrt{a}} (a-w^2)^{-1/2} \dw
	= \pi.
	\end{align*}
\end{lemma}
\begin{proof}
	\begin{align*}
	\int_{-\sqrt{a}}^{\sqrt{a}} (a-w^2)^{-1/2}\dw
	&=2a^{-1/2} \int_{0}^{\sqrt{a}} (1-w^2/a)^{-1/2}\dw\\
	&=2a^{-1/2} a^{1/2} \int_{0}^{1} (1-w^2)^{-1/2}\dw\\
	&= 2 \arcsin(1)=\pi.
	\end{align*}
\end{proof}

\end{appendix}

\begin{acks}[Acknowledgments]
This work was supported by the DFG research units 1735 and 5381.
\end{acks}

\bibliographystyle{imsart-number} 
\bibliography{bibliography}  

\end{document}